\documentclass[10pt]{article}

\usepackage[letterpaper, top=24mm, bottom=24mm, left=24mm, right=24mm, includefoot]{geometry}

\usepackage{latexsym,
amsthm,
color,
amssymb,
url,
bm,
cite,
amssymb,
microtype,
amsmath,
amsfonts,
mathtools}
\usepackage{verbatim,listings}
\usepackage{todonotes}
\usetikzlibrary{calc}
\usepackage{tikz}
\tikzstyle{uStyle}=[shape = circle, minimum size = 2pt, inner sep =2.5pt, outer sep = 0pt, draw, fill=white]
\usetikzlibrary{decorations.pathmorphing}
\usetikzlibrary{decorations.markings, calligraphy}
\usetikzlibrary{bbox}

\usepackage{xcolor}
\definecolor{mypurple}{RGB}{208,134,255}
\definecolor{myblue}{RGB}{10,120,253}
\definecolor{myredred}{RGB}{163,0,0}
\colorlet{myred}{myredred!70}
\definecolor{mygreen}{RGB}{126,198,52}
\definecolor{myothergreen}{RGB}{0,153,0}
\definecolor{myorange}{RGB}{244,154,33}
\definecolor{myyellow}{RGB}{225,225,2}
\definecolor{mydarkgreen}{RGB}{51,102,0}
\definecolor{mydarkpurple}{RGB}{127,0,225}

\usepackage{caption}
\usepackage{setspace}
\usepackage{subcaption}
\usepackage{setspace}

 \usepackage[pdftex, plainpages = false, pdfpagelabels, 
                 bookmarks = true,
                 bookmarksopen = true,
                 bookmarksnumbered = true,
                 breaklinks = true,
                 linktocpage,
                 pagebackref,
                 colorlinks = true,  
                 linkcolor = blue,
                 urlcolor  = blue,
                 citecolor = red,
                 anchorcolor = green,
                 hyperindex = true,
                 hyperfigures
                 ]{hyperref} 
                     
\usepackage[inline]{enumitem}
\usepackage[mathscr]{euscript}
\usepackage{amsthm}
\usepackage{nicefrac}
\usepackage{dsfont}
\usepackage{tikz}
\usepackage[ruled,linesnumbered]{algorithm2e}
\usepackage{nicefrac}
\usepackage[mathscr]{euscript}
\graphicspath{{./figures/}}

\RequirePackage[T1]{fontenc}
\RequirePackage[utf8]{inputenc}
\usepackage[english]{babel}
\usepackage{alphabeta}

\usetikzlibrary{math}

\makeatletter
\input{colordvi}

\usepackage{aliascnt}

\definecolor{MidnightBlack}{rgb}{0.1,0.1,.34}
\definecolor{MidnightBlue}{rgb}{0.1,0.1,0.43}
\definecolor{Black}{rgb}{0,0, 0}
\definecolor{Blue}{rgb}{0, 0 ,1}
\definecolor{Red}{rgb}{1, 0 ,0}
\definecolor{White}{rgb}{1, 1, 1}
\definecolor{grey}{rgb}{.6, .6, .6}
\definecolor{Mygreen}{rgb}{.0, .7, .0}
\definecolor{Yellow}{rgb}{.55,.55,0}
\definecolor{Mustard}{rgb}{1.0, 0.86, 0.33}
\definecolor{applegreen}{rgb}{0.55, 0.71, 0.0}
\definecolor{darkturquoise}{rgb}{0.0, 0.81, 0.82}
\definecolor{celestialblue}{rgb}{0.29, 0.59, 0.82}
\definecolor{green_yellow}{rgb}{0.68, 1.0, 0.18}
\definecolor{crimsonglory}{rgb}{0.75, 0.0, 0.2}
\definecolor{darkmagenta}{rgb}{0.30, 0.0, 0.30}
\definecolor{magenta}{rgb}{0.50, 0.0, 0.50}
\definecolor{internationalorange}{rgb}{1.0, 0.31, 0.0}
\definecolor{darkorange}{rgb}{1.0, 0.55, 0.0}
\definecolor{ao}{rgb}{0.0, 0.5, 0.0}
\definecolor{awesome}{rgb}{1.0, 0.13, 0.30}
\definecolor{darkcyan}{rgb}{0.0, 0.50, 0.50}
\definecolor{violet}{rgb}{0.93, 0.51, 0.93}
\definecolor{brown}{rgb}{0.65, 0.16, 0.16}
\definecolor{orange}{rgb}{1.0, 0.65, 0.0}
\definecolor{cornflowerblue}{rgb}{0.39, 0.58, 0.93}

\newcommand{\cref}[1]{\autoref{#1}}

\newcommand{\revcommentref}[1]{\hyperref[#1]{\textbf{Comment:\ }}}

\newcommand{\remove}[1]{}

\newcommand{\Acal}{\mathcal{A}}

\newcommand{\Ccal}{\mathcal{C}}
\newcommand{\Dcal}{\mathcal{D}}
\newcommand{\Ecal}{\mathcal{E}}

\newcommand{\Gcal}{\mathcal{G}}
\newcommand{\Hcal}{\mathcal{H}}

\newcommand{\Mcal}{\mathcal{M}}

\newcommand{\Ocal}{\mathcal{O}}

\newcommand{\Qcal}{\mathcal{Q}}

\newcommand{\Tcal}{\mathcal{T}}
\newcommand{\Ucal}{\mathcal{U}}

\newcommand{\Xcal}{\mathcal{X}}
\newcommand{\Ycal}{\mathcal{Y}}
\newcommand{\Zcal}{\mathcal{Z}}
\newcommand{\Abbb}{\mathbb{A}}
\newcommand{\Bbbb}{\mathbb{B}}
\newcommand{\Cbbb}{\mathbb{C}}

\newcommand{\Lbbb}{\mathbb{L}}

\newcommand{\Nbbb}{\mathbb{N}}

\newcommand{\Qbbb}{\mathbb{Q}}

\newcommand{\Ybbb}{\mathbb{Y}}

\RequirePackage{stmaryrd}
\usepackage{textcomp}
\DeclareUnicodeCharacter{2286}{\subseteq}
\DeclareUnicodeCharacter{2192}{\ifmmode\to\else\textrightarrow\fi}
\DeclareUnicodeCharacter{2203}{\ensuremath\exists}
\DeclareUnicodeCharacter{183}{\cdot}
\DeclareUnicodeCharacter{2200}{\forall}
\DeclareUnicodeCharacter{2264}{\leq}
\DeclareUnicodeCharacter{2265}{\geq}
\DeclareUnicodeCharacter{8614}{\mathbin{\mapsto}}
\DeclareUnicodeCharacter{8656}{\Leftarrow}
\DeclareUnicodeCharacter{8657}{\Uparrow}
\DeclareUnicodeCharacter{8658}{\Rightarrow}
\DeclareUnicodeCharacter{8659}{\Downarrow}
\DeclareUnicodeCharacter{8669}{\rightsquigarrow}
\newcommand{\eqdef}{\stackrel{{\scriptsize\rm def}}{=}}
\DeclareUnicodeCharacter{8797}{\eqdef}
\DeclareUnicodeCharacter{8870}{\vdash}
\DeclareUnicodeCharacter{8873}{\Vdash}
\DeclareUnicodeCharacter{22A7}{\minor models}
\DeclareUnicodeCharacter{9121}{\lceil}
\DeclareUnicodeCharacter{9123}{\lfloor}
\DeclareUnicodeCharacter{9124}{\rceil}
\DeclareUnicodeCharacter{2208}{\in}
\DeclareUnicodeCharacter{9126}{\rfloor}
\DeclareUnicodeCharacter{9655}{\triangleright}
\DeclareUnicodeCharacter{9665}{\triangleleft}
\DeclareUnicodeCharacter{9671}{\diamond}
\DeclareUnicodeCharacter{9675}{\circ}
\DeclareUnicodeCharacter{10178}{\bot}
\DeclareUnicodeCharacter{10214}{} % needs stmaryrd
\DeclareUnicodeCharacter{10215}{} % needs stmaryrd
\DeclareUnicodeCharacter{10229}{\longleftarrow}
\DeclareUnicodeCharacter{10230}{\longrightarrow}
\DeclareUnicodeCharacter{10231}{\longleftrightarrow}
\DeclareUnicodeCharacter{10230}{\Longleftarrow}
\DeclareUnicodeCharacter{10233}{\Longrightarrow}
\DeclareUnicodeCharacter{10234}{\Longleftrightarrow}
\DeclareUnicodeCharacter{10236}{\longmapsto}
\DeclareUnicodeCharacter{10238}{\Longmapsto} % needs stmaryrd
\DeclareUnicodeCharacter{10503}{\Mapsto}    % needs stmaryrd
\DeclareUnicodeCharacter{10971}{\mathrel{\not\hspace{-0.2em}\cap}}
\DeclareUnicodeCharacter{65294}{\ldotp}
\DeclareUnicodeCharacter{65372}{\mid}

\newtheorem{theorem}{Theorem}[section]

\newaliascnt{question}{theorem}

\aliascntresetthe{question}

\newaliascnt{lemma}{theorem}
\newtheorem{lemma}[lemma]{Lemma}
\aliascntresetthe{lemma}

\newaliascnt{remark}{theorem}

\aliascntresetthe{remark}

\newaliascnt{claim}{theorem}
\newtheorem{claim}[claim]{Claim}
\aliascntresetthe{claim}

\newaliascnt{invariant}{theorem}

\aliascntresetthe{invariant}

\newaliascnt{proposition}{theorem}
\newtheorem{proposition}[proposition]{Proposition}
\aliascntresetthe{proposition}

\newaliascnt{observation}{theorem}
\newtheorem{observation}[observation]{Observation}
\aliascntresetthe{observation}

\newaliascnt{corollary}{theorem}
\newtheorem{corollary}[corollary]{Corollary}
\aliascntresetthe{corollary}

\newaliascnt{definition}{theorem}

\aliascntresetthe{definition}

\newaliascnt{conjecture}{theorem}
\newtheorem{conjecture}[conjecture]{Conjecture}
\aliascntresetthe{conjecture}

\newaliascnt{counterexample}{theorem}

\aliascntresetthe{counterexample}

\newcommand{\hh}{\end{document}}

\newcommand{\obs}{{\sf obs}}
\newcommand{\excl}{{\sf excl}}
\newcommand{\minors}{{\sf minors}}
\newcommand{\ex}{\textsf{ex}}
\newcommand{\cobs}{\mbox{\rm \textsf{cobs}}}
\newcommand{\gall}{\mathcal{G}_{{\text{\rm  \textsf{all}}}}}
 \newcommand{\lin}[1]{\langle #1\rangle}% linear society

\newcommand{\tw}{\mathsf{tw}\xspace}%
\newcommand{\poly}{\mathsf{poly}\xspace}%

\newenvironment{claimproof}[1][\proofname]{%
 \begin{proof}[#1]%
}{%
 \end{proof}%
}

\begin{document}

\title{Obstructions for Minor-Closed Classes of \\ limiting Densities Below $\nicefrac{3}{2}$\thanks{Emails of authors: 
\texttt{akominato@gmail.com},  \texttt{rm7230@nyu.edu}, and   \texttt{sedthilk@thilikos.info}.}}

\author{Antonios Kominatos\thanks{Department of Mathematics, National and Kapodistrian University of Athens, Athens, Greece.}~$^{,}$\thanks{Inter-Institutional Graduate Program ``Algorithms, Logic, and Discrete Mathematics'' (ALMA), Athens, Greece.}~$^{,}$\thanks{Supported by the Scientific Research Network ``Graphs, Association schemes and Geometries: structures, algorithms and computation'' (grant no. W003324N) of the Research Foundation - Flanders (Belgium).}\and Reem Mahmoud\thanks{NYUAD, CS-program, Abu Dhabi, UAE.}  \and Dimitrios M. Thilikos\thanks{LIRMM, Univ Montpellier, CNRS, Montpellier, France.}~$^{,}$\thanks{Supported by the Franco-Norwegian project PHC AURORA 2024-25 and the French National Research Agency (ANR)  under project GODASse ANR-24-CE48-4377 and under the France 2030 grant reference number
ANR-24-RRII-0002 operated by the Inria Quadrant Program.}}

\date{\today}

\maketitle

\begin{abstract}
\noindent Given a graph class $\mathcal{G}$, the limiting density of $\mathcal{G}$ is defined as $\delta(\mathcal{G})=\lim_{n\to\infty} {\sf ex}(\mathcal{G},n)/n$ where  ${\sf ex}(\mathcal{G},n)$ is the maximum number of edges of a graph in $\mathcal{G}$ on $n$ vertices. The limiting density $\delta(\mathcal{G})$ is known to be a rational number  when $\mathcal{G}$ is a minor-closed graph class. 
For every $\delta\in[0,\nicefrac{3}{2})$, we prove that 
the set of $\subseteq$-minimal minor-closed graph  classes with densities $>\delta$ is \textsl{finite} and we identify it completely.
A consequence of our results is an algorithm that,
given a finite set of graphs $\mathcal{Z}$, of total size $n$, either outputs  
the value of $\delta(\excl(\mathcal{Z}))$ or reports that  $\delta(\excl(\mathcal{Z}))\geq \frac{3}{2}$, where 
$\excl(\mathcal{Z})$ is the class of graphs excluding  the graphs in $\mathcal{Z}$ as minors. The algorithm runs in $2^{\poly(n)}$ time.
\end{abstract}


\newpage 

\tableofcontents
\newpage

\section{Introduction}

All graphs in this paper are finite and have no loops or multiple edges.
Given a graph $G$, we use $V(G)$ and $E(G)$ for the set of its vertices and edges, respectively. 
Given a graph class $\Gcal$ and a non-negative integer $n$ 
we define $\ex(\Gcal,n):=\max\{|E(G)|\mid G\in \Gcal, |V(G)|=n\}$, i.e., 
 
the maximum number of edges of an $n$-vertex graph in $\Gcal$.

\subsection{Limiting densities}

The  \emph{limiting density} of  an infinite graph class  $\Gcal$
is defined as  $$\delta(\Gcal)=\lim_{n\to\infty} \frac{{\sf ex}(\Gcal,n)}{n}.$$
If $\Gcal$ is a finite class then the limiting density of 
$\Gcal$ is $- \infty$.
In this paper we  
 
study limiting densities of minor-closed graph classes (see \cref{psredlds} for definitions on the minor relation).
A minor-closed 
graph class $\Gcal$ is \emph{proper} if it does not 
contain all graphs.
Given that $\Gcal$ is proper, there is some graph $H$ that does not belong in $\Gcal$, therefore 
$\Gcal$ excludes a clique on $k=|V(H)|$ vertices as a minor. It was proved by Thomasson in \cite{Thomason01thee} that if a 
graph $G$ excludes $K_{k}$ as a minor then 
  $|E(G)|/|V(G)|\leq (\alpha+o(1))\cdot k\cdot\sqrt{ \log k}$, where $\alpha\approx 0.319$, therefore $\delta(\Gcal)$ is always a non-negative real number (see also \cite{Mader67Homomorphieeigenschaften,Kostochka84,Thomason84Anextrema,ThomasonW22Onthe,Thomason06Extremal}). 
  
  We denote by $\Lbbb$ the set of all limiting densities of infinite proper
minor-closed graph classes.  

The study of  the set $\Lbbb$ was initiated by David Eppstein in \cite{Eppstein10Densities}, where he proved that 
\begin{eqnarray}
\Lbbb\cap [0,\nicefrac{3}{2}) & = & \{\frac{x}{x+1}\mid x\in \Nbbb\}\cup\{\frac{3x+2y}{2x+y+1}\mid x\in\Nbbb_{\geq 1}, y\in\{0,1,2\}\}.    
\label{eq_p_kil}
\end{eqnarray}

The above implies that $[0,\nicefrac{3}{2})$ contains two accumulation points, namely $1$ and $\frac{3}{2}$.
We denote by $\Bbbb$ the set of all accumulation points of $\Lbbb$.
Interestingly, Eppstein observed in \cite{Eppstein10Densities} that if $a\in \Lbbb$, then 
$a+1\in\Bbbb\cap\Lbbb$.

As a next step, McDiarmid and  Przykucki  in \cite{McDiarmidP19On}
completely identified $\Lbbb\cap {[0,2)}$ by proving that 
for every $k\in\Nbbb_{\geq 2}$  it holds that 
\vspace{-8mm}

\begin{eqnarray}
\begin{minipage}{13cm}
\begin{center}
\begin{eqnarray*}
\Lbbb\cap {\big[2-\frac{1}{k-1},2-\frac{1}{k}\big)} &=& \{2-\frac{1}{k-1}\}\cup\{2-\frac{1}{k}-\frac{2k-t-1}{kn} \}
\end{eqnarray*}
\end{center}
\end{minipage}&&\label{mosekrer}
\end{eqnarray}
{where} $n =mk+1+t, m  \geq   1, 0\leq t\leq   k-1, \text{~and~}n  >  (2k-1-t)(k-1).$
According to the above, $\Lbbb\cap {[0,2)}$ contains
infinitely many accumulation points, which are the points in $\Bbbb\coloneqq\{2-\frac{1}{k}\mid k\in \Nbbb_{\geq 1}\}$. 
Recently, Kapadia and Norin proved in \cite{KapadiaN20Densities} that   
all limiting densities are rational numbers, therefore $\Lbbb\subseteq\Qbbb$.
Moreover, they showed that 
 for each proper minor-closed class $\Gcal$ there exists a minor-closed class $\Gcal' \subseteq \Gcal$ of bounded pathwidth\footnote{We avoid the definition of pathwidth here. We just mention that 
a minor-closed graph class has bounded pathwidth iff it does not contain all acyclic graphs \cite{BienstockRST91Quickly,RobertsonS83GMI}.} such that $\delta(\Gcal') = \delta(\Gcal)$.

\subsection{Deciding limiting densities} The question that motivated this paper is the following:
\begin{eqnarray}
\begin{minipage}{14cm}
\textsl{Let $\Gcal$ be an infinite proper minor-closed graph class and $\delta$ be a  non-negative rational number. \\ Can we decide   whether $\delta(\Gcal)\leq \delta$ or whether $\delta(\Gcal)<\delta$?}
\end{minipage}
\label{cojhknssjtraint}
\end{eqnarray}

Notice that it is important to distinguish the questions $\delta(\Gcal)\leq \delta$ and  $\delta(\Gcal)<\delta$. The second question becomes particularly relevant when $\delta$ is an \textsl{accumulation point} of $\mathbb{L}$. This is because 
 \textsc{limiting Density $<\delta$} when $\delta$ is not an accumulation point 
is equivalent to \textsc{limiting Density $\leq \delta'$} where
$\delta'$ is the biggest limiting density   that is smaller than $\delta$.

In order to formulate the above questions algorithmically, we need a finite encoding of $\Gcal$.
A typical way to describe minor-closed graph classes is by  excluding as minors all graphs of some finite set $\Zcal$. The resulting minor-closed graph class is denoted by $\excl(\Zcal)$. Robertson and Seymour showed that every 
minor-closed graph class $\Gcal$  
has such a finite description, i.e, there is some finite $\Zcal$, called the \emph{minor-obstruction set} of $\Gcal$, such that $\Gcal=\excl(\Zcal)$ \cite{RobertsonS04GMXX}. 
We use $\obs(\Gcal)$
 for the minor-obstruction set of the class $\Gcal$ and observe that $\Gcal=\excl(\obs(\Gcal)$).
Moreover, we may assume that $\Zcal$ is a \emph{minor anti-chain}, that is, 
its elements are pairwise non-comparable under the minor relation.
For every non-negative rational number $\delta$, we may now formulate the question in \eqref{cojhknssjtraint} algorithmically as follows. 

\begin{eqnarray}
\fbox{
\begin{minipage}{11cm}
\textsc{limiting Density $\leq \delta$ ($<\delta$)}

\noindent\textsl{Input}: A minor anti-chain $\Zcal$ and a non-negative rational number  $\delta$.

\noindent\textsl{Question}: Is the limiting density of $\excl(\Zcal)$ at most (strictly less than) $\delta$?
\end{minipage}}\label{probl_reset}
\end{eqnarray}

As a measure of the input sizes of the above problems we define $n(\Zcal)  = \sum_{Z\in\Zcal}|V(Z)|$, i.e., the sum of the sizes of the graphs in $\Zcal$.

\subsection{Our results}
In this paper, we prove that, for every value 
$\delta\in[0,\nicefrac{3}{2})$,  
\textsc{limiting Density $\leq \delta$} can be solved in time linear 
on the size $n(\Zcal)$ of the input. We also prove the same for 
\textsc{limiting Density $<\delta$} when $\delta$ is one of the two accumulation points in $\frak{B}\cap [0,\nicefrac{3}{2}]$, i.e., $\delta=1$ and $\delta=\frac{3}{2}$. 
  Our algorithms come as a consequence 
of a series of  combinatorial results that we explain below.

\paragraph{Class properties on limiting densities.} Given some non-negative rational $\delta$, we define $\Cbbb_{\leq \delta}$ (resp. $\Cbbb_{<\delta}$) as the set of all minor-closed graph classes whose limiting density is at most (resp. strictly less than) $\delta$.

We define the \emph{class obstruction} of $\Cbbb_{\leq \delta}$, denoted $\cobs(\Cbbb_{\leq \delta})$,
as the set of all $\subseteq$-minimal minor-closed graph classes 
that do not belong in $\Cbbb_{\leq \delta}$.
We define $\cobs(\Cbbb_{<\delta})$ analogously. It is worth noting that we have no general 
theory implying that $\cobs(\Cbbb_{\leq \delta})$ (or $\cobs(\Cbbb_{<\delta})$) is finite.
This would certainly follow if  
 
minor-closed graph classes were well-quasi-ordered with respect to the set inclusion relation.
However, this is an open problem that is known as the ``$\omega^2$-wqo conjecture'' for the minor relation  \cite{paul2023graph,paul2023universal,Jancar99ANote,Marcone01Fine}.

{Proving that \(\cobs(\Cbbb_{\le \delta})\) (or \(\cobs(\Cbbb_{<\delta})\)) is finite has algorithmic consequences. Indeed, as we will see in \cref{alg_ons}, knowing the finite obstruction sets for the finitely many classes in a class obstruction yields algorithms for the problems in \eqref{probl_reset}, based on finitely many minor-checking queries. In turn, using the recent result of \cite{KorhonenPS24Minor}, this implies that these problems can be solved in almost linear time; see also \cite{RobertsonS95b,KawarabayashiKR12Thedisjoint}.}\medskip

\paragraph{Combinatorics and algorithms on $\cobs(\Cbbb_{\leq \delta})$ and $\cobs(\Cbbb_{<\delta})$.}
In this paper, we make a first step towards the study of   $\cobs(\Cbbb_{\leq \delta})$ (resp. $\cobs(\Cbbb_{<\delta})$) by focusing our attention on the cases where $\delta$ is smaller (resp. strictly smaller) than $\frac{3}{2}$.  
In \cref{xclsoosocnbsutuidison}, we  provide exact descriptions 
of  $\cobs(\Cbbb_{\leq \delta})$, for $\delta\in[0,\nicefrac{3}{2})$, and  of $\cobs(\Cbbb_{<1})$ and $\cobs(\Cbbb_{<\nicefrac{3}{2}})$. It follows that $\cobs(\Cbbb_{<1})$ consists of three graph classes,
while $\cobs(\Cbbb_{<\nicefrac{3}{2}})$  consists of 33 graph classes. According to our results,  the size of $\cobs(\Cbbb_{\leq \delta})$, i.e., the number of class obstructions, can be bounded as follows:

\begin{itemize}
\item If $\delta\in [0,1)$, then 
$|\cobs(\Cbbb_{\leq \delta})|$ is 2 plus the number of trees on $\lfloor\frac{\delta}{1- \delta}\rfloor+2$ vertices.
\item If $\delta\in [1,\nicefrac{3}{2})$, then 
$|\cobs(\Cbbb_{\leq \delta})|\leq 32+n_1+n_2+n_3+n_4$
where 
\begin{itemize}
    \item $n_1$ is the number of connected graphs with $1+\lfloor \frac{\delta}{3-2\delta}\rfloor$ triangle blocks,
    \item $n_2$ is the number of connected graphs with $\max\{1+\lfloor \frac{4\delta-5}{3-2\delta}\rfloor,0\}$ triangle blocks with one block being the unique $2$-tree on $4$ vertices,
    \item $n_3$ is the number of connected graphs with $\max\{1+\lfloor \frac{5\delta-7}{3-2\delta}\rfloor,0\}$ triangle blocks with one block being a $2$-tree on $5$ vertices, and 
    \item $n_4$ is the number of connected graphs with $\max\{1+\lfloor \frac{7\delta-10}{3-2\delta}\rfloor,0\}$ triangle blocks with two blocks being the unique $2$-tree on $4$ vertices. 
\end{itemize}
\end{itemize}

In \cref{sec_ident}, for every  $\delta\in [0,\nicefrac{3}{2})$ and for each $\Ccal\in \cobs(\Cbbb_{\leq \delta})$,  we completely determine $\obs(\Ccal)$. When $\delta\in [0,1)$,
none of these obstructions  has size  bigger than $\lfloor\frac{\delta}{1- \delta}\rfloor+3$ and when $\delta\in [1,\nicefrac{3}{2})$,
none of these obstructions  has size  bigger than $2(1+\lfloor \frac{\delta}{3-2\delta}\rfloor)+1$.

Given the complete knowledge of all class obstructions 
for every $\delta\in[0,\nicefrac{3}{2})$, as well as of their own obstructions, in \cref{alg_ons}, we capitalize on the fact that 
the presence/absence of these obstructions  serves as a certificate for the negative/positive answer to \textsc{limiting Density $\leq \delta$} and \textsc{limiting Density $<\delta$}. This leads to the claimed linear-time algorithms.
Finally, using these algorithms as subroutines, 
we construct an algorithm that, given a $\delta \in [0,\nicefrac{3}{2})$
and an anti-chain $\Zcal$, either outputs 
the limiting density of $\excl(\Zcal)$
or reports that it is at least $\frac{3}{2}$
in time $2^{\poly(n)}$, where $n=n(\Zcal)$, i.e.,  time that is exponential to the size of the input $\Zcal$.\footnote{We use $\poly(n)$ as a shortcut for $n^{\Ocal(1)}$.}

\section{Definitions and preliminary results}
\label{psredlds}

 In this section we give a series of concepts and preliminary results that will be used in the rest of the paper.

\subsection{Basic concepts on graphs}
We denote by $\mathbb{N}$ the set of non-negative integers.
 
We use the notation $\Nbbb_{\geq n}=\{x\in \Nbbb\mid x\geq n\}$.
 Given some $p,q\in\Nbbb_{\geq1}$ with $p<q$ we denote by $[p]$ and $[p,q]$ the sets $\{1, \dots, p\}$ and $\{p,p+1,\dots,q\}$, respectively.

We now provide some elementary concepts on graphs that we use in this paper. 
For other concepts on graphs that are not defined neither here nor in the introduction, we 
adopt the notation of \cite{Diestel2010Book}.\medskip

We denote by $\gall$ the set of all graphs. Every graph class $\Gcal$ that is different from $\gall$ is called \emph{proper}. {Given a graph class $\Gcal$, we denote by $\Gcal_n$ the set of $n$-vertex graphs in $\Gcal$.}
 Given some  $X\subseteq V(G),$ we define  the subgraph of $G$ \emph{induced} by $X$ and denote it by   $G[X]=(X,\{e\in E(G)\mid e\subseteq X\})$. The vertex set $X$ is \emph{connected}  (in $G$) if $G[X]$ is a connected graph.

We say that a graph $G$ \emph{contains} a graph $Z$ as a \emph{minor}
if there exist a collection  $\Xcal=\{X_v\mid v\in V(Z)\}$ of pairwise disjoint connected vertex sets  of $G$ such that for every edge $vu\in E(Z)$ the set $X_{v}\cup X_{u}$ is connected. 
 We call  $\Xcal=\{X_v\mid v\in V(Z)\}$   a \emph{minor model of $Z$} in $G$. We write $Z\leq G$ to denote  that $Z$ is a minor of $G$. 
 Given some $S \subseteq V(G)$, we call a minor model $\Xcal$ of a graph $Z$ in $G$ an \emph{$S$-minor model} if every vertex set of the minor model contains at least one vertex of $S$. 
 A graph $Z$ is an \emph{$S$-minor} of $G$ if there exists a minor model of $Z$ in $G$ that is an $S$-minor model.

Given a graph $G$ and an edge $e$, we denote by $G/e$ the graph obtained from 
$G$ after contracting the edge $e$. The \emph{contraction} of an edge $e=xy$ is the replacement of $x$ and $y$ by a vertex $v_{xy}$ 
whose neighbors are the vertices in  $(N_{G}(x)\cup N_{G}(y))\setminus\{x,y\}$. 
Notice that $H$ is a minor of $G$ 
if and only if a graph isomorphic to $H$ can be obtained from a subgraph of $G$ after contracting a (possibly empty) set  of edges.

A \emph{bridge} in a graph is a subgraph 
that is isomorphic to $K_{2}$ where the removal of its unique edge increases the number of connected components.
A \emph{block} of a graph $G$ is a maximal 2-connected subgraph of $G$ or a bridge of $G$.  
A block in $G$ is a \emph{leaf block} if it is a leaf of $T$, i.e., if exactly one of its vertices is a cut vertex.

\subsection{Obstructions}

A  \emph{minor anti-chain} $\Zcal$ is a set of graphs that are pairwise non-comparable by $\leq $. 
Given a class $\Gcal$ of graphs, we define ${\sf min}(\Gcal)$ as the subset of $\Gcal$ containing the minor-minimal elements of $\Gcal$.
Notice that  ${\sf min}(\Gcal)$   is a minor anti-chain that, 
according to the Graph Minors Theorem of Robertson and Seymour \cite{RobertsonS04GMXX}, is always finite.

\paragraph{Obstructions.}
Given a minor-closed graph class $\Gcal$, we define  its \emph{obstruction set} as
the minor-minimal graphs that do not belong in $\Gcal$, i.e., the set 
 $\obs(\Gcal)={\sf min}(\gall\setminus \Gcal)$.  We also refer to the members of $\obs(\Gcal)$ as the \emph{minor-obstructions} of $\Gcal$.
 Given some minor anti-chain $\Zcal,$ we denote by $\excl(\Zcal)$ the set of 
 all graphs excluding every graph in $\Zcal$ as a minor.
 Notice that, for every minor anti-chain $\Zcal$, it holds that $\obs(\excl(\Zcal))=\Zcal$ and 
 for every minor-closed graph class $\Gcal$, we have $\excl(\obs(\Gcal))=\Gcal$.

Notice that $\Zcal$ contains some edgeless graph if and only if $\excl(\Zcal)$ is finite. Recall also that finite classes have limiting density $-\infty$.
Therefore we have the following.

\begin{observation}
 
\label{obs_nulnul}
 Let $\Zcal$ be a finite set of graphs.  It holds that  $\delta(\excl(\Zcal))\geq 0$  if and only if none 
 of the graphs in $\Zcal$ is edgeless.  
\end{observation}

 A \emph{class property} 
is a family $\Cbbb$ of  graph classes. It is \emph{subset closed} if every subclass of a class 
in $\Cbbb$ also belongs in $\Cbbb$. Throughout this paper, we always assume the following:
\begin{itemize}
\item All graph classes that we consider are proper minor-closed and 
\item All class properties are subset-closed and contain only proper minor-closed graph classes.
\end{itemize}

\paragraph{Second order obstructions.}
Given two minor anti-chains $\Zcal_{1}$ and $\Zcal_{2},$ we say that 
$\Zcal_{1}$   is a \emph{Smyth minor} of $\Zcal_{2}$, denoted $\Zcal_{1}\leq_{*}\Zcal_{2}$, if for every $Z_2\in \Zcal_{2}$
there is some $Z_1\in\Zcal_{1}$ such that $Z_{1}\leq Z_{2}$. The relation $\leq _{*}$ is known as the \emph{Smyth extension} 
of $\leq $.  
It is easy to verify  that, given two  graph classes 
$\Gcal^{1}$ and $\Gcal^{2}$, 
\begin{eqnarray}
\label{eq_oksui}
\Gcal^{1}\subseteq \Gcal^{2}\iff \obs(\Gcal^{1})\leq_{*}\obs(\Gcal^{2}). \label{ppog99fdhye}
\end{eqnarray}

The \emph{class obstruction set} of a class property $\Cbbb$ is the set of all $\subseteq$-minimal minor-closed 
graph classes that do not belong in $\Cbbb$ and we denote it with ${\sf cobs}(\Cbbb)$.
The fact that ${\sf cobs}(\Cbbb)$ exists follows from Robertson \& Seymour's Graph Minors Theorem \cite{RobertsonS04GMXX}.
However, it is not known whether ${\sf cobs}(\Cbbb)$ is finite in general.

The  \emph{(second order) obstruction set} of $\Cbbb$ is defined as $${\sf OBS}(\Cbbb)=\{\obs(\Gcal)\mid \Gcal\in {\sf cobs}(\Cbbb)\}.$$
The above definitions directly imply that for every class property $\Cbbb$
and every minor anti-chain $\Zcal$, \ 
\begin{eqnarray}
\excl(\Zcal)\in\Cbbb \iff \forall {\Ycal\in{\sf OBS}(\Cbbb)},\Ycal\nleq_{*}\Zcal.\label{pposg99fdhye}
\end{eqnarray}
See \cite{paul2023graph} for more details on the above definitions and their properties.
Notice that every element of ${\sf OBS}(\Cbbb)$ is some finite minor anti-chain. On the other hand, as mentioned above, there is no result implying that  ${\sf OBS}(\Cbbb)$ (or, equivalently,  ${\sf cobs}(\Cbbb)$) is a finite set. However, when this is the case, \eqref{pposg99fdhye} gives a finite description of the class property 
$\Cbbb$ in terms of second order obstructions.
 
Whether or not the set of all finite anti-chains is well-quasi-ordered with respect to $\leq _{*}$
is known as the  \emph{$\omega^{2}$-Well-Quasi-Ordering Conjecture on minors}   
and can be seen as a special case of the conjecture that graphs are ``Better-Quasi-Ordered'' by  the minor relation (see \cite{Jancar99ANote,Thomas1989wellquasi,Requinot2017TowardsBetter,Marcone01Fine,paul2023graph} for more on these conjectures).

\subsection{Parametric graphs}

A \emph{parametric graph} is a {sequence}  $\mathscr{G}=\lin{\mathscr{G}_{k}}_{k\in\Nbbb_{\ge1}}$ of graphs and is \emph{minor-monotone} if, for every $i\leq j$, it holds that $\mathscr{G}_{i}$ is a minor of $\mathscr{G}_{j}$.
Moving forward, all parametric graphs considered are assumed to be minor-monotone. 
Given a parametric graph $\mathscr{P}=\lin{\mathscr{P}_{k}}_{k\in\Nbbb_{\ge1}}$ we denote by $\mathscr{P}    \!\downarrow$ the set of all minors of the graphs in $\mathscr{P}$. For more on parametric graphs and their relation with class properties see \cite{paul2023graph,paul2023universal}.

Let $\mathscr{G}=\langle \mathscr{G}_{k}\rangle_{k\in\Nbbb_{\ge1}}$ be some parametric graph. 
We define its \emph{apex extension} as 
the parametric graph  $\mathscr{G}^{\mathsf{a}}=\langle \mathscr{G}_{k}^{\mathsf{a}}\rangle_{k\in\Nbbb_{\ge1}}$
where $\mathscr{G}_{k}^{\mathsf{a}}$
is the graph obtained from $\mathscr{G}_{k}$ after adding a new vertex and making it adjacent to all vertices in $\mathscr{G}_{k}$.

Given a graph $G$ we denote by $k\cdot G$ the disjoint union of $k$ copies of $G$.
Given a graph $H$ we define $\mathscr{P}^{H}=\langle \mathscr{P}^{H}_{k}\rangle_{k\in \Nbbb_{\ge1}}$ where $\mathscr{P}^{H}_{k}=k\cdot H$.
We refer to $\mathscr{P}^{H}_{k}$ as the  \emph{$k$-packing} of $H$. Also given a graph $H$, we define the graph class 

\begin{equation}
\label{eq_nwriss}
    \Gcal^H = \mathscr{P}^H \!\downarrow.
\end{equation}

We use $K_{k}$ for the complete graph on $k$ vertices, also called the \emph{$k$-clique}.
We refer to $K_{3}$ as the \emph{triangle}.
We denote by $K_{q,r}$ the complete bipartite graph whose parts have $
q$ and $r$ vertices, respectively, and by $P_{k}$ the path on $k$ vertices.
Notice that $K_{1,k}=(\mathscr{P}^{K_{1}}_{k})^{\mathsf{a}}$.

\paragraph{Daisies, chains, and diamonds.}
We next define three parametric graphs that will be used in our results; see~\cref{fig_para_graphs} for examples.
We  define $\mathscr{D}_{k}\coloneqq (\mathscr{P}_{k}^{K_{2}})^{\mathsf{a}}$
and we refer to it as a \emph{$k$-daisy}.
We define $\mathscr{C}_{k}$ as the unique 
graph where all blocks are triangles, the maximum degree is 4, and no block has more than two cut-vertices. We use the term \emph{$k$-chain} for  $\mathscr{C}_{k}$.
We define $\mathscr{M}_{k}\coloneqq ((\mathscr{P}_{k}^{K_{1}})^{\mathsf{a}})^{\mathsf{a}}$
and call it a \emph{$k$-diamond}. Notice that 
$\mathscr{M}_{k}$ contains $K_{2,k}$ as a spanning subgraph
and that $\mathscr{M}_{k}$ is also contained in $K_{2,k+1}$
as a minor.

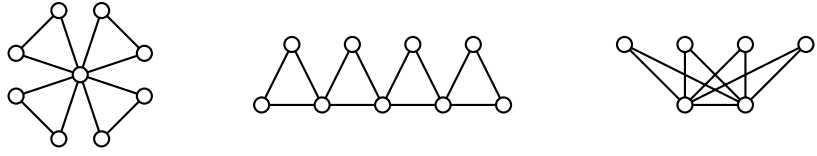
\begin{figure}[!h]
\centering
    \begin{tikzpicture}[scale=0.8, every node/.style={scale=0.8}]
    \begin{scope}[rotate=45]
    \draw[thick] (0,0) -- (0.5,1) -- (-0.5,1) -- cycle (0,0) -- (0.5,-1) -- (-0.5,-1) -- cycle (0,0) -- (-1,0.5) -- (-1,-0.5) -- cycle (0,0) -- (1,0.5) -- (1,-0.5) -- cycle;
    \foreach \i/\j in {0/0, 0.5/1, -0.5/1, 0.5/-1, -0.5/-1, -1/0.5, -1/-0.5, 1/0.5, 1/-0.5}
    \draw[thick] (\i,\j) node[uStyle] {};
    \end{scope}
    \begin{scope}[xshift=3cm, yshift=-0.5cm]
    \draw[thick] (0.5,1) -- (0,0) -- (1,0) -- cycle (1.5,1) -- (1,0) -- (2,0) -- cycle (2.5,1) -- (2,0) -- (3,0) -- cycle (3.5,1) -- (3,0) -- (4,0) --cycle;   
    \foreach \i/\j in {0/0, 0.5/1, 1/0, 1.5/1, 2/0, 2.5/1, 3/0, 3.5/1, 4/0}
    \draw[thick] (\i,\j) node[uStyle] {};
    \draw[white] (0,-1) node {}; %invisible for alignment
    \end{scope}
    \begin{scope}[xshift=10cm, yshift=-0.5cm]
    \draw[thick] (0,0) -- (1,0) -- (1,1) -- cycle (0,0) -- (0,1) -- (1,0) (0,0) -- (-1,1) -- (1,0) (0,0) -- (2,1) -- (1,0);
    \foreach \i/\j in {0/0, 1/0, 1/1, 0/1, -1/1, 2/1}
    \draw[thick] (\i,\j) node[uStyle] {};
    \draw[white] (0,-1) node {}; %invisible for alignment
    \end{scope}
    \end{tikzpicture}
    \caption{The $k$-daisy (left), the $k$-chain (center), and the $k$-diamond (right), when $k=4$.}
    \label{fig_para_graphs}
\end{figure}

\noindent The following is a direct consequence of the main result of Qiao and  Zhan in \cite{QiaoZhan2021Relation}.

\begin{proposition}
\label{hko_take}
If a tree $T$ has $\leq l$ leaves and diameter $\leq d$, then 
$|V(T)|\leq l\Big\lfloor\frac{d}{2}\Big\rfloor+1$.    
\end{proposition}

Given a graph $G$ and a vertex $v\in V(G)$
with two neighbors $x$ and $y$, the \emph{dissolution} of 
$v$ in $G$ is the graph obtained from $G-v$ after adding the edge $xy$ (if this edge does not already exist).

\begin{lemma} 
\label{msain_riop}
Let $G$ be a graph and $S \subseteq V(G)$. If, for some $k\in\Nbbb_{\geq 1}$, $G$ does not have $P_{k}$ and $K_{1,k}$ as an $S$-minor, then $|S|\le {k^2}$.
\end{lemma}

\begin{proof}
Let $T$ be a spanning tree of $G$. From $T$ we construct $T_{S}$ by repeatedly removing leaves that do not belong in $S$ until every leaf is in $S$,  and then by dissolving every degree-2 vertex not in $S$. These operations never delete vertices of $S$, and as they consist of vertex deletions and edge contractions, so $T_S$  is a minor of $T$. Thus, any $S$-minor model of $P_k$ or $K_{1,k}$ in $T_S$ would give one in $T$ and therefore in $G$, so $T_S$ also has no such $S$-minors. Notice that $T_S$ has at most  $k$ leaves, otherwise $T_{S}$ would contain $K_{1,k}$ as  an $S$-minor. 
Moreover, $T_{S}$ cannot contain a path on $k$ vertices, otherwise it contains $P_k$ as an $S$-minor. Therefore, $T_S$ has diameter at most $k-2$ and at most $k$ leaves. By \cref{hko_take}, we have that $|S| \leq  1 + k\Big\lfloor\frac{k-2}{2}\Big\rfloor\leq  k^2$
\end{proof}

\subsection{Erdős-Pósa properties for minors}

For our proofs, we need the following result on the Erdős-Pósa property for minors proved by Robertson and Seymour in \cite{RobertsonS86GMV}.
{Actually we provide a variant that is adequate for our purposes.}
\begin{proposition}

\label{erpos}
There is a function $f:\Nbbb^{2}\to\Nbbb$ such that 
if $\Zcal$ is a set of graphs containing at least one planar graph  and $h=n(\Zcal)$, then   every  graph $G$, that does not 
contain $k$ disjoint copies of a graph in $\Zcal$ as minors,  contains a vertex set $S\subseteq V(G)$  such that none of the graphs in $\Zcal$  is a minor of $G-S$ and $|S|\leq f(h,k)$.
\end{proposition}

In fact, \cref{erpos}, when $\Zcal=\{K_{3}\}$ is the classic result of  Erdős and Pósa in \cite{ErdosPosaOriginal} where $f(3,k)=O(k\log k)$. In general, it was proved in \cite{CamesVBat2019TightEP} that 
 $f(h,k)=f(h)\cdot k\cdot \log k$ for some function $f:\Nbbb\to\Nbbb$.
In our proofs, these optimizations on the function $f$  will not play any significant role as   we are dealing with the limiting behavior of graph classes.

\section{Class obstructions}
\label{xclsoosocnbsutuidison}

Given some non-negative rational number $\delta\geq 0$, we define 
\begin{eqnarray*}
\Cbbb_{\leq \delta} & = &\{\Gcal\mid  \text{$\Gcal$ is a minor-closed graph class where~}\delta(\Gcal)\le \delta\}
\end{eqnarray*}

Clearly, $\{\obs(\Gcal)\mid \Gcal\in\Cbbb_{\leq \delta}\}$    is the the set of the ${\sf yes}$-instances of \textsc{limiting Density $\leq \delta$}. 
Notice that if $\Gcal^{1}\subseteq \Gcal^{2}$, then $\delta(\Gcal^{1})\leq \delta(\Gcal^{2})$, therefore both  $\Cbbb_{\le\delta}$ and $\Cbbb_{<\delta}$ are subset-closed class properties.

From \cref{obs_nulnul}, it directly 
follows that $\cobs(\Cbbb_{<0})=\{\langle k\cdot K_{1}\rangle_{k\in\Nbbb}\downarrow\}.$

This section is devoted to  
the complete identification of ${\sf cobs}(\Cbbb_{<1})$ (\cref{subsec_atmost1}), of ${\sf cobs}(\Cbbb_{<\nicefrac{3}{2}})$ (\cref{subsec_3_2}), and of ${\sf cobs}(\Cbbb_{\leq \delta}),$ for every $\delta\in[0,1)$ (\cref{subsec_less1}) and
for every $\delta\in[1,\nicefrac{3}{2})$ (\cref{subsec_1_to_3_2}).

\subsection{Class obstructions for $\Cbbb_{<1}$}
\label{subsec_atmost1}

  In this {sub}section we identify the class obstructions of $\Cbbb_{<1}$.
We define 
$\mathscr{S}=\langle K_{1,k}\rangle_{k\in\Nbbb}$ and $\mathscr{P}=\langle P_{k}\rangle_{k\in\Nbbb}$. Let $\Gcal^{1}=\mathscr{P}^{K_{3}}    \!\downarrow$, $\Gcal^{2}=\mathscr{S}  \!\downarrow$, and  $\Gcal^{3}=\mathscr{P}  \!\downarrow$. 

\begin{lemma}
\label{jvdkjhjkdfkd}
If $\Gcal\in\{\Gcal^{1}, \Gcal^{2}, \Gcal^{3}\}$ then
$\delta(\Gcal)=1$.
\end{lemma}

 \begin{proof}
 If $\Gcal=\Gcal^{1}$, then let $n\in\Nbbb$ and  consider a graph $G\in\Gcal_{n}$.
Let $n_1,\ldots,n_{r}$ be the sizes of the connected components of $G$ and observe that $1\leq n_{i}\leq 3$, $i\in[r]$. Let $k=\lfloor n/3\rfloor$ and observe that $|E(G)|\leq \sum_{i\in[r]}\binom{n_i}{2}\leq k\binom{3}{2}+b_{n},$ where $b_n$ is 1 or 0 depending on whether $n   \!\mod 3 =2$ or not. 
This upper bound is sharp as witnessed by the disjoint union of $k\cdot K_{3}$ and the graph formed from $n   \!\mod 3$ pairwise adjacent vertices. Therefore, ${\sf ex}(\Gcal,n)=\lfloor n/3\rfloor\cdot \binom{3}{2}+b_{n}$ and $$\delta(\Gcal)=\lim_{n\to\infty}\frac{{\sf ex}(\Gcal,n)}{n}=\lim_{n\to\infty}\frac{\lfloor n/3\rfloor\cdot \binom{3}{2}+b_{n}}{n}=1.$$

If $\Gcal=\Gcal^{2}$ (resp. $\Gcal= \Gcal^{3}$), then for every $k\in\Nbbb$, it holds that  $K_{1,k}\in \Gcal$ (resp. $P_{k}\in \Gcal$). 
As every graph $G\in \Gcal_n$ is a subgraph of $K_{1,n-1}$ (resp. $P_{n}$), we have that ${\sf ex}(\Gcal,n)\leq n-1$.
Clearly, this upper bound is sharp as witnessed by $K_{1,n-1}$ (resp. $P_{n}$). Therefore, ${\sf ex}(\Gcal,n)=n-1$ and $\delta(\Gcal)=\lim_{n\to\infty}\frac{{\sf ex}(\Gcal,n)}{n}=\lim_{n\to\infty}\frac{n-1}{n}=1$. 
\end{proof}

\begin{lemma}
\label{l664gdjd}
If $\Gcal$ is a graph class that does not contain $\Gcal^{1}$
as a subset and excludes the graphs $K_{1,k}$ and $P_{k}$ as minors, then $\delta(\Gcal)\le 1-1/k^2$. 
\end{lemma}

\begin{proof}
As $\Gcal$ is a graph class that does not contain $\Gcal^{1}$, there is 
some $c$ such that $c\cdot K_{3}$ is not a minor of $\Gcal$.
This means that if $G\in\Gcal$,
then less than $c$ of the connected components of $G$ are non-trees.
Let $C$ be a connected component of $G$ that is not a tree.
As $C$ excludes $K_{1,k}$ and $P_{k}$, it has bounded degree $<k$ and diameter $<k-1$.
Therefore, $|V(C)|\leq k^k$. Let $C_G$ be the union of all connected components 
of $G$ that are non-trees. Observe that $|V(C_G)|<ck^k$.
Similarly, if $H$ is a connected component of $G$ that is a tree, then its diameter is less than $k-1$. Moreover, since $K_{1,k}$
is excluded as a minor, $H$ has less than $k$ leaves. By \cref{hko_take}, we have $|V(H)|\le 1+(k-1)\lfloor\frac{k-2}{2}\rfloor\le k^2$.

Now let $G\in \Gcal_{n}$ and define $z:=|V(C_G)|$. Let $r$ be the number of 
connected components of $G$ that are trees, and let  $n_{1},\ldots,n_{r}$ be their sizes.
Recall that $n_{i}\leq k^2$, for every $i\in[r]$. Therefore, $\sum_{i\in[r]}n_i\leq rk^2$. Furthermore, if we let $p:=c\cdot {k^k\choose 2}$, then 
$\ex(\Gcal,n)\leq p+\sum_{i\in[r]}(n_i-1)=p+(\sum_{i\in[r]}n_i)-r$. Moreover, 
$n\geq \sum_{i\in[r]}n_i\geq n-ck^k$. This implies that  
\begin{eqnarray*}
\frac{\ex(\Gcal,n)}{n} & \leq  & \frac{p+(\sum_{i\in[r]}n_i)-r}{\sum_{i\in[r]}n_i}\\
& \leq  & \frac{p}{\sum_{i\in[r]}n_i}+1-\frac{r}{\sum_{i\in[r]}n_i}\\
& \leq  & \frac{p}{n-ck^k}+1-\frac{r}{\sum_{i\in[r]}n_i}\\
& \leq  & \frac{p}{n-ck^k}+1-\frac{r}{rk^2}\\
& =& \frac{p}{n-ck^k}+1-\frac{1}{k^2}
\end{eqnarray*}
We conclude  that $\delta(\Gcal)=\lim_{n\to\infty}\frac{\ex(\Gcal,n)}{n}\leq (\lim_{n\to\infty}\frac{p}{n-ck^k})+1-\frac{1}{k^2}=1-\frac{1}{k^2}$.
\end{proof}

\begin{theorem}
\label{adsfssbfbscf}
$\cobs(\Cbbb_{<1})=\{\Gcal^{1}, \Gcal^{2}, \Gcal^{3}\}$.
\end{theorem}

\begin{proof}
Notice first that none of the classes $\Gcal^{1}, \Gcal^{2}, \Gcal^{3}$ is a subset of the other. 
Therefore, they constitute a $\subseteq$-anti-chain.
By \cref{jvdkjhjkdfkd}, if a class $\Gcal$ contains some $\Gcal^i\in\{
\Gcal^{1},\Gcal^{2},\Gcal^{3}\}$ as a subset, then $\delta(\Gcal)\geq \delta(\Gcal^i)\geq 1$.
If, instead, $\Gcal$ does not contain any of $\Gcal^{1}$, $\Gcal^{2}$, or $\Gcal^{3}$ as subsets,
then by \cref{l664gdjd},  there is some $x$ such that $\delta(\Gcal)\leq (x-1)/x<1$.
\end{proof}

Clearly, $\Cbbb_{<1}\subseteq \Cbbb_{\leq 1}$. However, because of \cref{jvdkjhjkdfkd}, 
$\Cbbb_{\leq 1}$ contains classes that are not members of $\Cbbb_{<1}$ such as 
the class obstructions of $\Cbbb_{<1}$, i.e., $\Gcal^{1}, \Gcal^{2}, \Gcal^{3}$.

\subsection{Class obstructions for $\Cbbb_{<\nicefrac{3}{2}}$}

\label{subsec_3_2}

Our next step is to identify the class obstructions of $\Cbbb_{<\nicefrac{3}{2}}$. The proofs follow the same general approach as in the $\Cbbb_{<1}$ case, but they are more intricate and necessitate the introduction of additional concepts.

\paragraph{$k$-trees.}
Let $k\in \Nbbb_{\geq 1}$.
A vertex $v$ in a graph $G$ is \emph{$k$-simplicial} if $G[N_{G}(v)]$ is a $k$-clique. 
A \emph{$k$-tree}  is recursively defined as follows: $K_{k+1}$ is a $k$-tree
and every graph $G$ on $>k+1$ vertices containing a $k$-simplicial vertex $v$ such that 
$G-v$ is a $k$-tree is also a $k$-tree.
 
A graph $G$ is \emph{chordal} if it does not contain a cycle on $\geq 4$ vertices as an induced subgraph.
A graph is \emph{$k$-chordal}
if it is  a subgraph of a $k$-tree and it is chordal. Notice that $k$-trees are the $k$-chordal graphs that are  $k$-connected.\medskip

\begin{figure}[!h]
    \centering
    \begin{tikzpicture}[scale=0.8, every node/.style={scale=0.8}]
    \begin{scope}
    \draw[thick] (0,0) -- (1,0) -- (0.5,1) -- cycle; 
    \foreach \i/\j in {0/0, 1/0, 0.5/1}
    \draw[thick] (\i,\j) node[uStyle] {};
    \end{scope}
     \begin{scope}[xshift=3cm]
     \draw[thick] (0,0) -- (1,0) -- (0.5,1) -- cycle (0,0) -- (0.5,-1) -- (1,0);
     \foreach \i/\j in {0/0, 1/0, 0.5/1, 0.5/-1}
    \draw[thick] (\i,\j) node[uStyle] {};
     \end{scope}   
     \begin{scope}[xshift=6cm]
     \draw[thick] (0,0) -- (1,0) -- (0.5,1) -- cycle (0,0) -- (0.5,-1) -- (1,0) (0.5,1) -- (1.5,1) -- (1,0); 
     \foreach \i/\j in {0/0, 1/0, 0.5/1, 0.5/-1, 1.5/1}
    \draw[thick] (\i,\j) node[uStyle] {};
     \end{scope} 
     \begin{scope}[xshift=8cm]
     \draw[thick] (0,0) -- (1,0) (0,0) -- (0,1) -- (1,0) (0,0) -- (1,1) -- (1,0) (0,0) -- (0.5,-1) -- (1,0); 
     \foreach \i/\j in {0/0, 1/0, 0/1, 0.5/-1, 1/1}
    \draw[thick] (\i,\j) node[uStyle] {};
     \end{scope} 
    \end{tikzpicture}
    \caption{The 2-trees on 3 (left), 4 (center), and 5 (right) vertices.}
    \label{2-trees}
\end{figure}
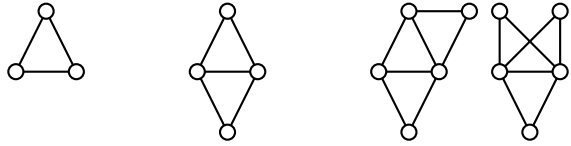

Consider the sets of graphs in~\cref{fig_chor_2}. We denote by $\Acal_6$ all the $2$-trees  on 6 vertices. 
We denote by $\Acal_8$ the class of connected graphs with  exactly two blocks one isomorphic to a $5$-vertex $2$-tree and the other to the unique $4$-vertex $2$-tree (see~\cref{2-trees}). We denote by $\Acal_{10}$ all the connected graphs with three blocks, each isomorphic to the unique $4$-vertex $2$-tree. Lastly, we let $\Acal_4 = \{K_4\}$. Let $$\Acal =\bigcup\limits_{j\in\{4,6,8,10\}}\Acal_j$$ and let $$\Abbb _j = \{\mathscr{P}^{A}\downarrow \; |\; A \in \mathcal{A}_j\}\mbox{ where $j\in\{4,6,8,10\}$}$$ We set $\Abbb  = \bigcup\limits_{j\in\{4,6,8,10\}}\Abbb _j$. Notice that $\Abbb$ contains 30 graph classes.
We also define $\Dcal=\langle \mathscr{D}_{k}\rangle_{k\in \Nbbb}\!\downarrow$,
$\Ccal=\langle \mathscr{C}_{k}\rangle_{k\in \Nbbb}\!\downarrow$, $\Mcal=\langle \mathscr{M}_{k}\rangle_{k\in \Nbbb}\!\downarrow$.

\begin{figure}[h]
\centering
\scalebox{1.17}{\tikzstyle{uStyle}=[shape = circle, minimum size = 2pt, inner sep =1.3pt, outer sep = 0pt, draw, fill=white]
\begin{tikzpicture}[scale = 0.4]
\begin{scope}[local bounding box=groupD, xshift=-15.75cm, yshift=-11.7cm]
  \node[uStyle] (s) at (5,3.7) {};
  \node[uStyle] (d) at ($(s)+ (90:1)$) {};
  \node[uStyle] (c) at ($(s)+ (0:1)$) {};
  \node[uStyle] (a) at ($(c)+ (90:1)$) {};
  \node (ggg) at (5.5, 2.7) {\footnotesize$\mathcal{A}_ {4}$};
  \draw (a)--(c)--(d)--(s)--(a)--(d) (s)--(c);
\end{scope}
\begin{scope}
   \draw
     ($(groupD.south west)+(-7pt,-7pt)$) rectangle
     ($(groupD.north east)+(7pt,7pt)$);
\end{scope}
\begin{scope}[local bounding box=groupA, xshift=-7.75cm, yshift=-7cm]
  \node[uStyle] (a) at (0,0) {};
  \node[uStyle] (b) at ($(a) + (60:1)$) {};
  \node[uStyle] (c) at ($(a) + (0:1)$) {};
  \node[uStyle] (d) at ($(a) + (120:1)$) {};
  \node[uStyle] (e) at ($(c) + (60:1)$) {};
  \node[uStyle] (f) at ($(a) + (-60:1)$) {};
  \draw (d) -- (a) -- (b) -- (c) -- (e);
  \draw (d) -- (c) -- (f) -- (a) -- (c);
  \draw (e) -- (a);
  
  \node[uStyle] (g) at (2,0) {};
  \node[uStyle] (h) at ($(g) + (60:1)$) {};
  \node[uStyle] (i) at ($(g) + (0:1)$) {};
  \node[uStyle] (j) at ($(i) + (60:1)$) {};
  \node[uStyle] (k) at ($(i) + (0:1)$) {};
  \node[uStyle] (l) at ($(k) + (-120:1)$) {};
  \draw (g) -- (h) -- (i) -- (j) -- (k) -- (l) -- (i) -- (k);
  \draw (h) -- (j);
  \draw (i) -- (g);
  
  \node[uStyle] (s) at (5,0) {};
  \node[uStyle] (t) at ($(s) + (60:1)$) {};
  \node[uStyle] (u) at ($(s) + (0:1)$) {};
  \node[uStyle] (v) at ($(s) + (120:1)$) {};
  \node[uStyle] (w) at ($(u) + (60:1)$) {};
  \node[uStyle] (x) at ($(s) + (-60:1)$) {};
  \draw (v) -- (t) --  (w) -- (u) -- (x) -- (s) -- (v);
  \draw (s) -- (t) -- (u) -- (s);
  
  \node[uStyle] (m) at (7,0) {};
  \node[uStyle] (n) at ($(m) + (60:1)$) {};
  \node[uStyle] (o) at ($(m) + (0:1)$) {};
  \node[uStyle] (p) at ($(o) + (60:1)$) {};
  \node[uStyle] (q) at ($(o) + (0:1)$) {};
  \node[uStyle] (r) at ($(q) + (60:1)$) {};
  \draw (m) -- (n) -- (p) -- (r) -- (q) -- (p) -- (o) -- (n);
  \draw (m) -- (o) -- (q);

  \node[uStyle] (y) at (10,0) {};
  \node[uStyle] (z) at ($(y) + (60:1)$) {};
  \node[uStyle] (A) at ($(y) + (0:1)$) {};
  \node[uStyle] (B) at ($(A) + (60:1)$) {};
  \node[uStyle] (C) at ($(A) + (0:1)$) {};
  \node[uStyle] (D) at ($(C) + (-120:1)$) {};
  \draw (y) -- (z) -- (A) -- (B) -- (C) -- (A);
  \draw (z) -- (B);
  \draw (A) -- (y);
  \draw (A) -- (D) -- (B);
  \node (E) at (5.5,-2) {\footnotesize$\mathcal{A}_6$};
\end{scope}
\begin{scope}
   \draw
     ($(groupA.south west)+(-7pt,-7pt)$) rectangle
     ($(groupA.north east)+(7pt,7pt)$);
\end{scope}
\begin{scope}[local bounding box=groupB, xshift=8.25cm, yshift=0cm]
\node[uStyle] (F) at (-2.69,-5) {};
\node[uStyle] (G) at ($(F) + (60:1)$) {};
\node[uStyle] (H) at ($(F) + (0:1)$) {};
\node[uStyle] (I) at ($(F) + (-60:1)$) {};
\node[uStyle] (K) at ($(H) + (60:1)$) {};
\node[uStyle] (L) at ($(H) + (0:1)$) {};  
\node[uStyle] (M) at ($(L) + (60:1)$) {};
\node[uStyle] (N) at ($(L) + (0:1)$) {};
\draw (F) -- (G) -- (H) -- (I) -- (F) -- (H);
\draw (H) -- (K) -- (L) -- (M) -- (N) -- (L) -- (H);
\draw (K) -- (M);
\node[uStyle] (O) at (1.31,-5) {};
\node[uStyle] (P) at ($(O) + (60:1)$) {};
\node[uStyle] (Q) at ($(O) + (0:1)$) {};
\node[uStyle] (R) at ($(O) + (-60:1)$) {};
\node[uStyle] (S) at ($(Q) + (90:1)$) {};
\node[uStyle] (T) at ($(Q) + (-90:1)$) {};  
\node[uStyle] (U) at ($(Q) + (30:1)$) {};
\node[uStyle] (V) at ($(Q) + (-30:1)$) {};
\draw (O) -- (P) -- (Q) -- (R) -- (O) -- (Q);
\draw (Q) -- (R) (Q) -- (S)  (Q) -- (T) (Q) -- (U) (Q) -- (V);
\draw (S) -- (U) -- (V) -- (T);
\node[uStyle] (W) at (4.17,-5) {};
\node[uStyle] (X) at ($(W) + (60:1)$) {};
\node[uStyle] (Y) at ($(W) + (0:1)$) {};
\node[uStyle] (Z) at ($(W) + (-60:1)$) {};
\node[uStyle] (aa) at ($(Y) + (-60:1)$) {};
\node[uStyle] (bb) at ($(Y) + (0:1)$) {};  
\node[uStyle] (cc) at ($(aa) + (0:1)$) {};
\node[uStyle] (dd) at ($(Y) + (60:1)$) {};
\draw (W) -- (X) -- (Y) -- (Z) -- (W) -- (Y);
\draw (bb) -- (Y) (bb) -- (aa) (bb) -- (cc) (bb) -- (dd);
\draw (dd) -- (Y) -- (aa) -- (cc);
\node[uStyle] (ee) at (7.17,-5) {};
\node[uStyle] (ff) at ($(ee) + (30:1)$) {};
\node[uStyle] (gg) at ($(ee) + (-30:1)$) {};
\node[uStyle] (hh) at ($(gg) + (30:1)$) {};
\node[uStyle] (ii) at ($(hh) + (90:1)$) {};
\node[uStyle] (jj) at ($(hh) + (-90:1)$) {};  
\node[uStyle] (kk) at ($(hh) + (30:1)$) {};
\node[uStyle] (ll) at ($(hh) + (-30:1)$) {};
\draw (ff) -- (hh) -- (gg) -- (ee) -- (ff) -- (gg);
\draw (hh) -- (ii) (hh) -- (jj)  (hh) -- (kk) (hh) -- (ll);
\draw (ii) -- (kk) -- (ll) -- (jj);
\node[uStyle] (mm) at (10.75,-5) {};
\node[uStyle] (nn) at ($(mm) + (30:1)$) {};
\node[uStyle] (oo) at ($(mm) + (-30:1)$) {};
\node[uStyle] (pp) at ($(oo) + (30:1)$) {};
\node[uStyle] (qq) at ($(pp) + (90:1)$) {};  
\node[uStyle] (ss) at ($(pp) + (30:1)$) {};
\node[uStyle] (tt) at ($(pp) + (-30:1)$) {};
\node[uStyle] (rr) at ($(tt) + (30:1)$) {};
\draw (nn) -- (pp) -- (oo) -- (mm) -- (nn) -- (oo);
\draw (ss) -- (qq) (ss) -- (pp)  (ss) -- (tt) (ss) -- (rr);
\draw (qq) -- (pp) -- (tt) -- (rr);
\node[uStyle] (ss) at (-2.79,-7.8) {};
\node[uStyle] (tt) at ($(ss) + (30:1)$) {};
\node[uStyle] (uu) at ($(ss) + (-30:1)$) {};
\node[uStyle] (vv) at ($(uu) + (30:1)$) {};
\node[uStyle] (ww) at ($(vv) + (60:1)$) {};
\node[uStyle] (xx) at ($(vv) + (0:1)$) {};  
\node[uStyle] (yy) at ($(xx) + (60:1)$) {};
\node[uStyle] (zz) at ($(xx) + (0:1)$) {};
\draw (tt) -- (vv) -- (uu) -- (ss) -- (tt) -- (uu);
\draw (xx) -- (yy) (xx) -- (zz)  (xx) -- (ww) (xx) -- (vv);
\draw (vv) -- (ww) -- (yy) -- (zz);
\node[uStyle] (AA) at (1.93,-8.1) {};
  \node[uStyle] (BB) at ($(AA) + (60:1)$) {};
  \node[uStyle] (CC) at ($(AA) + (0:1)$) {};
  \node[uStyle] (DD) at ($(AA) + (120:1)$) {};
  \node[uStyle] (EE) at ($(CC) + (60:1)$) {};
  \node[uStyle] (FF) at ($(EE) + (60:1)$) {};
  \node[uStyle] (GG) at ($(EE) + (-60:1)$) {};
  \node[uStyle] (HH) at ($(EE) + (0:1)$) {};
  \draw (DD) -- (AA) -- (BB) -- (CC) -- (EE);
  \draw (DD) -- (CC) -- (AA);
  \draw (EE) -- (AA);
  \draw (EE) -- (FF) -- (HH) -- (GG) -- (EE) -- (HH);
  \node[uStyle] (II) at (5.43,-7.7) {};
  \node[uStyle] (JJ) at ($(II) + (60:1)$) {};
  \node[uStyle] (KK) at ($(II) + (0:1)$) {};
  \node[uStyle] (LL) at ($(II) + (120:1)$) {};
  \node[uStyle] (MM) at ($(KK) + (60:1)$) {};
  \node[uStyle] (OO) at ($(MM) + (-60:1)$) {};
  \node[uStyle] (PP) at ($(MM) + (0:1)$) {};
  \node[uStyle] (NN) at ($(PP) + (-60:1)$) {};
  \draw (LL) -- (II) -- (JJ) -- (KK) -- (MM);
  \draw (LL) -- (KK) -- (II);
  \draw (MM) -- (II);
  \draw (OO) -- (MM) -- (PP) -- (NN) -- (OO) -- (PP);
  \node[uStyle] (XX) at (8.93,-7.3) {};
\node[uStyle] (QQ) at ($(XX) + (60:1)$) {};
\node[uStyle] (RR) at ($(XX) + (0:1)$) {};
\node[uStyle] (SS) at ($(XX) + (-60:1)$) {};
\node[uStyle] (TT) at ($(RR) + (60:1)$) {};
\node[uStyle] (UU) at ($(RR) + (0:1)$) {};  
\node[uStyle] (VV) at ($(RR) + (-60:1)$) {};
\node[uStyle] (WW) at ($(VV) + (0:1)$) {};
\draw (XX) -- (QQ) -- (RR) -- (SS) -- (XX) -- (RR);
\draw (TT) -- (RR) -- (UU) -- (VV) -- (WW);
\draw (TT) -- (VV) (WW) -- (RR) -- (VV);  
\node[uStyle] (YY) at (11.9,-7.3) {};
  \node[uStyle] (ZZ) at ($(YY) + (60:1)$) {};
  \node[uStyle] (aaa) at ($(YY) + (0:1)$) {};
  \node[uStyle] (bbb) at ($(YY) + (120:1)$) {};
  \node[uStyle] (ccc) at ($(aaa) + (60:1)$) {};
  \draw (bbb) -- (YY) -- (ZZ) -- (aaa) -- (ccc);
  \draw (bbb) -- (aaa) -- (YY);
  \draw (ccc) -- (YY);
  \node[uStyle] (ddd) at ($(aaa) + (-60:1)$) {};
  \node[uStyle] (eee) at ($(aaa) + (0:1)$) {};
  \node[uStyle] (fff) at ($(ddd) + (0:1)$) {};
  \draw (eee) -- (aaa) -- (ddd) -- (fff) -- (eee) -- (ddd); 
  \node (ggg) at (5.5, -9) {\footnotesize $\mathcal{A}_8$};
  \end{scope}
  \begin{scope}
      \draw
     ($(groupB.south west)+(-7pt,-7pt)$) rectangle
     ($(groupB.north east)+(7pt,7pt)$);
  \end{scope}
  \begin{scope}[local bounding box = groupC, xshift=1.5cm, yshift=0.15cm]
      \node[uStyle] (hhh) at (-10.5,-11.8) {};
      \node[uStyle] (iii) at ($(hhh) + (60:1)$) {};
      \node[uStyle] (jjj) at ($(hhh) + (0:1)$) {};
      \node[uStyle] (kkk) at ($(hhh) + (-60:1)$) {};
      \node[uStyle] (lll) at ($(jjj) + (60:1)$) {};
      \node[uStyle] (mmm) at ($(jjj) + (-60:1)$) {};  
      \node[uStyle] (nnn) at ($(jjj) + (0:1)$) {};
      \node[uStyle] (ooo) at ($(nnn) + (60:1)$) {};
      \node[uStyle] (ppp) at ($(nnn) + (-60:1)$) {};
      \node[uStyle] (qqq) at ($(nnn) + (0:1)$) {};
      \draw (hhh) -- (iii) -- (jjj) -- (kkk) -- (hhh) -- (jjj);
      \draw (jjj) -- (lll) --  (nnn) -- (mmm) -- (jjj) -- (nnn);
      \draw (nnn) -- (ooo) -- (qqq) -- (ppp) -- (nnn) -- (qqq);
      \node[uStyle] (rrr) at (-6.56,-11.8) {};
      \node[uStyle] (sss) at ($(rrr) + (60:1)$) {};
      \node[uStyle] (ttt) at ($(rrr) + (0:1)$) {};
      \node[uStyle] (uuu) at ($(rrr) + (-60:1)$) {};
      \node[uStyle] (vvv) at ($(ttt) + (60:1)$) {};
      \node[uStyle] (yyy) at ($(ttt) + (0:1)$) {};
      \node[uStyle] (xxx) at ($(ttt) + (-60:1)$) {};  
      \node[uStyle] (zzz) at ($(yyy) + (30:1)$) {};
      \node[uStyle] (AAA) at ($(yyy) + (-30:1)$) {};
      \node[uStyle] (BBB) at ($(zzz) + (-30:1)$) {};
      \draw (rrr) -- (sss) -- (ttt) -- (uuu) -- (rrr) -- (ttt);
      \draw (ttt) -- (vvv) --  (yyy) -- (xxx) -- (ttt) -- (yyy);
      \draw (zzz) -- (BBB) -- (AAA) -- (yyy) -- (zzz) -- (AAA);
      \node[uStyle] (CCC) at (-1.94,-11.8) {};
      \node[uStyle] (DDD) at ($(CCC) + (30:1)$) {};
      \node[uStyle] (EEE) at ($(CCC) + (-30:1)$) {};
      \node[uStyle] (FFF) at ($(DDD) + (-30:1)$) {};
      \node[uStyle] (GGG) at ($(FFF) + (60:1)$) {};
      \node[uStyle] (HHH) at ($(FFF) + (0:1)$) {};
      \node[uStyle] (III) at ($(FFF) + (-60:1)$) {};  
      \node[uStyle] (JJJ) at ($(HHH) + (30:1)$) {};
      \node[uStyle] (KKK) at ($(HHH) + (-30:1)$) {};
      \node[uStyle] (LLL) at ($(JJJ) + (-30:1)$) {};
      \draw (DDD) -- (FFF) -- (EEE) -- (CCC) -- (DDD) -- (EEE);
      \draw (FFF) -- (GGG) --  (HHH) -- (III) -- (FFF) -- (HHH);
      \draw (JJJ) -- (LLL) -- (KKK) -- (HHH) -- (JJJ) -- (KKK);
      \node[uStyle] (MMM) at (3.48,-11.8) {};
      \node[uStyle] (NNN) at ($(MMM) + (30:1)$) {};
      \node[uStyle] (OOO) at ($(MMM) + (-30:1)$) {};
      \node[uStyle] (PPP) at ($(NNN) + (-30:1)$) {};
      \node[uStyle] (QQQ) at ($(PPP) + (30:1)$) {};
      \node[uStyle] (RRR) at ($(PPP) + (-30:1)$) {};
      \node[uStyle] (SSS) at ($(QQQ) + (-30:1)$) {};  
      \node[uStyle] (TTT) at ($(SSS) + (30:1)$) {};
      \node[uStyle] (UUU) at ($(SSS) + (-30:1)$) {};
      \node[uStyle] (VVV) at ($(TTT) + (-30:1)$) {};
      \draw (NNN) -- (PPP) -- (OOO) -- (MMM) -- (NNN) -- (OOO);
      \draw (QQQ) -- (SSS) --  (RRR) -- (PPP) -- (QQQ) -- (RRR);
      \draw (TTT) -- (VVV) -- (UUU) -- (SSS) -- (TTT) -- (UUU);
      \node[uStyle] (XXX) at (9.52,-11.8) {};
      \node[uStyle] (YYY) at ($(XXX) + (30:1)$) {};
      \node[uStyle] (WWW) at ($(XXX) + (-30:1)$) {};
      \node[uStyle] (ZZZ) at ($(YYY) + (-30:1)$) {};
      \node[uStyle] (aaaa) at ($(ZZZ) + (30:1)$) {};
      \node[uStyle] (bbbb) at ($(ZZZ) + (-30:1)$) {};
      \node[uStyle] (cccc) at ($(aaaa) + (-30:1)$) {};  
      \node[uStyle] (dddd) at ($(cccc) + (60:1)$) {};
      \node[uStyle] (eeee) at ($(cccc) + (0:1)$) {};
      \node[uStyle] (ffff) at ($(cccc) + (-60:1)$) {};
      \draw (YYY) -- (ZZZ) -- (WWW) -- (XXX) -- (YYY) -- (WWW);
      \draw (aaaa) -- (cccc) --  (bbbb) -- (ZZZ) -- (aaaa) -- (bbbb);
      \draw (cccc) -- (dddd) -- (eeee) -- (ffff) -- (cccc) -- (eeee);
      \node[uStyle] (hhh) at (15.02,-11.8) {};
      \node[uStyle] (iii) at ($(hhh) + (60:1)$) {};
      \node[uStyle] (jjj) at ($(hhh) + (0:1)$) {};
      \node[uStyle] (kkk) at ($(hhh) + (-60:1)$) {};
      \node[uStyle] (lll) at ($(jjj) + (30:1)$) {};
      \node[uStyle] (mmm) at ($(jjj) + (-30:1)$) {};  
      \node[uStyle] (nnn) at ($(lll) + (-30:1)$) {};
      \node[uStyle] (ooo) at ($(nnn) + (60:1)$) {};
      \node[uStyle] (ppp) at ($(nnn) + (-60:1)$) {};
      \node[uStyle] (qqq) at ($(nnn) + (0:1)$) {};
      \draw (hhh) -- (iii) -- (jjj) -- (kkk) -- (hhh) -- (jjj);
      \draw (lll) --  (nnn) -- (mmm) -- (jjj) -- (lll) -- (mmm);
      \draw (nnn) -- (ooo) -- (qqq) -- (ppp) -- (nnn) -- (qqq);
      \node[uStyle] (rrr) at (-10.56,-14.3) {};
      \node[uStyle] (sss) at ($(rrr) + (60:1)$) {};
      \node[uStyle] (ttt) at ($(rrr) + (0:1)$) {};
      \node[uStyle] (uuu) at ($(rrr) + (-60:1)$) {};
      \node[uStyle] (vvv) at ($(ttt) + (60:1)$) {};
      \node[uStyle] (yyy) at ($(ttt) + (0:1)$) {};
      \node[uStyle] (xxx) at ($(ttt) + (-60:1)$) {};  
      \node[uStyle] (zzz) at ($(xxx) + (-60:1)$) {};
      \node[uStyle] (AAA) at ($(xxx) + (-120:1)$) {};
      \node[uStyle] (BBB) at ($(AAA) + (-60:1)$) {};
      \draw (rrr) -- (sss) -- (ttt) -- (uuu) -- (rrr) -- (ttt);
      \draw (ttt) -- (vvv) --  (yyy) -- (xxx) -- (ttt) -- (yyy);
      \draw (zzz) -- (BBB) -- (AAA) -- (xxx) -- (zzz) -- (AAA);
      \node[uStyle] (rrr) at (-7.62,-14.3) {};
      \node[uStyle] (sss) at ($(rrr) + (60:1)$) {};
      \node[uStyle] (ttt) at ($(rrr) + (0:1)$) {};
      \node[uStyle] (uuu) at ($(rrr) + (-60:1)$) {};
      \node[uStyle] (vvv) at ($(ttt) + (60:1)$) {};
      \node[uStyle] (yyy) at ($(ttt) + (0:1)$) {};
      \node[uStyle] (xxx) at ($(ttt) + (-60:1)$) {};  
      \node[uStyle] (zzz) at ($(xxx) + (-30:1)$) {};
      \node[uStyle] (AAA) at ($(xxx) + (-150:1)$) {};
      \node[uStyle] (BBB) at ($(xxx) + (-90:1)$) {};
      \draw (rrr) -- (sss) -- (ttt) -- (uuu) -- (rrr) -- (ttt);
      \draw (ttt) -- (vvv) --  (yyy) -- (xxx) -- (ttt) -- (yyy);
      \draw (xxx) -- (zzz) -- (BBB) -- (AAA) -- (xxx) -- (BBB);
      \node[uStyle] (rrr) at (-4.68,-14.3) {};
      \node[uStyle] (sss) at ($(rrr) + (60:1)$) {};
      \node[uStyle] (ttt) at ($(rrr) + (0:1)$) {};
      \node[uStyle] (uuu) at ($(rrr) + (-60:1)$) {};
      \node[uStyle] (vvv) at ($(ttt) + (30:1)$) {};
      \node[uStyle] (yyy) at ($(ttt) + (-30:1)$) {};
      \node[uStyle] (xxx) at ($(yyy) + (30:1)$) {};  
      \node[uStyle] (zzz) at ($(yyy) + (-60:1)$) {};
      \node[uStyle] (AAA) at ($(yyy) + (-120:1)$) {};
      \node[uStyle] (BBB) at ($(AAA) + (-60:1)$) {};
      \draw (rrr) -- (sss) -- (ttt) -- (uuu) -- (rrr) -- (ttt);
      \draw (ttt) -- (vvv) --  (yyy) -- (xxx) (ttt) -- (yyy) (xxx) -- (vvv);
      \draw (zzz) -- (BBB) -- (AAA) -- (yyy) -- (zzz) -- (AAA);
      \node[uStyle] (rrr) at (-1.02,-14.3) {};
      \node[uStyle] (sss) at ($(rrr) + (30:1)$) {};
      \node[uStyle] (ttt) at ($(rrr) + (-30:1)$) {};
      \node[uStyle] (uuu) at ($(sss) + (-30:1)$) {};
      \node[uStyle] (vvv) at ($(uuu) + (30:1)$) {};
      \node[uStyle] (yyy) at ($(uuu) + (-30:1)$) {};
      \node[uStyle] (xxx) at ($(yyy) + (30:1)$) {};  
      \node[uStyle] (zzz) at ($(yyy) + (-60:1)$) {};
      \node[uStyle] (AAA) at ($(yyy) + (-120:1)$) {};
      \node[uStyle] (BBB) at ($(AAA) + (-60:1)$) {};
      \draw (rrr) -- (sss) -- (ttt) -- (uuu) (rrr) -- (ttt)(uuu) -- (sss);
      \draw (uuu) -- (vvv) --  (yyy) -- (xxx) (uuu) -- (yyy) (xxx) -- (vvv);
      \draw (zzz) -- (BBB) -- (AAA) -- (yyy) -- (zzz) -- (AAA);
%      \node[uStyle] (rrr) at (3.36,-14.3) {};
%      \node[uStyle] (sss) at ($(rrr) + (30:1)$) {};
%      \node[uStyle] (ttt) at ($(rrr) + (-30:1)$) {};
%      \node[uStyle] (uuu) at ($(sss) + (-30:1)$) {};
%      \node[uStyle] (vvv) at ($(uuu) + (30:1)$) {};
%      \node[uStyle] (yyy) at ($(uuu) + (-30:1)$) {SSS};
%      \node[uStyle] (xxx) at ($(yyy) + (30:1)$) {};  
%      \node[uStyle] (zzz) at ($(yyy) + (-30:1)$) {};
%      \node[uStyle] (AAA) at ($(yyy) + (-150:1)$) {};
%      \node[uStyle] (BBB) at ($(AAA) + (-30:1)$) {};
%      \draw (rrr) -- (sss) -- (ttt) -- (uuu) (rrr) -- (ttt)(uuu) -- (sss);
%      \draw (uuu) -- (vvv) --  (yyy) -- (xxx) (uuu) -- (yyy) (xxx) -- (vvv);
%      \draw (zzz) -- (BBB) -- (AAA) -- (yyy) -- (zzz)  (BBB) -- (yyy);
%      \node[uStyle] (rrr) at (7.74,-14.3) {};
%      \node[uStyle] (sss) at ($(rrr) + (30:1)$) {};
%      \node[uStyle] (ttt) at ($(rrr) + (-30:1)$) {sss};
%      \node[uStyle] (uuu) at ($(sss) + (-30:1)$) {};
%      \node[uStyle] (vvv) at ($(uuu) + (60:1)$) {};
%      \node[uStyle] (yyy) at ($(uuu) + (-60:1)$) {};
%      \node[uStyle] (xxx) at ($(uuu) + (0:1)$) {};  
%      \node[uStyle] (zzz) at ($(yyy) + (-30:1)$) {};
%      \node[uStyle] (AAA) at ($(yyy) + (-150:1)$) {};
%      \node[uStyle] (BBB) at ($(AAA) + (-30:1)$) {};
%      \draw (rrr) -- (sss) -- (ttt) -- (uuu) (rrr) -- (ttt)(uuu) -- (sss);
%      \draw (uuu) -- (vvv) (yyy) -- (xxx) (uuu) -- (yyy) (xxx) -- (vvv) (uuu) -- (xxx);
%      \draw (zzz) -- (BBB) -- (AAA) -- (yyy) -- (zzz)  (BBB) -- (yyy);
      \node[uStyle] (rrr) at (6.8,-14.3) {};
      \node[uStyle] (sss) at ($(rrr) + (30:1)$) {};
      \node[uStyle] (ttt) at ($(rrr) + (-30:1)$) {};
      \node[uStyle] (uuu) at ($(sss) + (-30:1)$) {};
      \node[uStyle] (vvv) at ($(uuu) + (30:1)$) {};
      \node[uStyle] (yyy) at ($(uuu) + (-30:1)$) {};
      \node[uStyle] (xxx) at ($(yyy) + (30:1)$) {};  
      \node[uStyle] (zzz) at ($(uuu) + (-60:1)$) {};
      \node[uStyle] (AAA) at ($(uuu) + (-120:1)$) {};
      \node[uStyle] (BBB) at ($(AAA) + (-60:1)$) {};
      \draw (rrr) -- (sss) -- (ttt) -- (uuu) (rrr) -- (ttt)(uuu) -- (sss);
      \draw (uuu) -- (vvv) --  (yyy) -- (xxx) (uuu) -- (yyy) (xxx) -- (vvv);
      \draw (zzz) -- (BBB) -- (AAA) -- (uuu) -- (zzz) -- (AAA);
      \node[uStyle] (rrr) at (3.48,-14.8) {};
      \node[uStyle] (sss) at ($(rrr) + (60:1)$) {};
      \node[uStyle] (ttt) at ($(rrr) + (0:1)$) {};
      \node[uStyle] (uuu) at ($(rrr) + (-60:1)$) {};
      \node[uStyle] (vvv) at ($(ttt) + (15:1)$) {};
      \node[uStyle] (yyy) at ($(ttt) + (75:1)$) {};
      \node[uStyle] (xxx) at ($(yyy) + (15:1)$) {};  
      \node[uStyle] (zzz) at ($(ttt) + (-15:1)$) {};
      \node[uStyle] (AAA) at ($(ttt) + (-75:1)$) {};
      \node[uStyle] (BBB) at ($(AAA) + (-15:1)$) {};
      \draw (rrr) -- (sss) -- (ttt) -- (uuu) -- (rrr) -- (ttt);
      \draw (ttt) -- (vvv) --  (yyy) -- (xxx) (ttt) -- (yyy) (xxx) -- (vvv);
      \draw (zzz) -- (BBB) -- (AAA) -- (ttt) -- (zzz) -- (AAA);
      \node[uStyle] (rrr) at (11.3,-14.35) {};
      \node[uStyle] (sss) at ($(rrr) + (30:1)$) {};
      \node[uStyle] (ttt) at ($(rrr) + (-30:1)$) {};
      \node[uStyle] (uuu) at ($(sss) + (-30:1)$) {};
      \node[uStyle] (vvv) at ($(uuu) + (30:1)$) {};
      \node[uStyle] (yyy) at ($(uuu) + (-30:1)$) {};
      \node[uStyle] (xxx) at ($(yyy) + (30:1)$) {};  
      \node[uStyle] (zzz) at ($(uuu) + (-60:1)$) {};
      \node[uStyle] (AAA) at ($(uuu) + (-120:1)$) {};
      \node[uStyle] (BBB) at ($(AAA) + (-60:1)$) {};
      \draw (rrr) -- (sss)  (ttt) -- (uuu) (rrr) -- (ttt)(uuu) -- (sss) (rrr)--(uuu);
      \draw (uuu) -- (vvv)   (yyy) -- (xxx) (uuu) -- (yyy) (xxx) -- (vvv) (uuu) -- (xxx);
      \draw (zzz) -- (BBB) -- (AAA) -- (uuu) -- (zzz) (uuu) -- (BBB);
      \node[uStyle] (rrr) at (15.9,-14.35) {};
      \node[uStyle] (sss) at ($(rrr) + (30:1)$) {};
      \node[uStyle] (ttt) at ($(rrr) + (-30:1)$) {};
      \node[uStyle] (uuu) at ($(sss) + (-30:1)$) {};
      \node[uStyle] (vvv) at ($(uuu) + (30:1)$) {};
      \node[uStyle] (yyy) at ($(uuu) + (-30:1)$) {};
      \node[uStyle] (xxx) at ($(yyy) + (30:1)$) {};  
      \node[uStyle] (zzz) at ($(uuu) + (-60:1)$) {};
      \node[uStyle] (AAA) at ($(uuu) + (-120:1)$) {};
      \node[uStyle] (BBB) at ($(AAA) + (-60:1)$) {};
      \draw (rrr) -- (sss)  (ttt) -- (uuu) (rrr) -- (ttt)(uuu) -- (sss) (rrr)--(uuu);
      \draw (uuu) -- (vvv)   (yyy) -- (xxx) (uuu) -- (yyy) (xxx) -- (vvv) (uuu) -- (xxx);
      \draw (zzz) -- (BBB) -- (AAA) -- (uuu) -- (zzz) (AAA) -- (zzz);
      \node (ggg) at (5.5, -17) {\footnotesize{$\mathcal{A}_{10}$}};
  \end{scope}
  \begin{scope}
      \draw
     ($(groupC.south west)+(-7pt,-7pt)$) rectangle
     ($(groupC.north east)+(7pt,7pt)$);
  \end{scope}
\end{tikzpicture}
}
  \caption{The sets of graphs $\Acal_4,\Acal_6,\Acal_8$, and $\Acal_{10}$.}
    \label{fig_chor_2}
\end{figure}

\begin{lemma}
\label{lem_uokdhd}
For every $\Gcal\in\Abbb\cup\{\Dcal,\Ccal\}$, it holds that 
$\delta(\Gcal)=\frac{3}{2}$. Also $\delta(\Mcal)=2$.
\end{lemma}
\begin{proof}
Consider a class $\Gcal\in\mathbb A_{j}$  for some $j\in\{4,6,8,10\}$. For $0\le s\le j$,  we define
\[
f_4(s)=\binom{s}{2}\qquad\text{and}\qquad
f_6(s)=\begin{cases}0,&s \in \{0,1\},\\ 2s-3,&2\le s\le 6.\end{cases}
\] 
\[
f_8(s)=\begin{cases}0,&s \in \{0,1\},\\ 2s-3,&2\le s\le 5,\\2s-4,&6\le s\le 8\end{cases}\qquad\text{and}\qquad
f_{10 }(s)=
\begin{cases}
0, & s \in \{0,1\},\\
2s-3, & 2\le s\le 4,\\
2s-4, & 5\le s\le 7,\\
2s-5, & 8\le s\le 10
\end{cases}
\]
Observe that any edge-maximal graph in $\Gcal$ on $s$ vertices in $\Abbb_j$ has exactly $f_j(s)$ edges.
 
Also, observe that for any $j\in\{4,6,8,10\}$, $\frac{f_j(j)}{j}=\frac{3}{2}$ and   that $0\leq  f_j(j) \leq  15$. Let $G \in \Gcal_n$ for some $n \in \Nbbb$. Let $n_1,\dots,n_r$ be the sizes of the connected components of $G$ and observe that $1\le n_i \le j$, $i \in [r]$. Let $k = \lfloor n/j \rfloor$ and observe that
\[|E(G)|\le \sum_{i \in [r]} f_j(n_i) \le k\,f_j(j)+f_j(s)\] where $s = n \mod j$. Bounds are sharp, witnessed by $k$ disjoint copies of an edge maximal
$j$-vertex graph plus an edge maximal graph on the remaining $s$ vertices. Therefore, $\text{ex}(\Gcal,n) = k\,f_j(j)+f_j(s) $ and $\delta(\Gcal)
=\lim_{n\to\infty}\frac{\text{ex}(\Gcal,n)}{n}
=\lim_{n\to\infty}\frac{\,\lfloor n/j \rfloor f_j(j)+f_j(s)}{n}=\frac{f_j(j)}{j}=\frac{3}{2}
$. So if $\Gcal \in \Abbb$ we have that $\delta(\Gcal) = \frac{3}{2}$ .\\

Assume now that $\Gcal = \Dcal$ (resp. $\Gcal = \Ccal$). Then, for every $k \in \Nbbb,$ it holds that $\mathscr{D}_k \in \mathcal{D}$   (resp. $\mathscr{C}_k \in \mathcal{C}$). Let $G \in \mathcal{G}_n$ for some $n \in \Nbbb$, we claim that, in any case, 
\[
\text{ex}(\mathcal{G},n) = n-1 + \Big\lfloor \frac{n-1}{2} \Big\rfloor = 
\begin{cases}
\frac{3}{2}(n-1),& n\text{ odd},\\[2pt]
\frac{3}{2}n-2,& n\text{ even}.
\end{cases}
\]\\
In order to prove the above claim, observe that, for every nontrivial connected component $C$ of $G$, every block of $C$ is isomorphic to $K_i$ for some, $i\in\{2,3\}$. We call such a block a $K_i$-block of $G$. Let $t_3$ (resp. $t_{2}$) be the number of $K_3$-blocks (resp. $K_2$-blocks) of $G$, and let $t_1$ be the number of connected components of $G$. Now $n = 2t_3 + t_1 + t_2$ and $m = 3t_3 + t_2$ where $m$ and $n$ are the numbers of edges and vertices of $G$, respectively. So, $m = 3t_3 + t_2 = n-t_1 + t_3 \le n-1 + t_3 \le n-1 + \Big\lfloor \frac{n-1}{2} \Big\rfloor$ since $t_3 \le \Big\lfloor \frac{n-1}{2} \Big\rfloor$. Therefore, $\text{ex}(\Gcal,n) \leq  n-1 + \Big\lfloor \frac{n-1}{2} \Big\rfloor$.
For sharpness, if $n=2k+1$, take $\mathscr{D}_k$ (resp. $\mathscr{C}_k$)  which has $m=3k$. In case $n=2k$, contract an edge in $\mathscr{D}_k$ (resp. contract an edge in $\mathscr{C}_k$). This removes 2 edges so $m=3k-2$. In both cases, 
$m=n-1+\lfloor (n-1)/2\rfloor$. Therefore, $\text{ex}(\Gcal,n) \geq  n-1 + \Big\lfloor \frac{n-1}{2} \Big\rfloor$. This completes the proof of the claim.
We conclude that
\[
\delta(\Gcal)
=\lim_{n\to\infty}\frac{\text{ex}(\Gcal,n)}{n}
=\frac{3}{2}
\]
We now consider the remaining case where  $\Gcal = \Mcal$. Then, for every $k \in \Nbbb,$ it holds that $\mathscr{M}_{k} \in \Gcal$. As every graph $G \in \Gcal_n$, for $n \in \Nbbb_{\geq 2},$ is a subgraph of $\mathscr{M}_{n-2}$ we have that $\text{ex}(\Gcal,n) \leq  2n-3$. Clearly, this upper bound is sharp as witnessed by $\mathscr{M}_{n-2}$. Therefore, $\text{ex}(\Gcal,n) = 2n-3$, thus 
$$ \delta(\Gcal) =\lim_{n\to\infty }\frac{\text{ex}(\Gcal,n)}{n}=\lim_{n\to\infty}\frac{2n-3}{n}=2$$
as required.
\end{proof}

A \emph{cyclic} graph $G$ is a graph with at least one cycle. We say that an edge $xy$ of $G$ is
\emph{reducible}  if $N_{G}(x)\cap N_{G}(y)=\emptyset$.
If $G$ has no reducible edges then we call it   \emph{simplified}. We say that we \emph{simplify} a graph $G$ if we repetitively contract  reducible edges of $G$ until it  becomes simplified.
Notice that all graphs in $\Acal$ are simplified.

\begin{observation}
\label{dis_pen}
Let $G$ be a graph where $|E(G)|/|V(G)|\geq 1.$
If $G$ contains a reducible edge $e$, then $|E(G/e)|/|V(G/e)|\geq |E(G)|/|V(G)|$.  
\end{observation}

We define $S_{3}$ to be the graph formed from a triangle by adding for each edge $xy$ a new vertex  $v_{xy}$ and making it adjacent to both $x$ and $y$. Observe that $S_{3}\in \Acal_6$.

\begin{lemma}
\label{thi_oplod}
If $G$ is a simplified biconnected graph that minor-excludes every graph in $\Acal_4\cup \Acal_6$, then $G$ is some $2$-tree on $3$, $4$, or $5$ vertices (see~\cref{2-trees}).
\end{lemma}

\begin{proof}

We claim that every induced cycle $C$ of $G$ must be a triangle, i.e., $G$ is chordal. To see this, suppose $|V(C)|\ge4$. Let $v_1,v_2,\dots,v_n$ be the vertices of $C$ in cyclical order. Since $G$ is simplified and $C$ is induced, the endpoints of every edge $v_iv_{i+1}$ in $C$ must have at least one common neighbor $x_i$ outside $C$ ($i\in[n]$ with indices taken modulo $n$). If $x_i=x_j$ for some $i,j\in[n]$ with $i\neq j$, then contracting from $v_{i+1}$ to $v_j$ and from $v_{j+1}$ to $v_i$ gives a $K_4$ minor (see \cref{fig1}), a contradiction since $K_4\in\Acal_4$. Therefore, the $x_i$'s are distinct. But now contracting from $v_4$ to $v_1$ gives an $S_{3}$ minor (see \cref{fig2}), a contradiction since $S_{3}\in\Acal_6$. Thus, $|V(C)|=3$.  

\begin{figure}[!h]
     \centering
     \begin{subfigure}{0.45\textwidth}
         \centering
         \begin{tikzpicture}[bezier bounding box, scale=0.8, every node/.style={scale=0.8}]
          \draw[thick] (0,0) edge[bend left=10] (1,0.5) (1,0.5) edge[decorate, decoration={snake, amplitude=0.4mm}] (2,0.5) (2,0.5) edge[bend left=10] (3,0) (0,0) edge[decorate, decoration={snake, amplitude=0.4mm}, bend right=120, looseness=2] (3,0); 
          \draw[thick] (0,0) -- (1.5,1.5) -- (1,0.5) (2,0.5) -- (1.5,1.5) -- (3,0);
          %\draw[thick] (0,0) -- (0.15,1) -- (1,0.5) (2,0.5) -- (2.85,1) -- (3,0);
          \foreach \i/\j in {0/0,1/0.5,2/0.5,3/0}
          \draw[thick] (\i, \j) node[uStyle] {};
          \draw[thick] (1.5,-0.5) node {$C$} (1.5,1.5) node[uStyle] {} (0.2,-0.3) node {$v_i$} (1,0.1) node {$v_{i+1}$} (2,0.1) node {$v_j$} (2.7,-0.3) node {$v_{j+1}$} (1.5,1.8) node {$x_i$};
         \end{tikzpicture}
         \caption{}
         \label{fig1}
     \end{subfigure}\hspace{2mm}
     \begin{subfigure}{0.45\textwidth}
     \centering
     \begin{tikzpicture}[bezier bounding box, scale=0.8, every node/.style={scale=0.8}]
          \draw[thick] (0,0) edge[bend left=10] (1,0.5) (1,0.5) -- (2,0.5) (2,0.5) edge[bend left=10] (3,0) (0,0) edge[decorate, decoration={snake, amplitude=0.4mm}, bend right=120, looseness=2] (3,0); 
          \draw[thick] (0,0) -- (0.15,1) -- (1,0.5) (2,0.5) -- (2.85,1) -- (3,0) (1,0.5) -- (1.5,1.4) -- (2,0.5);
          \foreach \i/\j in {0/0,1/0.5,2/0.5,3/0}
          \draw[thick] (\i, \j) node[uStyle] {};
          \draw[thick] (1.5,-0.5) node {$C$} (1.5,1.4) node[uStyle] {} (0.15,1) node[uStyle] {} (2.85,1) node[uStyle] {} (0.2,-0.3) node {$v_1$} (1,0.1) node {$v_2$} (2,0.1) node {$v_3$} (2.85,-0.3) node {$v_4$} (0.15,1.3) node {$x_1$} (1.5,1.7) node {$x_2$} (2.85,1.3) node {$x_3$};
         \end{tikzpicture}
     \caption{}
    \label{fig2}
     \end{subfigure}\hspace{2mm}
     \caption{Cycle $C$ from~\cref{thi_oplod} where (a) two edges $v_iv_{i+1}$ and $v_jv_{j+1}$ have a common neighbor $x_i$, and where (b) each edge has a distinct neighbor outside $C$.}
        \label{figs}
\end{figure}
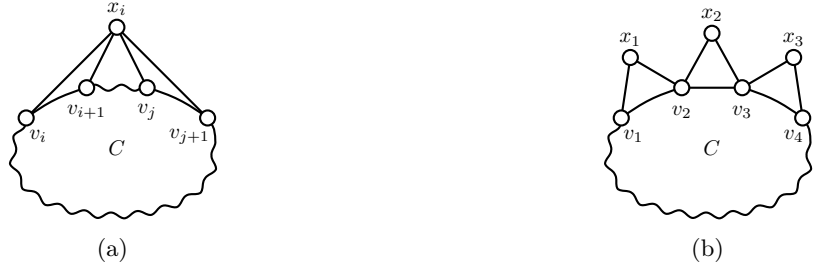

It is easy to prove that every chordal, $K_4$-minor-free, biconnected graph is a 2-tree. Moreover, since every 2-tree on $n\ge 6$ vertices contains a $6$-vertex 2-tree as a minor (repeatedly delete simplicial vertices until you have 6 vertices) and $G$ minor-excludes every $6$-vertex 2-tree, we have $|V(G)|\le5$. Thus, $G$ is a 2-tree on $3$, $4$, or $5$ vertices.
\end{proof}

A \emph{tree of triangles} is a connected graph whose blocks 
are all triangles.

\begin{lemma}
\label{tree_blokd_b}
If $G$ is a tree of triangles that minor-excludes $\mathscr{D}_{k}$ and $\mathscr{C}_{k}$ for some $k\in\Nbbb_{\geq 1}$, then $G$ has at most $k^2$ blocks.
\end{lemma}

\begin{proof}
Let $T$ be the block–cut tree of $G$ and let $B$ and $C$ be the vertices of $T$ that represent the blocks and cut vertices of $G$, respectively. If $T$ has at least $k$ leaves, then $G$ contains $\mathscr{D}_k$ as a minor. To see this, contract all edges not in leaf blocks in $G$ so that all their cut vertices merge into one vertex. Thus, $T$ has at most $\le k-1$ leaves. If $T$ has a path with $k$ vertices in $B$, then $G$ contains $\mathscr{C}_k$ as a minor. Thus, every path in $T$ has $\le k-1$ vertices in $B$. Note that a maximal path in $T$ starts and finishes on vertices of $B$ and alternates between vertices of $C$ and $B$. Therefore, the diameter of $T$ is at most $2k-4$. Since $T$ has $\le k-1$ leaves and diameter $\le 2k -4$, we have $|B|\le |V(T)|\le (k-1)(2k -4)/2+1\le k^2$,
by \cref{hko_take}.
\end{proof}

\begin{lemma}
\label{djri_osw}
If $G$ is a cyclic connected graph minor-excluding all graphs in $\Acal$
and also minor-excluding $\mathscr{D}_{k}$ and $\mathscr{C}_{k}$ for some $k\in\Nbbb_{\geq 2}$, then $1\leq \frac{|E(G)|}{|V(G)|}\leq \frac{3}{2}-\frac{1}{4k^2+14}$.
\end{lemma}

\begin{proof}

Note that $\frac{|E(G)|}{|V(G)|}\ge1$ since $G$ is cyclic. By \cref{dis_pen}, we may assume $G$ is simplified.
Let $B$ be a block of $G$. Since $G$ is simplified, $B$ is biconnected.
Furthermore, $B$ is some  $2$-tree on
$i$-vertices where $i\in\{3,4,5\}$, by \cref{thi_oplod}.  
We call these blocks $i$-blocks, $i\in[3,5]$. 

Observe that $G$ (i) cannot have more than one $5$-block,
(ii) cannot have more than two $4$-blocks, and (iii) cannot have a $5$-block and a $4$-block. Indeed, if $G$ does not satisfy (i) or (iii) (resp. (ii)), then it contains some graph in $\Acal_8$ (resp. $\Acal_{10}$) as a minor.
Let $b$ be the number of $3$-blocks in $G$.
Because of \cref{tree_blokd_b}, $b\leq k^2$.
We distinguish the following cases:

\noindent\emph{Case 1.} All blocks of $G$ are 3-blocks.   Then 
$|V(G)|=2b+1$ and  $|E(G)|=3b$, therefore $\frac{|E(G)|}{|V(G)|} = \frac{3b}{2b+1} = \frac{3}{2} - \frac{3}{4b+2} \leq  \frac{3}{2} - \frac{3}{4k^2+2}$.

\noindent\emph{Case 2.}
One block of $G$ is a $4$-block and the rest are $3$-blocks. Then 
$|V(G)|=2b+4$ and  $|E(G)|=3b+5$, therefore $\frac{|E(G)|}{|V(G)|} = \frac{3b+5}{2b+4} = \frac{3}{2} - \frac{1}{2b+4} \leq  \frac{3}{2} - \frac{1}{2k^2+4}$.

\noindent\emph{Case 3.} One block of $G$ is a $5$-block and the rest are $3$-blocks. Then 
$|V(G)|=2b+5$ and  $|E(G)|=3b+7$, therefore $\frac{|E(G)|}{|V(G)|} = \frac{3b+7}{2b+5} = \frac{3}{2} - \frac{1}{4b+10} \leq  \frac{3}{2} - \frac{1}{4k^2+10}$.

\noindent\emph{Case 4.} Two blocks of $G$ are $4$-blocks and the rest are $3$-blocks.  Then 
$|V(G)|=2b+7$ and  $|E(G)|=3b+10$, therefore $\frac{|E(G)|}{|V(G)|} = \frac{3b+10}{2b+7} = \frac{3}{2} - \frac{1}{4b+14} \leq  \frac{3}{2} - \frac{1}{4k^2+14}$.

In any case, we have that $\frac{|E(G)|}{|V(G)|}\leq \frac{3}{2}-\frac{1}{4k^2+14}$.
\end{proof}

\begin{lemma}
\label{do_econnec_do}
There is a function $f_{\ref{do_econnec_do}}:\Nbbb^2\to\Nbbb$ such that
for every $k\in\Nbbb_{\geq 1}$, every connected  graph $G$, and  every $S\subseteq V(G)$,
if   
\begin{itemize}
\item[(i)] $G$ excludes both $\mathscr{M}_{k}$ and $\mathscr{D}_{k}$ as a minor and 
\item[(ii)]  for every connected component $C$ of $G-S$, either  $C$ contains a cycle or there are at least two edges between $S$ and $V(C)$,
\end{itemize}
then  $G-S$ has at most $f_{\ref{do_econnec_do}}(|S|,k)$ connected components.
\end{lemma}
\begin{proof}
    Let $G$ be a connected graph, $S\subseteq V(G)$, and $k\in\Nbbb_{\geq 1}$. Observe that every component of $G-S$ has a neighbor in $S$ since $G$ is connected.

    \begin{figure}[!h]
        \centering
        \begin{tikzpicture}[scale=0.8, every node/.style={scale=0.8}]
        \begin{scope}
        \draw[thick] (0,1) node[uStyle] {} (0,1.3) node {$x$} (0,1) ellipse (2cm and 0.5cm) (0,1.8) node {$S$} (0,-2) ellipse (1.2cm and 1cm) (-0.55,-2) edge[decorate, decoration={snake, amplitude=0.4mm}] (0.55,-2) (-0.55,-2) -- (0,1) -- (0.55,-2) (-0.55,-2) node[uStyle] {} (0.55,-2) node[uStyle] {} (-0.55,-2.3) node {$u$} (0.55,-2.3) node {$v$} (0,-3.3) node {$C$};  
        \end{scope}

        \begin{scope}[xshift=8cm]
        \draw[thick] (0,1) node[uStyle] {} (0,1.3) node {$x$} (0,1) ellipse (2cm and 0.5cm) (0,1.8) node {$S$} (0,-2) ellipse (1.2cm and 1cm) (-0.6,-1.9) edge[bend left=45] (0,-1.5) (0,-1.5) edge[bend left=45] (0.6,-1.9) (-0.6,-1.9) edge[decorate, decoration={snake, amplitude=0.4mm}, in=-90, out=-90, looseness=1.5] (0.6,-2) (0,1) -- (0,-1.5) (0,-1.5) node[uStyle] {} (-0.6,-1.9) node[uStyle] {} (0.6,-1.9) node[uStyle] {} (-0.9,-1.9) node {$v$} (0.9,-1.9) node {$w$} (0.2,-1.25) node {$u$} (0,-2) node {$Y$} (0,-3.3) node {$C$};  
        \end{scope}
        \end{tikzpicture}
        \caption{Component $C\in\Qcal_x$ from \cref{do_econnec_do} where $C$ is a tree (left), and where $C$ contains a cycle $Y$ (right).}
        \label{fig-function1}
    \end{figure}
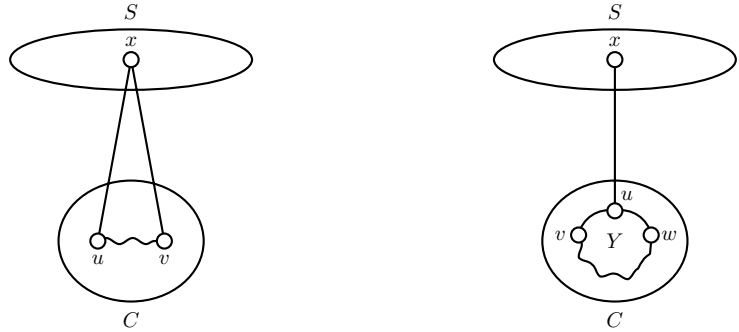

    Fix $x \in S$ and let $\Qcal_x$ denote the set of components of $G-S$ that only receive edges from $x$ in $S$. Observe that, for every $C\in\Qcal_x$, the graph induced by $C\cup\{x\}$ contains $K_3$ as a minor with $x\in V(K_3)$. Indeed, if $C$ is a tree, then there exist distinct $u,v \in V(C)$ such that $xu,xv\in E(G)$, by (ii). Contracting the path between $u$ and $v$ in $C$ into a single edge $uv$ gives a $K_3$ induced by $\{x,v,u\}$; see the left of~\cref{fig-function1}. If instead $C$ contains a cycle $Y$, then there exist $u$ on $Y$ adjacent to $x$, by the definition of $\Qcal_x$.   
    Let $v,w$ be the two consecutive neighbors of $u$ on $Y$. Contracting the subpath of $Y$ avoiding $x$ from $v$ to $w$ into a single edge $vw$ and contracting the edge $xu$ gives a $K_3$ induced by $\{x,v,w\}$; see the right of~\cref{fig-function1}. Thus, every component $C$ of $\Qcal_x$ yields a $K_3$ whose only vertex in $S$ is $x$. If $|\Qcal_x|\ge k$, then this collection of triangles implies $G$ contains $\mathscr{D}_k$ as a minor, a contradiction. Thus, $|\Qcal_x|<k$. 
    
    Now fix $x,y \in S$ and let $\Hcal_{x,y}$ denote the set containing every component $C$ of $G-S$ such that $\{x,y\} \subseteq N_{S}(V(C))$ (i.e. those are the components of $G-S$ which receive edges from at least two vertices, $x$ and $y$). Form $G'$ from $G$ by contracting every $C \in \Hcal_{x,y}$ into a single vertex. If $|\Hcal_{x,y}|\ge k+1$, then $G'$, and therefore $G$, contains $\mathscr{M}_{k}$ as a minor, a contradiction. Hence, $|\Hcal_{x,y}| < k+1$. 
    
    Recall that for every component $C$ of $G-S$ there exist $x,y\in S$ such that $C$ lies either in $\Qcal_x$ or in $\Hcal_{x,y}$. Thus, the total number of components is at most 
    \[
      k|S| + k\binom{|S|}{2} = k\binom{|S|+1}{2}
    \]
    Therefore, the lemma holds if we set $f_{\ref{do_econnec_do}}(|S|,k)\coloneqq k\binom{|S|+1}{2}.$ 
\end{proof}
Using \cref{msain_riop} we prove the following.
\begin{lemma}
\label{cosxl}
There is a function $f_{\ref{cosxl}}:\Nbbb^2\to\Nbbb$ such that 
for every $k\in\Nbbb_{\geq 1}$, every connected graph $G$, and  every $S\subseteq V(G)$,
if   
\begin{itemize}
\item[(i)] $G$  excludes both $\mathscr{M}_{k}$ and $\mathscr{D}_{k}$ as a minor and  
\item[(ii)] $G-S$ is connected,
\end{itemize}
then $G$ contains  at most  $f_{\ref{cosxl}}(|S|,k)$ edges  between $S$ and $V(G-S)$.
\end{lemma}

\begin{proof}
 
Let $G$ be a connected graph, $S\subseteq V(G)$, and $k\in\mathbb N_{\ge1}$. Fix $x\in S$ and let $X=:|N_{G}(x)\cap V(G-S)|$. Since $G-S$ is connected by (ii), if $|X|>(2k)^2$, then the connected graph $G-S$
contains as an $X$-minor either
$K_{1,2k}$ or $P_{2k}$, by \cref{msain_riop}. This implies that $G$ contains either $\mathscr{M}_{k}$ or $\mathscr{D}_{k}$ as a minor (since all vertices of $X$ are adjacent to $x$), contradicting (i). Thus, $|X|\le (2k)^2$.

We conclude that $G$ contains at most 
$
f_{\ref{cosxl}}(|S|,k) \coloneqq  |S|\cdot (2k)^2
$
edges between $S$ and $V(G-S)$, as required.
\end{proof}

We are now ready to prove the main result of this section.  

\begin{lemma}
\label{theo_asdvre}
If $\Gcal$ is a graph class that does not contain any class in $\Abbb$
as a subset and excludes the graphs $\mathscr{D}_k$, $\mathscr{C}_k$ and $\mathscr{M}_k$ as minors, then $\delta(\Gcal)\le \frac{3}{2}-\frac{1}{4k^2+14}$.
\end{lemma}

\begin{proof}
 Let $\Gcal$ be a minor-closed graph class for which there exists $k\in\Nbbb_{\ge1}$ 
such that every graph  $G\in\Gcal$ excludes $\mathscr{D}_{k}$, $\mathscr{C}_{k}$, $\mathscr{M}_{k}$,
and also excludes $\mathscr{P}^{A}_{k}$, for every $A\in \Acal$.
Note that, for every graph $A\in\Acal$, every graph $G\in \Gcal$ 
has at most $k-1$ connected components 
containing   $A$ as a minor.

Fix $G\in \Gcal$.
Let $\Ccal_{1}$ be the connected components of $G$
that contain some graph in $\Acal$ as a minor and let $\Ccal_{0}$
be those that minor-exclude all graphs in $\Acal$.
As we already observed $|\Ccal_{1}|\le |\Acal|\cdot (k-1)=30\cdot (k-1)$.

Let $C\in\Ccal_{1}$.
Observe that for every graph $A\in \Acal$ it holds that 
$C$ excludes  $k\cdot A$ as a minor.  
As $\Acal$ contains some planar graph, by \cref{erpos}, there is some $c_{2}$, depending on $\Acal$, and some 
set $S_C\subseteq V(C)$ where $|S_C|\leq c_{2}$
such that $J_C\coloneqq C-S_{C}$ minor-excludes all 
graphs in $\Acal$.

We say that a connected component $Z$ of $J_{C}$ is \emph{$C$-poor} if it is a tree and  
there is only one edge between $S_C$ and the vertices of $Z$. We also say that $Z$ is \emph{$C$-rich} 
if it not poor. We denote by $\Qcal_C^{0}$  and $\Qcal_C^{1}$ the set of $C$-poor and $C$-rich connected components of $J_C$, respectively; see \cref{fig-longlem}. 
Observe that the  graph $G[S\cup V(\bigcup\Qcal_{C}^{1})]$
is connected, therefore we may apply  to it  \cref{do_econnec_do}
and  obtain that $|\Qcal_C^1|\leq  f_{\ref{do_econnec_do}}(c_{2},k)$

We now partition the edges of $C$ as follows:
Let $E_{C}^{\mathsf{p}}$ be the edges of $C$ with at least one endpoint in a  $C$-poor component, let $E_{C}^{\mathsf{r}}$  be the edges  of the  rich components and let $E_{C}^{\mathsf{s}}$ be the edges  of $C$ 
with one endpoint in $S$ and the other in either $S$ or in some vertex of a $C$-rich component. Notice that $\{E_{C}^{\mathsf{p}},E_{C}^{\mathsf{r}},E_{C}^{\mathsf{s}}\}$ is a  partition of $E(C)$; see \cref{fig-longlem}.
We also define $m_{C}^{\mathsf{p}}=|E_{C}^{\mathsf{p}}|$, $m_{C}^{\mathsf{r}}=|E_{C}^{\mathsf{r}}|,$ and 
$m_{C}^{\mathsf{s}}=|E_{C}^{\mathsf{s}}|$. Also we set 
$n^{\mathsf{s}}_{C}=|S_{C}|$ and let $n_{C}^{\mathsf{r}}$ (resp. $n_{C}^{\mathsf{p}}$) be   the number of vertices 
of the $C$-rich ($C$-poor) components of $J_{C}$.

\begin{figure}[!h]
    \centering
\scalebox{.78}{
    \begin{tikzpicture}[scale=0.85, every node/.style={scale=0.6}]
    \begin{scope}
    \draw[thick] (-2,1.5) rectangle (2,-3.5) (0,1.8) node {\huge{$\Ccal_0$}}; 
    \draw[thick, myblue] (0,0) ellipse (0.5cm and 0.5cm) (1,-2) ellipse (0.5cm and 0.5cm) (-1,-2) ellipse (0.5cm and 0.5cm); 
    \end{scope}
    \begin{scope}[xshift=4cm]
    \draw[thick] (4,2.5) edge[dotted] (0.5,0) (0.5,0) edge[dotted] (4,-4.5) (-2,1.5) rectangle (2,-3.5) (0,0) ellipse (0.5cm and 0.5cm) (1,-2) ellipse (0.5cm and 0.5cm) (-1,-2) ellipse (0.5cm and 0.5cm) (0,0.8) node {\huge{$C$}} (0,1.8) node {\huge{$\Ccal_1$}};
    \draw[thick, myothergreen] (-0.25,0) ellipse (0.1cm and 0.3cm) (-1.25,-2) ellipse (0.1cm and 0.3cm) (0.75,-2) ellipse (0.1cm and 0.3cm);
    \draw[thick, myred] (0.15,0.2) ellipse (0.2cm and 0.1cm) (-0.85,-1.8) ellipse (0.2cm and 0.1cm) (1.15,-1.8) ellipse (0.2cm and 0.1cm);
    \draw[thick, myblue] (0.15,-0.2) ellipse (0.2cm and 0.1cm) (-0.85,-2.2) ellipse (0.2cm and 0.1cm) (1.15,-2.2) ellipse (0.2cm and 0.1cm);
    \end{scope}
    \begin{scope}[xshift=9cm]
    \draw[thick] (-1,2.5) rectangle (6,-4.5);
    \draw[thick, decorate, decoration={calligraphic brace, amplitude=3mm}] (2.3,1.2) -- (5.7,1.2);
    \draw[thick, decorate, decoration={calligraphic brace, mirror, amplitude=3mm}] (2.3,-3.3) -- (5.7,-3.3);
    \draw[thick, myothergreen] (0,-1) ellipse (0.5cm and 2cm) (0,1.3) node {\huge{$S_C$}};
    \draw[thick, myred] (3,0.5) circle (0.7cm) (5,0.5) circle (0.7cm) (4,1.9) node {\huge{$\Qcal^0_C$}};
    \draw[thick, myblue] (3,-2.5) circle (0.7cm) (5,-2.5) circle (0.7cm) (4,-4) node {\huge{$\Qcal^1_C$}};  
    \draw[thick] (0,0) edge[myothergreen] (0,-1) (0,-1) edge[myothergreen] (0,-2) (3,1) edge[myred] (3,0.2) (3,1) edge[myred] (2.7,0.2) (3,1) edge[myred] (3.3,0.2) (3,1) edge[myred] (0,0)  (5,0.8) edge[myred] (5,0.2) (5,0.2) edge[myred, bend left=20] (0,-1) (2.7,-2.5) edge[myblue] (3.3,-2.5) (2.7,-2.5) edge[myothergreen] (0,-2) (3.3,-2.5) edge[myothergreen] (0,-1) (4.7,-2.8) edge[myblue] (5.3,-2.8) (5.3,-2.8) edge[myblue] (5,-2.2) (5,-2.2) edge[myblue] (4.7,-2.8) (5,-2.2) edge[myothergreen, bend right=5] (0,-1);
    \foreach \i/\j/\c in {0/0/myothergreen, 0/-1/myothergreen, 0/-2/myothergreen, 3/1/myred, 3/0.2/myred, 3.3/0.2/myred, 2.7/0.2/myred, 5/0.8/myred, 5/0.2/myred, 2.7/-2.5/myblue, 3.3/-2.5/myblue, 4.7/-2.8/myblue, 5.3/-2.8/myblue, 5/-2.2/myblue}
    \draw[thick] (\i,\j) node[uStyle, draw=\c] {};
    \end{scope}
    \end{tikzpicture}
}
    \caption{An example graph $G$ from \cref{theo_asdvre}: Components of $G$ are split into $\Ccal_0$ and $\Ccal_1$. An enlargement of a component $C\in\Ccal_1$ is shown to the right with examples of $S_C$ (green), $\Qcal^0_C$ (red), and $\Qcal^1_C$ (blue). Edges $E^\mathsf{p}_C$, $E^\mathsf{r}_C$, and $E^\mathsf{s}_C$ are also shown in red, blue, and green, respectively. All blue components combined form $Q$.}
    \label{fig-longlem}
\end{figure}
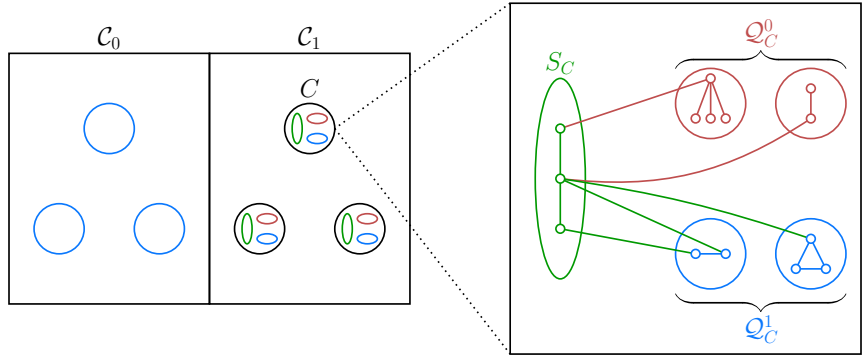

Observe that, given that there are $q$  $C$-poor components,
the edges in $E_{C}^{\mathsf{p}}$ are $n^{\mathsf{p}}_{C}-q$
plus  $q$ extra edges, one from each 
$C$-poor  component to a vertex of $S_{C}$. Therefore,
\begin{eqnarray}
m_{C}^{\mathsf{p}} & = & n_{C}^{\mathsf{p}}.\label{eq_esq_oplos}
\end{eqnarray}

    \begin{claim}\label{cl_diagGrid}
       $|m_{C}^{\mathsf{s}}|\leq \binom{c_{2}}{2}+f_{\ref{do_econnec_do}}(c_{2},k)\cdot f_{\ref{cosxl}}(c_{2},k)$.
    \end{claim}

\begin{claimproof} 
Let $Z\in \Qcal_C^{1}$.
Notice that, from \cref{cosxl}, there are at most $f_{\ref{cosxl}}(c_{2},k)$
edges in the connected graph $C$ with one endpoint in $S_{C}$ and the other in $V(Z)$. 
The claim follows as $|\Qcal_C^{1}|\leq f_{\ref{do_econnec_do}}(c_{2},k)$
and there are at most $\binom{c_{2}}{2}$ edges with both endpoints in $S_C$.
\end{claimproof}

Let $Q_{C}=\bigcup \Qcal_{C}^{1}$, i.e., $Q_{C}$ is the union of all $C$-rich components of $J_{C}$.  
Also we set $$Q=\bigcup \Ccal_{0}\cup\bigcup_{Z\in\Ccal_{1}}Q_Z$$  
and observe that every connected  component of $Q$ excludes  every  graph in $\Acal$ as a minor. \medskip

\begin{claim} \label{cl_asverga}
 $\frac{|E(Q)|}{|V(Q)|}\leq \frac{3}{2}-\frac{1}{4k^2+14}$.
\end{claim}

\begin{claimproof} 
The claim is obvious if $Q$ is acyclic because $\frac{|E(Q)|}{|V(Q)|}\leq 1<\frac{3}{2}-\frac{1}{4k^2+14}$. Therefore we may assume that $Q$ is cyclic.
We construct the graph $M$ as follows:
we start by setting $M\coloneqq Q$ and we  delete from it all the components that are trees.  Next we simplify  every cyclic component  and among the  components remove every simplified component that does not  attain the maximum density. Then we observe that $\frac{|E(Q)|}{|V(Q)|}\leq \frac{|E(M)|}{|V(M)|}$. Then by applying \cref{djri_osw} to each component of $M$ we have that $\frac{|E(M)|}{|V(M)|}\leq \frac{3}{2}-\frac{1}{4k^2+14}$. We conclude that $\frac{|E(Q)|}{|V(Q)|}\leq \frac{|E(M)|}{|V(M)|}\leq \frac{3}{2}-\frac{1}{4k^2+14}$, as required.\medskip 
\end{claimproof}

\medskip

We now define $n^\mathsf{s}=\sum_{C\in\Ccal_{1}}n_{C}^{\mathsf{s}}$,
$n^\mathsf{r}=\sum_{C\in \Ccal_{0}}|V(C)|+\sum_{C\in\Ccal_{1}}n_{C}^{\mathsf{r}}$, and 
$n^\mathsf{p}=\sum_{C\in\Ccal_{1}}n_{C}^{\mathsf{p}}$. Observe that 
\begin{eqnarray}
n^{\mathsf{s}} & \leq  & {30\cdot (k-1)\cdot c_{2}}. \label{eq_djdkywkhgcd}
\end{eqnarray}

Similarly, we set 
$m^\mathsf{s}=\sum_{C\in\Ccal_{1}}m_{C}^{\mathsf{s}}$,
$m^\mathsf{r}=\sum_{C\in \Ccal_{0}}|E(C)|+\sum_{C\in\Ccal_{1}}m_{C}^{\mathsf{r}}$, and $m^\mathsf{p}=\sum_{C\in\Ccal_{1}}m_{C}^{\mathsf{p}}$.
Let $c\coloneqq k\cdot (\binom{c_{2}}{2}+f_{\ref{do_econnec_do}}(c_{2},k)\cdot f_{\ref{cosxl}}(c_{2},k))$.
Observe that from Claim~\ref{cl_diagGrid} and from Claim~\ref{cl_asverga} we obtain the following inequalities 
\begin{eqnarray}
 m^s & \leq & c\label{eq_fkduhs}\\
 m^r&\leq & (\frac{3}{2}-\frac{1}{4k^2+14})n^{r}\label{eq_fksdsuhs}
\end{eqnarray}

\allowdisplaybreaks
Let $G$ be a graph in $\Gcal$ on $n$ vertices and $m$ edges.
Observe that
\begin{eqnarray*}
m/n & =& \frac{m^{\mathsf{s}}+m^{\mathsf{r}}+m^{\mathsf{p}}}{n^{\mathsf{s}}+n^{\mathsf{r}}+n^{\mathsf{p}}}\\
& =^{\eqref{eq_esq_oplos}} & 
\frac{m^{\mathsf{s}}+m^{\mathsf{r}}+n^{\mathsf{p}}}{n^{\mathsf{s}}+n^{\mathsf{r}}+n^{\mathsf{p}}}\\
&\leq ^{\eqref{eq_fkduhs}}&
\frac{c+m^{\mathsf{r}}+n^{\mathsf{p}}}{n^{\mathsf{r}}+n^{\mathsf{p}}}\\
&=&
\frac{c+m^{\mathsf{r}}+n^{\mathsf{p}}}{n-n^{\mathsf{s}}}\\
&\leq ^\eqref{eq_djdkywkhgcd}&
\frac{c+m^{\mathsf{r}}+n^{\mathsf{p}}}{n-30\cdot (k-1)\cdot c_{2}}\\
&\leq ^{\eqref{eq_fksdsuhs}}&
\frac{c+(\frac{3}{2}-\frac{1}{4k^2+14})n^{\mathsf{r}}+n^{\mathsf{p}}}{n-30\cdot (k-1)\cdot c_{2}}\\
&=&
\frac{c+\frac{1}{2}n^{\mathsf{r}}+n^{\mathsf{r}}+n^{\mathsf{p}}-\frac{1}{4k^2+14} n^{\mathsf{r}}}{n-30\cdot (k-1)\cdot c_{2}}\\
&=&
\frac{c+(\frac{1}{2}-\frac{1}{4k^2+14})n^{\mathsf{r}}+n-n^{\mathsf{s}}}{n-30\cdot (k-1)\cdot c_{2}}\\
&\leq &
\frac{c+(\frac{1}{2}-\frac{1}{4k^2+14})n+n}{n-30\cdot (k-1)\cdot c_{2}}\\
&=& 
\frac{c+(\frac{3}{2}-\frac{1}{4k^2+14})n}{n-30\cdot (k-1)\cdot c_{2}}.
\end{eqnarray*}
Therefore, $\frac{\ex(\Gcal,n)}{n}\leq \frac{c+(\frac{3}{2}-\frac{1}{4k^2+14})n}{n-30\cdot (k-1)\cdot c_{2}}$ which implies that $\delta(\Gcal)\leq \frac{3}{2}-\frac{1}{4k^2+14}$.

\end{proof}
\begin{theorem}
    $\cobs(\Cbbb_{<\nicefrac{3}{2}})=\Abbb\cup\{\Dcal,\Ccal,\Mcal\}$.
    \label{theo_abkvaabo}
\end{theorem}
\begin{proof}
    
Notice first that none of the classes $\mathcal H \in \Abbb \cup\{\mathcal D,\mathcal M,\mathcal C\}$  is a subset of the other. Therefore, they constitute a $\subseteq$-anti-chain. Let $\Gcal\in\Cbbb_{<\nicefrac{3}{2}}$. By \cref{lem_uokdhd}, if $\mathcal G$ contains some $\mathcal H \in \Abbb \cup\{\mathcal D,\mathcal M,\mathcal C\}$ as a subset, then $\delta(\Hcal) \in \{\nicefrac{3}{2},2\}$, and therefore $\delta(\mathcal{G}) \ge \delta(\mathcal H)\geq \nicefrac{3}{2}$. So, $\Gcal$ does not contain any of $\Abbb,\Dcal,\Ccal,\Mcal$ as subsets. Now by \cref{theo_asdvre},  $\delta(\Gcal)\leq \frac{3}{2}-\frac{1}{4k^2+14}<\frac{3}{2}$.
\end{proof}

\subsection{Class obstructions for $\Cbbb_{\leq \delta}$ with $\delta\in[0,1)$}
\label{subsec_less1}

We now proceed with the identification of all obstructions for  $\Cbbb_{\leq \delta}$ in the case where $\delta\in[0,1)$.

We make use of the following trivial fact.

\begin{observation}\label{obsvarious}
Assume $a,a',b,b'\in \mathbb{R}_{>0}$, $a<b$, and $a'<b'$. If $\frac{a}{b}\leq\delta$ and $\frac{a'}{b'}\leq\delta$, then $\frac{a+a'}{b+b'}\leq\delta$.      
\end{observation}

\paragraph{The construction.}

Recall the following definitions: $\mathscr{P}^{H}=\langle \mathscr{P}^{H}_{k}\rangle_{k\in \Nbbb}$ and $\Gcal^H = \mathscr{P}^H \!\downarrow$ where $\mathscr{P}^{H}_{k}=k\cdot H$,
 
$\mathscr{S}=\langle K_{1,k}\rangle_{k\in\Nbbb}$, $\Gcal^{1}=\mathscr{P}^{K_{3}}    \!\downarrow$, and $\Gcal^{2}=\mathscr{S}  \!\downarrow$. Given some $a\in\Nbbb_{\ge1}$, we denote by $\Tcal_{a}$ the set of all trees on $a$ vertices.

For the statement of the next theorem, we need  a ``second-order'' version of ${\sf min}$ to apply to minor-closed graph classes. In particular, 
given a finite set $\Lbbb$ of minor-closed graph classes,  ${\sf min}(\Lbbb)$
is the subset of $\Lbbb$ that contains the $\subseteq$-minimal elements of $\Lbbb$.

\begin{theorem}
\label{hio8io3e}
For every $\delta\in[0,1)$, it holds that $\cobs(\Cbbb_{\leq \delta})={\sf min}\big(\{\Gcal^{1},\Gcal^{2}\}\cup \{\Gcal^{T}\mid T\in\Tcal_{x+2}\}\big)$ where $x=\lfloor\frac{\delta}{1-\delta}\rfloor$.
\end{theorem}

\begin{proof}
Let $\Gcal$ be a minor-closed graph class. If $\Gcal^{1}\subseteq \Gcal$ or $\Gcal^{2}\subseteq \Gcal$, then by \cref{jvdkjhjkdfkd}, it follows that $\delta(\Gcal)\geq 1>\delta.$ Recall that $\frac{x}{x+1}\leq \delta$, by the definition of $x$.

Now let $T\in \Tcal_{x+2}$ and assume that $\Gcal^{T}\subseteq \Gcal$.
For every $n\in\Nbbb$, let $k:=\lfloor n/(x+2)\rfloor$
and define $G_{n}$ to be the disjoint union of 
 $k\cdot T$ and $n   \!\mod (x+2)$ isolated vertices.
 Notice that $G_{n}\in \Gcal^{T}_n\subseteq \Gcal_n$.
 Moreover, $|E(G_{n})|=(x+1)\cdot \lfloor n/(x+2)\rfloor$. Therefore, $\frac{{\sf ex}(\Gcal,n)}{n}\geq \frac{(x+1)\cdot \lfloor n/(x+2)\rfloor}{n},$ which implies that  $$\delta(\Gcal)=\lim_{n\to\infty}\frac{{\sf ex}(\Gcal,n)}{n}\geq \lim_{n\to\infty}\frac{(x+1)\cdot \lfloor n/(x+2)\rfloor}{n}=\frac{x+1}{x+2}>\delta.$$
 To verify the last inequality,
observe that $x=\lfloor\frac{\delta}{1-\delta}\rfloor>\frac{\delta}{1-\delta}-1=\frac{2\delta-1}{1-\delta}$. Because $\delta<1$, this implies $x(1-\delta)>2\delta-1\Rightarrow x-\delta x>2\delta-1\Rightarrow x+1>\delta(2+x)\Rightarrow \frac{x+1}{x+2}>\delta$. 
Thus, if $\Gcal$ contains any of the classes 
in $\{\Gcal^{1},\Gcal^{2}\}\cup \{\Gcal^{T}\mid T\in\Tcal_{x+2}\}$ as a subset, 
then $\Gcal\not\in\Cbbb_{\leq \delta}$.
\medskip

Now assume that $\Gcal$ does not contain any of the classes 
in ${\sf min}\big(\{\Gcal^{1},\Gcal^{2}\}\cup \{\Gcal^{T}\mid T\in\Tcal_{x+2}\}\big)$ as a subset.
This means that there is some $c$ such that  $\Gcal$ minor-excludes $K_{1,c}$, $c\cdot K_{3}$,  and every $c\cdot T$ for all $T\in \Tcal_{x+2}$.  
Let $G\in \Gcal$. Since $K_{3}$ is  planar, there is a set $S\subseteq V(G)$ where $|S|$ is bounded by some function of $c$ and $x$, and $G-S$  excludes $K_{3}$ 
and every tree in $\Tcal_{x+2}$ as minors, by \cref{erpos}. Let $G':=G-S$. Since $G'$ minor-excludes $K_3$,
it follows that $G'$ is a forest. Furthermore, since $G'$ minor-excludes
all trees of size $x+2$, each connected component of $G'$ is a tree $Q$ where $|V(Q)|\leq x+1$.
This implies that $\frac{|E(Q)|}{|V(Q)|}=\frac{|V(Q)|-1}{|V(Q)|}\leq \frac{x}{x+1}\leq \delta$.  By \cref{obsvarious}, $G'$ also has density at most $\delta$. 
Finally, since $G$ minor-excludes $K_{1,c}$, the maximum degree
of $G$ is upper bounded by $c$. Therefore, if $F$ is the set of edges incident to the vertices of $S$,
then $|F|\leq c|S|$.

Let $G\in \Gcal_{n}$ where $|E(G)|={\sf ex}(\Gcal,n)$. 
Observe that $n=|S|+|V(G')|$ and $|E(G)|=|F|+|E(G[S])|+|E(G')|\leq c|S|+{|S|\choose 2}+|E(G')|$. 
Therefore, 
\begin{eqnarray*}
\frac{{\sf ex}(\Gcal,n
)}{n} & = & \frac{|E(G)|}{n}\leq \frac{c|S|+{|S|\choose 2}+|E(G')|}{|S|+|V(G')|}\\
& \leq  &  \frac{c|S|+{|S|\choose 2}+|E(G')|}{|V(G')|} \\
& =  & \frac{c|S|+{|S|\choose 2}}{|V(G')|}+\frac{|E(G')|}{|V(G')|}\\
& \leq  & \frac{c|S|+{|S|\choose 2}}{n-|S|}+\delta.
\end{eqnarray*}

Recall that $|S|$ is bounded by some constant depending on $c$ and $x$. Equivalently, there is an integer bounding 
$|S|$ that depends on the class $\Gcal$ and the choice of $\delta$. Let $c'$ be the smallest integer such that 
$c|S|+{|S|\choose 2}\leq c'$. This implies that $|S|\leq c'$.
Now 
$\delta(\Gcal)=\lim_{n\to\infty}\frac{{\sf ex}(\Gcal,n
)}{n}\leq \lim_{n\to\infty}(\frac{c'}{n-c'}+\delta)=\delta$; thus, $\Gcal\in\Cbbb_{\leq \delta}$.
\end{proof}

\subsection{Class obstructions for $\Cbbb_{\leq \delta}$ with $\delta\in[1,\nicefrac{3}{2})$}
\label{subsec_1_to_3_2}

We conclude this section  with the identification of all obstructions for  $\Cbbb_{\leq \delta}$ in the case where $\delta\in[1,\nicefrac{3}{2})$.
For this, we define a series of 
graph classes, introduced 
by Eppstein in\cite{Eppstein10Densities}, where he proved that their densities capture \textsl{exactly} the limit-densities in $[1,\nicefrac{3}{2})$. These classes are an important ingredient of our proofs.
\\

Given some $b\in\mathbb N$, we define the sets of graphs below.

\begin{itemize}
  \item $\Ecal^1_b$ is the set of connected graphs with exactly $b$ blocks, each block being a triangle (in this case $b\geq 1$).

  \item $\Ecal^2_b$ is the set of connected graphs with exactly $b+1$ blocks, where $b$ blocks are triangles and one is the unique $2$-tree on $4$ vertices.

  \item $\Ecal^3_b$: is the set of connected graphs with exactly $b+1$ blocks, where $b$ blocks are triangles and one is a $2$-tree on $5$ vertices.

  \item $\Ecal^4_b$: is the set of connected graphs with exactly $b+2$ blocks, where $b$ blocks are triangles and two are the unique $2$-tree on $4$ vertices.
  
\end{itemize}

For each $i\in\{1,2,3,4\},$ set
\[
(a_i,c_i)=
\begin{cases}
(0,1), & i=1,\\
(5,4), & i=2,\\
(7,5), & i=3,\\
(10,7),& i=4,
\end{cases}
\qquad
b_i=
\begin{cases}
1, & i=1,\\
0, & i=2,3,4.
\end{cases}
\]

Now for any integer $b\ge b_i$ and any $G\in \Ecal^i_b$ the density of $G$ is
\begin{eqnarray}
  \frac{|E(G)|}{|V(G)|}& = & \frac{3b+a_i}{2b+c_i}. \label{fig_density}  
\end{eqnarray}

Let $\delta$ be a limiting density in $[1,\nicefrac{3}{2})$. For each $i$, we define
\[
\beta_i(\delta)=\max\!\left\{\, b_i,\;
1+\left\lfloor \frac{\delta\,c_i-a_i}{\,3-2\delta\,}\right\rfloor \right\}.
\]
In other words, $\beta_i(\delta)$ is 
 the smallest integer $b$, such that every graph in $\Ecal^i_b$ has density strictly greater than $\delta$.
We set $\mathcal  U_{\delta} = \bigcup_{i \in [4]}\Ecal^i_{\beta_i(\delta)}$ and $\Ecal_{\delta} = \mathsf{min}(\Ucal_{\delta})$. Observe that $\beta_1(\delta)\geq 2$ for every limiting density $\delta \in [1,\nicefrac{3}{2})$.

\begin{observation}
\label{obs_stro23}
If $G$ is a graph with $\frac{|E(G)|}{|V(G)|}<\frac{3}{2}$ containing a leaf block $B$ that is a triangle, then removing all vertices of $B$ of degree $2$ results in a graph with density strictly smaller than that of $G$.
\end{observation}

\begin{lemma}
\label{lem_bjibjir}
Let $\delta\in [1,\nicefrac{3}{2})$. 
    If $G$ is a cyclic connected graph minor-excluding all graphs in $\mathcal A$ and all graphs in $\Ecal_{\delta}$, then $1\le\frac{|E(G)|}{|V(G)|} \le \delta$. 
\end{lemma}
\begin{proof}
    Note that $\frac{|E(G)|}{|V(G)|}\ge1$ since $G$ is cyclic. From \cref{dis_pen}, we may assume that $G$ is simplified.
Let $B$ be a block of $G$. As $G$ is simplified, $B$ is biconnected.
Furthermore, $B$ is some  $2$-tree on
$i$-vertices where $i\in\{3,4,5\}$, by \cref{thi_oplod}.  
We call these blocks $i$-blocks, $i\in[3,5]$.\\
Observe that $G$ (i) cannot have more than one $5$-block,
(ii) cannot have more than two $4$-blocks, and (iii) cannot have a $5$-block and a $4$-block. Indeed, if $G$ does not satisfy (i) or (iii) (resp. (ii)), then it contains some graph in $\Acal_8$ (resp. $\Acal_{10}$) as a minor.  This implies that $G \in \Ecal^i_b$, for some $b\in \mathbb N$ and for some $i \in [4]$.

As $G$  excludes all graphs in $\Ecal_{\delta}$, we have that $G \in \Ecal^i_b$ for some $i \in [4]$ and for some $b$ where $b_i \le b \le \beta_i(\delta)-1$ (such $i$ and $b$ exist since $b_1=1$ and $\beta_1(\delta) \ge 2$). Given that $\beta_1(\delta)-1\ge0$, we have 
\[
\frac{|E(G)|}{|V(G)|} \le \frac{3(\beta_i(\delta)-1)+a_i}{2(\beta_i(\delta)-1)+c_i}\le \delta 
\]
The first inequality follows from \cref{obs_stro23}
while the second one follows from the fact that  $\beta_i(\delta)$ is the smallest integer such that every graph in $\Ecal^i_{\beta_i(\delta)}$ has density strictly greater than $\delta$.
\end{proof}

We now proceed with the proof of the main result of this subsection. The proof follows along the same lines as \cref{theo_asdvre}.

\begin{theorem}
    For every limiting density  $\delta \in [1,\nicefrac{3}{2})$, it holds that $\mathsf{cobs}(\mathbb{C}_{\le\delta})=\mathsf{min}\big(\Abbb \cup\{\mathcal D,\mathcal M\}\cup\{\mathcal G^H \; | \; H \in \Ecal_{\delta}\}\big)$.
    \label{theo_vrighjivmn}
\end{theorem}
\begin{proof}
Let $\Gcal$ be a minor-closed graph class. If $\Gcal$ contains some $\Hcal \in \Abbb \cup\{\Dcal,\Mcal\}$ as a subset, then $\delta(\Gcal) \ge \delta(\Hcal)>\delta$ since $\delta(\Hcal) \in \{\frac{3}{2},2\}$, by~\cref{lem_uokdhd}. Now let \(H \in \Ecal_{\delta}\) (so $H \in \Ecal^i_{\beta_i(\delta)}$ for some $i \in [4]$) and assume that \(\mathcal G^H \subseteq\mathcal G\).
For every $n\in\mathbb N$, let $k:=\big\lfloor \frac{n}{(2\beta_i(\delta)+c_i)}\big\rfloor$ and define \(G_n\) to be the disjoint union of
$k\cdot H$ and $n\bmod(2\beta_i(\delta)+c_i)$ isolated vertices. Notice that
$G_n\in\mathcal G^{H}_n\subseteq\Gcal_n$. Moreover,
\[
|E(G_n)|=(3\beta_i(\delta)+a_i)\cdot\big\lfloor n/(2\beta_i(\delta)+c_i)\big\rfloor .
\]
Therefore,
\[
\frac{\operatorname{ex}(\mathcal G,n)}{n}
\;\ge\;
\frac{(3\beta_i(\delta)+a_i)\cdot\big\lfloor n/(2\beta_i(\delta)+c_i)\big\rfloor}{n},
\]
which implies that
\[
\delta(\mathcal G)
= \lim_{n\to\infty}\frac{\operatorname{ex}(\mathcal G,n)}{n}
\;\ge\;
\lim_{n\to\infty}\frac{(3\beta_i(\delta)+a_i)\cdot\big\lfloor n/(2\beta_i(\delta)+c_i)\big\rfloor}{n}
= \frac{3\beta_i(\delta)+a_i}{2\beta_i(\delta)+c_i}
> \delta .
\]
Thus, if $\mathcal G$ contains any of the classes in $\Abbb \cup\{\mathcal D,\mathcal M\}\cup\{\mathcal G^H \; | \; H \in \Ecal_{\delta}\}$ as a subset, then $\mathcal G \not\in \mathbb C_{\le\delta}$.

\medskip

Now let $\Gcal$ be a minor-closed graph class for which there exists $k$ such that every graph  $G\in\Gcal$ minor-excludes $\mathscr{D}_{k}$, $M_{k}$,
and also minor-excludes $\mathscr{P}^{A}_{k}$, for every $A\in \Acal\cup\Ecal_{\delta}$. 
Note that, for every graph $A\in\Acal\cup\Ecal_{\delta}$, every graph $G\in \Gcal$ 
has at most $k-1$ connected components 
containing  $A$ as a minor. 

Fix $G\in \Gcal$.
Let $\Ccal_{1}$ be the connected components of $G$
that contain some graph in $\Acal$ or in $\Ecal_{\delta}$ as a minor and let $\Ccal_{0}$
be those that minor-exclude all graphs in $\Acal$ and in $\Ecal_{\delta}$.
Let $\ell_\delta\coloneqq |\Acal\cup\Ecal_\delta|$.
By the previous observation, we have the following bound:
\begin{eqnarray}
    |\Ccal_{1}|& \leq  & (k-1)\cdot \ell_\delta\label{eq_obodr}
\end{eqnarray}

Let $C\in\Ccal_{1}$.
Observe that for every graph $A\in \Acal\cup\Ecal_{\delta}$ it holds that 
$C$ excludes  $k\cdot A$ as a minor. As $\Acal\cup\Ecal_{\delta}$ contains some planar graph, by \cref{erpos}, there is some $c_{2}$, depending on $\Acal\cup\Ecal_{\delta}$, and some 
set $S_C\subseteq V(C)$ where $|S_C|\leq c_{2}$
such that $J_C\coloneqq C-S_{C}$ minor-excludes all 
graphs in $\Acal\cup\Ecal_{\delta}$.

As we did in \cref{theo_asdvre}, 
We say that a connected component $Z$ of $J_{C}$ is \emph{$C$-poor} if it is a tree and $|N_{S_C}(V(Z))|=1$, i.e., 
there is only one edge between $S_C$ and the vertices of $Z$. We also say that $Z$ is \emph{$C$-rich} 
if it not poor. We denote by $\Qcal_C^{0}$  and $\Qcal_C^{1}$ the set of $C$-poor and $C$-rich connected components of $J_C$, respectively. From \cref{do_econnec_do}, we obtain that $|\Qcal_C^1|\leq  f_{\ref{do_econnec_do}}(c_{2},k)$.

We now partition the edges of $C$ as follows:
Let $E_{C}^{\mathsf{p}}$ be the edges of $C$ with at least one endpoint in a  $C$-poor component, let $E_{C}^{\mathsf{r}}$  be the edges  of the  rich components and let $E_{C}^{\mathsf{s}}$ be the edges  of $C$ 
with one endpoint in $S$ and the other in either $S$ or in some vertex of a $C$-rich component. Notice that $\{E_{C}^{\mathsf{p}},E_{C}^{\mathsf{r}},E_{C}^{\mathsf{s}}\}$ is  partition of $E(C)$.
We also define $m_{C}^{\mathsf{p}}=|E_{C}^{\mathsf{p}}|$, $m_{C}^{\mathsf{r}}=|E_{C}^{\mathsf{r}}|,$ and 
$m_{C}^{\mathsf{s}}=|E_{C}^{\mathsf{s}}|$. Also, we set 
$n^{\mathsf{s}}_{c}=|S_{C}|$ and let $n_{C}^{\mathsf{r}}$ (resp. $n_{C}^{\mathsf{p}}$) be   the number of vertices 
of the $C$-rich (resp. $C$-poor) components of $J_{C}$.

Observe that, given that there are $q$  $C$-poor components,
the edges in $E_{c}^{\mathsf{p}}$ are $n^{\mathsf{p}}_{C}-q$
plus  $q$ extra edges, one from each 
$C$-poor  component to a vertex of $S_{C}$. Therefore,
\begin{eqnarray}
m_{C}^{\mathsf{p}} & = & n_{C}^{\mathsf{p}}.\label{eq_qweqdxq}
\end{eqnarray}
\begin{claim}\label{cl_sverghsdvf}
 $|m_{C}^{\mathsf{s}}|\leq \binom{c_{2}}{2}+f_{\ref{do_econnec_do}}(c_{2},k)\cdot f_{\ref{cosxl}}(c_{2},k)$.
\end{claim}
The proof of this claim is identical to that of Claim~\ref{cl_diagGrid}
in the proof of \cref{theo_asdvre}, so we omit it.

Let $Q_{C}=\bigcup \Qcal_{C}^{1}$, i.e., $Q_{C}$ is the union of all $C$-rich components of $J_{C}$.  
Also, let $$Q=\bigcup \Ccal_{0}\cup\bigcup_{Z\in\Ccal_{1}}Q_Z$$
and observe that every connected  component of $Q$ excludes  every  graph in $\Acal \cup \Ecal_{\delta}$ as a minor. In case $Q$ is cyclic, we construct the graph $M$ from $Q$ as in the proof of  Claim~\ref{cl_asverga} in the proof of \cref{theo_asdvre}.
By \cref{lem_bjibjir} we have that $\frac{|E(M)|}{|V(M)|}\le\delta$. We conclude that $\frac{|E(Q)|}{|V(Q)|}\le\frac{|E(M)|}{|V(M)|}\le\delta$. In case $Q$ is acyclic we again have that $\frac{|E(Q)|}{|V(Q)|}\le1\le\delta$.

\medskip

We now define $n^\mathsf{s}=\sum_{C\in\Ccal_{1}}n_{C}^{\mathsf{s}}$,
$n^\mathsf{r}=\sum_{C\in \Ccal_{0}}|V(C)|+\sum_{C\in\Ccal_{1}}n_{C}^{\mathsf{r}}$, and 
$n^\mathsf{p}=\sum_{C\in\Ccal_{1}}n_{C}^{\mathsf{p}}$. From \eqref{eq_obodr}, it follows that   
\begin{eqnarray}
n^{\mathsf{s}} & \leq  & (k-1)\cdot \ell_\delta\cdot c_{2}. \label{eq_dwqfwec}
\end{eqnarray}

Similarly, we set 
$m^\mathsf{s}=\sum_{C\in\Ccal_{1}}m_{C}^{\mathsf{s}}$,
$m^\mathsf{r}=\sum_{C\in \Ccal_{0}}|E(C)|+\sum_{C\in\Ccal_{1}}m_{C}^{\mathsf{r}}$, and $m^\mathsf{p}=\sum_{C\in\Ccal_{1}}m_{C}^{\mathsf{p}}$.
Let $c\coloneqq k\cdot (\binom{c_{2}}{2}+f_{\ref{do_econnec_do}}(c_{2},k)\cdot f_{\ref{cosxl}}(c_{2},k))$.
Observe that from Claim~\ref{cl_sverghsdvf}, and from the fact that  $\frac{|E(Q)|}{|V(Q)|}\le\delta$ we obtain the following inequalities
\begin{eqnarray}
 m^s & \leq & c.\label{eq_fghkjfg}\\
 m^r&\leq & \delta n^{r}\label{eq_oivxoic}
\end{eqnarray}

\medskip

Fix $G\in\Gcal$ with $n$ vertices and $m$ edges.
Observe that
\begin{eqnarray*}
\frac{m}{n} & =& \frac{m^{\mathsf{s}}+m^{\mathsf{r}}+m^{\mathsf{p}}}{n^{\mathsf{s}}+n^{\mathsf{r}}+n^{\mathsf{p}}}\\
& =^{\eqref{eq_qweqdxq}} & 
\frac{m^{\mathsf{s}}+m^{\mathsf{r}}+n^{\mathsf{p}}}{n^{\mathsf{s}}+n^{\mathsf{r}}+n^{\mathsf{p}}}\\
&\leq ^{\eqref{eq_fghkjfg}}&
\frac{c+m^{\mathsf{r}}+n^{\mathsf{p}}}{n^{\mathsf{r}}+n^{\mathsf{p}}}\\
&=&
\frac{c+m^{\mathsf{r}}+n^{\mathsf{p}}}{n-n^{\mathsf{s}}}\\
&\leq ^\eqref{eq_dwqfwec}&
\frac{c+m^{\mathsf{r}}+n^{\mathsf{p}}}{n-(k-1)\cdot \ell_\delta \cdot c_{2}}\\
&\leq ^{\eqref{eq_oivxoic}}&
\frac{c+\delta n^{\mathsf{r}}+n^{\mathsf{p}}}{n-(k-1)\cdot \ell_\delta \cdot c_{2}}\\
&=&
\frac{c+\delta n^{\mathsf{r}}+n-n^{\mathsf{r}}-n^{\mathsf{s}}}{n-(k-1)\cdot \ell_\delta \cdot c_{2}}\\
&\leq &
\frac{c+(\delta -1)n+n}{n-(k-1)\cdot \ell_\delta \cdot c_{2}}\\
&=&
\frac{c+\delta n}{n-(k-1)\cdot \ell_\delta \cdot c_{2}}
\end{eqnarray*}

Therefore, $\delta(\Gcal)=\lim_{n\to\infty}\frac{\ex(\Gcal,n)}{n}\leq \lim_{n\to\infty}\frac{c+\delta n}{n-(k-1)\cdot \ell_\delta \cdot c_{2}}=\delta$.
\end{proof}

\section{Identifying second-order obstructions}
\label{sec_ident}

In the previous sections we gave complete characterizations of 
 $\Cbbb_{<\delta}$ with $\delta\in\{1,\nicefrac{3}{2}\}$ and $\Cbbb_{\leq \delta}$ with $\delta\in[0,\nicefrac{3}{2})$  in terms of parametric graphs and class obstructions. 
Since we require finite characterizations for algorithmic purposes, we proceed in this section with the identification of the second order obstructions.

\subsection{Second order obstructions}
We define $\Zcal_{1}=\{P_{4},K_{1,3}\}$, $\Zcal_{2}=\{K_3,2\cdot K_2\}$  
and $\Zcal_{3}=\{K_{1,3},K_{3}\}$. We also define 
$\mathcal{Z}_4,\mathcal{Z}_5,\mathcal{Z}_6$ to be the sets of graphs as shown in \cref{fig_obsCDM}.

\begin{figure}[!h]
\centering
\begin{tikzpicture}[scale=.66, every node/.style={scale=0.8}]
\begin{scope}
\draw[thick] (-0.5,2) rectangle (8.5,-2);
\draw[thick] (0,1) -- (0,0) -- (0,-1) (1,1) -- (1,0) -- (1,-1) (2.5,0) -- (3.5,0) -- (4.5,0) -- (5.5,0) (3.5,0) -- (4,-1) -- (4.5,0) (7,0) -- (8,0) -- (8,-1) -- (7,-1) -- cycle; 
\foreach \i/\j in {0/1, 0/0, 0/-1, 1/1, 1/0, 1/-1, 2.5/0, 3.5/0, 4.5/0, 5.5/0, 4/-1, 7/0, 8/0, 7/-1, 8/-1}
\draw[thick] (\i,\j) node[uStyle] {};
\draw[thick] (0.5,-1.5) node {\large $D_1$} (4,-1.5) node {\large $D_2$} (7.5,-1.5) node {\large $D_3$} (4,1.5) node {\large $\Zcal_4=\obs(\Dcal)$};
\end{scope} 
\begin{scope}[xshift=-0.5cm, yshift=-4.5cm]
\draw[thick] (-0.5,2) rectangle (9.5,-4);
\draw[thick] (0,0) -- (1,0) -- (2,0) -- (3,0) -- (4,0) (5.7,0) --++ (36:1) coordinate (a) --++ (-36:1) coordinate (b) --++ (252:1) coordinate (c) --++ (180:1) coordinate (d) --++ (108:1) coordinate (e) (0,-2) -- (0,-3) (1,-2) -- (1,-3) (2,-2) -- (2,-3) (3.5,-2) -- (4.5,-2) -- (4.5,-3) -- (3.5,-3) -- cycle (3.5,-2) -- (4.5,-3) (3.5,-3) -- (4.5,-2) (6,-2) -- (7,-2) -- (6.5,-3) -- cycle (8,-2) -- (8,-3); 
\foreach \i/\j in {0/0, 1/0, 2/0, 3/0, 4/0, 5.7/0, 0/-2, 0/-3, 1/-2, 1/-3, 2/-2, 2/-3, 3.5/-2, 4.5/-2, 4.5/-3, 3.5/-3, 6/-2, 7/-2, 6.5/-3, 8/-2, 8/-3}
\draw[thick] (\i,\j) node[uStyle] {};
\foreach \i in {a, b, c, d, e}
\draw[thick] (\i) node[uStyle] {};
\draw[thick] (2,-0.5) node {\large $M_1$} (8,0) node {\large $M_2$} (1,-3.5) node {\large $M_3$} (4,-3.5) node {\large $M_4$} (7.3,-3.5) node {\large $M_5$} (4.5,1.5) node {\large $\Zcal_5=\obs(\Mcal)$};
\end{scope}
\begin{scope}[xshift=-1.5cm, yshift=-12cm]
\draw[thick] (-0.5,3) rectangle (11.5,-1);
\draw[thick] (0,0) -- (1,0) -- (2,0) -- (3,0) -- (4,0) (2,0) -- (2,1) -- (2,2) (5.5,0) -- (6.5,0) -- (7.5,0) -- (8.5,0) (6.5,0) -- (7,1) -- (7.5,0) (7,1) -- (7,2) (10,0) -- (11,0) -- (11,1) -- (10,1) -- cycle;
\foreach \i/\j in {0/0, 1/0, 2/0, 3/0, 4/0, 2/1, 2/2, 5.5/0, 6.5/0, 7.5/0, 8.5/0, 7/1, 7/2, 10/0, 11/0, 10/1, 11/1}
\draw[thick] (\i,\j) node[uStyle] {};
\draw[thick] (2,-0.5) node {\large $C_1$} (7,-0.5) node {\large $C_2$} (10.5,-0.5) node {\large $C_3$} (5.5,2.5) node {\large $\Zcal_6=\obs(\Ccal)$};
\end{scope}
\end{tikzpicture}
\caption{The first-order obstruction sets for the classes $\Dcal,\Mcal,\Ccal$.}
\label{fig_obsCDM}
\end{figure}

The following are easy to verify. 

\begin{observation}
\label{h7ilq9pavx}
$\obs(\Gcal^{1})=\Zcal_{1}$, $\obs(\Gcal^{2})=\Zcal_{2}$, and $\obs(\Gcal^{3})=\Zcal_{3}$. 
\end{observation}

We now translate the set of class obstructions of \cref{theo_abkvaabo}, that is, the classes in $\Abbb$ and the classes $\Dcal$, $\Ccal$, and $\Mcal$, into the set of second-order obstructions ${\sf OBS}(\Cbbb_{<\nicefrac{3}{2}})$. We start with the obstructions of $\Dcal$, $\Mcal$, and $\Ccal$.

\begin{lemma}\label{lem_vbfsdfg}
    $\mathsf{obs}(\Dcal) = \Zcal_4$.  
\end{lemma}
\begin{proof}
    Let $G$ be a graph that minor-excludes $D_1$, $D_2$, and $D_3$. Suppose $G$ has $r$ components. Since $G$ minor-excludes $D_{1}$, at most one component of $G$ has $\ge 3$ vertices, i.e.,  every other component is $K_1$ or $K_2$. If all components of $G$ are cliques of size at most two, then $G$ is a minor of $\mathscr{D}_r$. So, we may assume $H$ is the unique component of $G$ with $\ge 3$ vertices. It suffices to show that $H$ is a minor of $\mathscr{D}_{k}$ for some $k$ (so $G$ is a minor of $\mathscr{D}_{k+r-1}$).

    Since $G$ minor-excludes $D_3$, each block of $H$ that contains a cycle is, in fact, a triangle block. Moreover, each such block is a leaf block since $D_2$ is minor-excluded. So, all non-leaf blocks are bridges. 
    
    If $H$ contains a triangle block, then only one cut vertex of $H$ is incident to at least three blocks since $D_1$ is minor-excluded. So, $H$ consists of a cut vertex $u$ with $s\ge1$ triangle leaf blocks and $t\ge0$ blocks that are either paths of length 1 or 2 hanging off $u$. Indeed, any attached path of length at least 3 would give a $D_1$ as a minor, which is a contradiction. Now $H$ is a minor of $\mathscr{D}_{r+t}$.

    If instead $H$ is a tree, then $H$ has diameter (i) 4, (ii) 3, or (iii) $\le2$ since $D_1$ is minor-excluded. In the case of (i), $H$ is either $P_5$ or $P_5$ with a pendant path of length at most 2 incident to the middle vertex. In the case of (ii), $H$ is either $P_4$ or $P_4$ with a pendant edge incident to exactly one non-leaf vertex. Both (i) and (ii) follow from the fact that $D_1$ is minor-excluded. In the case of (iii), $H$ is either $K_{1,3}$, $P_3$, $K_2$, or $K_1$. In all cases, $H$ is a minor of $\mathscr{D}_3$.
\end{proof}

\begin{lemma}
    $\mathsf{obs}(\Mcal) = \Zcal_5$.
    \label{lem_vbfsdasdg}
\end{lemma}
\begin{proof}

Let $G$ be any graph that minor-excludes $M_1$, $M_2$, $M_3$, $M_4$, and $M_5$. Suppose $G$ has $r$ components. Since $M_3$ is minor-excluded, at most two components of $G$ are nontrivial (that is, not $K_1$). If all components of $G$ are cliques of size one, then $G$ is a minor of $\mathscr{M}_r$. So, we may assume $H_1$ and $H_2$ are the (possible) unique nontrivial components of $G$. It suffices to show that $H_2\cup H_2$ is a minor of $\mathscr{M}_k$ for some $k$ (so $G$ is a minor of $\mathscr{M}_{k+r-2}$).  
 
If $H_1$ contains a cycle, then its length must be at most 4 since $M_2$ is minor-excluded. Moreover, since $M_5$ is minor-excluded, $H_2$ must be $K_1$ and only one block $B$ of $H_1$ must contain a cycle. Since $B$ is biconnected and both $M_4$ and $M_5$ are minor-excluded, it is easy to verify that $B$ is either $K_3$, $K_4$, or the unique 2-tree on 4 vertices. Furthermore, all blocks must be leaf blocks since $M_1$ is minor-excluded. Now $H_1\cup H_2$ is a minor of $\mathscr{M}_{k+2}$ where $k$ is the number of blocks of $H_1$. The same argument works by symmetry if $H_2$ contains a cycle.

If instead both $H_1$ and $H_2$ are trees, then both of their diameters must be at most 3 since $M_1$ is minor-excluded. We distinguish two cases: (i) $H_i$ has diameter 3 for some $i\in[2]$ and (ii) both $H_1$ and $H_2$ have diameter at most 2. In case (i), $H_{3-i}$ must be $K_1$ since $M_3$ is minor-excluded. In case (ii), both $H_1$ and $H_2$ must be either $K_{1,3}$, $P_3$, $K_2$, or $K_1$. In both cases, $H_1\cup H_2$ is a minor $\mathscr{M}_6$. 
\end{proof}

It is easy to prove the following (see  \cite[Lemma 3.23]{paul2023universal} for a proof).

\begin{proposition} 
    $\mathsf{obs}(\Ccal) = \Zcal_6$. \label{prop_btybbfv} 
\end{proposition}

We next deal with the  obstructions for the classes in $\Abbb$.

\paragraph{Connectivizations.}
Given a minor-closed graph class $\Gcal$, we define 
$\Gcal^{(c)}$ as the class containing every graph whose connected components belong in $\Gcal$. 
Given a graph $G$, we
define $\mathsf{conn}(G)$
as the set containing 
every connected graph $G'$ that contains $G$ as a subgraph and has the minimum possible number of edges. {Observe that $\mathsf{conn}(G)=\{G\}$ if $G$ is connected.} We extend this definition for every set of graphs 
$\Zcal$ by setting $\mathsf{conn}(\Zcal)=\{\mathsf{conn}(Z)\mid Z\in\Zcal\}$.
Bulian and Dawar in \cite{BulianD17fixe} proved the following.

\begin{proposition} 
For every minor-closed class $\Gcal$ it holds that $\obs(\Gcal^{(c)})=\mathsf{min}(\mathsf{conn}(\obs(\Gcal)))$.
 \label{prop_avnuenvue}
\end{proposition}

Given a graph $G$, we define $\mathsf{minors}(G)$ to be the set of all minors of $G$; that is, $\mathsf{minors}(G) = \{G\}\!\downarrow$.

\begin{lemma}

For every two graphs $G_{1}$ and $G_{2}$ where $G_{1}\leq G_2$, it holds that if $\Gcal_{1}=\minors(G_{1})$ 
and $\Gcal_{2}=\minors(G_{2})$, then 
$\obs(\Gcal_{1}^{(c)})\leq_{*}\obs(\Gcal_{2}^{(c)})$.
\label{lem_avhehiecni}
\end{lemma}

\begin{proof}
Notice that $\Gcal_{1}\subseteq \Gcal_{2}$ which, by the definition of $\Gcal^{(c)}$, implies that $\Gcal_{1}^{(c)}\subseteq \Gcal_{2}^{(c)}$. By \eqref{eq_oksui} the lemma follows.
\end{proof}

For every connected graph $G$ on $n$ vertices, we define
\begin{eqnarray}
    \mathcal{Z}^{G} & \coloneqq  & \mathsf{min}(\mathsf{conn}(\mathsf{min}(\mathsf{minors}(K_{n+1})\backslash \mathsf{minors}(G)))).\label{eq_size_b}
\end{eqnarray}

\begin{lemma}

\label{lem_con_con}
    Let $G$ be a connected graph on $n$ vertices and $\Gcal = \mathsf{minors}(G)$. It holds that $\obs(\Gcal^{(c)}) = \mathcal{Z}^{G}$. 
\end{lemma}
\begin{proof}
We claim that $\mathsf{obs(\mathcal{G})} =\mathsf{min}(\mathsf{minors}(K_{n+1})\setminus\mathsf{minors}(G))$. Every minor of 
$G$ has at most $n$ vertices, so any graph with at least $n+2$ vertices cannot be a minimal excluded minor (i.e. cannot be in $\obs(\Gcal)$): deleting one vertex keeps it outside $\mathsf{minors}(G)$. So minimal excluded minors have $\le n+1$ vertices. This bound is sharp since $n+1$ disjoint vertices is always an excluded minor for a graph on $n$ vertices. The graphs on $\le n+1$ vertices are exactly $\mathsf{minors}(K_{n+1})$. Therefore, the excluded minors for $\mathsf{minors}(G)$ are precisely the minimal elements of $\mathsf{minors}(K_{n+1})\setminus\mathsf{minors}(G)$ which proves the claim. Now by \cref{prop_avnuenvue} it follows that $\obs(\Gcal^{(c)}) = \mathcal{Z}^{G}$.
\end{proof}

From \cref{lem_avhehiecni} and \cref{lem_con_con} we have the following observation.

\begin{observation}
\label{obs_smith_trtansfer}
For every two graphs $G_{1},G_{2}$, 
if $G_{1}\leq G_{2}$,  then  $\Zcal^{G_{1}}\leq_{*}\Zcal^{G_{2}}$.
\end{observation}

Let $\Lbbb$ be a finite set of minor anti-chains. 
We define ${\sf MIN}(\Lbbb)$
as the $\leq_{*}$-minimal 
 subset $\Lbbb'$ of $\Lbbb$, i.e., 
for every $\Zcal\in \Lbbb$ there is some $\Zcal'\in\Lbbb'$ such that $\Zcal'\leq_{*}\Zcal$.

\begin{theorem}

\label{th_kifinalop}
The following hold:
\begin{enumerate}[label=\textit{\roman*.}, ref=Theorem~\thetheorem.\roman*]
    \item\label{it_kolwe} For every $\delta\in[0,1)$, ${\sf OBS}(\Cbbb_{\leq \delta})={\sf MIN}\big(\{\Zcal_1,\Zcal_2\}\cup\{\Zcal^{G}\mid G\in\Tcal_{x+2}\}\big)$, where $x=\lfloor\frac{\delta}{1-\delta}\rfloor$ \label{theo_avjkrrv}. 
    \item\label{it_kolwes} For every $\delta\in[1,\nicefrac{3}{2})$, ${\sf OBS}(\Cbbb_{\leq \delta})={\sf MIN}\big(\{\Zcal_4,\Zcal_5\}\cup\{\Zcal^{G}\mid G\in \Acal \cup \Ecal_{\delta}\}\big)$\label{toipkdi0ol}. 
    \item\label{it_olopks} Moreover, ${\sf OBS}(\Cbbb_{<0})=\big\{\{P_2\}\big\}$, ${\sf OBS}(\Cbbb_{<1})=\{\Zcal_{1},\Zcal_{2},\Zcal_{3}\}$ and $\mathsf{OBS}(\Cbbb_{<\nicefrac{3}{2}}) = \{\Zcal_4,\Zcal_5,\Zcal_6\}\cup\{\Zcal^G\mid G \in \Acal\}$\label{lem_asver}.
\end{enumerate}
\end{theorem}
\begin{proof} We prove \emph{i.}, \emph{ii.}, and \emph{iii.}  as follows.\smallskip

    \noindent Proof of \emph{i.} From \cref{hio8io3e}, we have $\cobs(\Cbbb_{\leq \delta})={\sf min}\big(\{\Gcal^{1},\Gcal^{2}\}\cup \{\Gcal^{G}\mid G\in\Tcal_{x+2}\}\big)$ where $x=\lfloor\frac{\delta}{1-\delta}\rfloor$. So, we need to find $\mathsf{obs}(\Gcal)$ for every $\Gcal\in\cobs(\Cbbb_{\leq \delta})$. By \cref{h7ilq9pavx}, we have $\obs(\Gcal^{1})=\Zcal_{1}$, $\obs(\Gcal^{2})=\Zcal_{2}$. Now observe that for every graph $G$ if $\Gcal = \mathsf{minors}(G)$, then $\Gcal^G = \Gcal^{(c)}$ (where $\Gcal^G$ is defined as in \eqref{eq_nwriss}). Therefore, by \cref{lem_con_con}, we have $\obs(\Gcal^G) = \mathcal{Z}^{G}$ for any tree $G\in\Tcal_{x+2}$ where $\Gcal = \mathsf{minors}(G)$. As a result, ${\sf OBS}(\Cbbb_{\leq \delta})={\sf MIN}\big(\big\{\Zcal_1,\Zcal_2\big\}\cup\big\{\Zcal^{G}\mid  G\in\Tcal_{x+2}\big\}\big)$.\\
    
    \noindent Proof of \emph{ii.} From \cref{theo_vrighjivmn}, we have $\mathsf{cobs}(\mathbb{C}_{\le\delta})=\mathsf{min}\big(\Abbb \cup\{\mathcal D,\mathcal M\}\cup\{\mathcal G^H \; | \; H \in \Ecal_{\delta}\}\big)$. So, we need to find $\mathsf{obs}(\Gcal)$ for every $\Gcal\in\cobs(\Cbbb_{\leq \delta})$. By \cref{lem_vbfsdfg} and \cref{lem_vbfsdasdg}, we have $\obs(\mathcal{D})=\Zcal_{4}$, and $\obs(\mathcal{M})=\Zcal_{5}$, respectively. As in the proof of $i.$ above, $\Gcal^G = \Gcal^{(c)}$ for any graph $G$ where $\Gcal = \mathsf{minors}(G)$. So, for any graph $G\in\Acal \cup \Ecal_\delta$ where $\Gcal = \mathsf{minors}(G)$, we have $\obs(\Gcal^G) = \mathcal{Z}^{G}$, by \cref{lem_con_con}. As a result, ${\sf OBS}(\Cbbb_{\leq \delta})={\sf MIN}\big(\{\Zcal_4,\Zcal_5\}\cup\big\{\Zcal^{G}\mid G\in \Acal \cup \Ecal_{\delta}\big\}\big)$.\\ 

    \noindent Proof of \emph{iii.}
    The fact ${\sf OBS}(\Cbbb_{<0})=\big\{\{P_2\}\big\}$  follows directly from \cref{obs_nulnul}.
    From \cref{adsfssbfbscf} and \cref{theo_abkvaabo}, we have $\mathsf{cobs}(\mathbb{C}_{<1})=\{\Gcal^1, \Gcal^2, \Gcal^3\}$ and $\mathsf{cobs}(\mathbb{C}_{<\nicefrac{3}{2}})=\{\Dcal, \Ccal, \Mcal\} \cup \Abbb$, respectively. So, we need to find $\mathsf{obs}(\Gcal)$ for every $\Gcal\in\cobs(\Cbbb_{<1})$ and for every $\Gcal\in\cobs(\Cbbb_{<\nicefrac{3}{2}})$. By \cref{h7ilq9pavx}, we have $\obs(\Gcal^1)=\Zcal_1$, $\obs(\Gcal^2)=\Zcal_2$, and $\obs(\Gcal^3)=\Zcal_3$, respectively. Further, by \cref{lem_vbfsdfg}, \cref{prop_btybbfv}, and \cref{lem_vbfsdasdg}, we have $\obs(\mathcal{D})=\Zcal_{4}$, $\obs(\mathcal{C})=\Zcal_{6}$, and $\obs(\mathcal{M})=\Zcal_{5}$, respectively. Finally, for any graph $G\in\Acal$ where $\Gcal = \mathsf{minors}(G)$, we have $\obs(\Gcal^G) = \mathcal{Z}^{G}$, by \cref{lem_con_con} as above. As a result, ${\sf OBS}(\Cbbb_{<1})=\{\Zcal_1,\Zcal_2,\Zcal_3\}$ and ${\sf OBS}(\Cbbb_{<\nicefrac{3}{2}})=\{\Zcal_4,\Zcal_5, \Zcal_6\}\cup\big\{\Zcal^{G}\mid G\in \Acal\}$. 
\end{proof}
\medskip

\begin{figure}[!h]
\centering
\scalebox{.81}{
\centering
\begin{tikzpicture}
\draw (-1,-1) rectangle (4,-4.9);
\begin{scope}[yshift=-0.5cm]
\draw[thick] (-0.5,-2) -- (-0.5,-1) -- (0.5,-1) -- (0.5,-2); %P_4
\foreach \x/\y in {-0.5/-2, -0.5/-1, 0.5/-1, 0.5/-2}
\draw[thick] (\x,\y) node[uStyle] {};
\end{scope}
\begin{scope}[yshift=-0.5cm, xshift=-0.5cm]
\draw[thick] (2,-2) -- (3,-1) -- (3,-2) (3,-1) -- (4,-2); %K_{1,3}
\foreach \x/\y in {2/-2, 3/-2, 4/-2, 3/-1}
\draw[thick] (\x,\y) node[uStyle] {};
\end{scope}
\draw[thick] (1.5,-4) node {\large$\Zcal_{1}=\obs(\Gcal^1)$};
\draw (-1,-5.05) rectangle (4,-9);
\begin{scope}[yshift=-0.5cm, xshift=0.5cm]
\draw[thick] (0,-5) -- (-0.5,-6) -- (0.5,-6) -- cycle; %K_3
\foreach \x/\y in {0/-5, -0.5/-6, 0.5/-6}
\draw[thick] (\x,\y) node[uStyle] {};
\end{scope}
\begin{scope}[yshift=-0.5cm, xshift=-0.5cm]
\draw[thick] (2.5,-5) -- (2.5,-6) (3.5,-5) -- (3.5,-6); %2\cdot K_2
\foreach \x/\y in {2.5/-5, 2.5/-6, 3.5/-5, 3.5/-6}
\draw[thick] (\x,\y) node[uStyle] {};
\end{scope}
\draw[thick] (1.5,-8) node {\large$\Zcal_{2}=\obs(\Gcal^2)$};
\end{tikzpicture}
\begin{tikzpicture}
\draw (0,-1) rectangle (4,-9);
\begin{scope}[yshift=-0.5cm, xshift=2cm]
\draw[thick] (-0.5,-2) -- (0,-1) -- (0.5,-2) (-1.5,-2) -- (0,-1) -- (1.5,-2); %K_{1,4} for K_3 plus a pendant edge 
\foreach \x/\y in {-0.5/-2, 0/-1, 0.5/-2, -1.5/-2, 1.5/-2} 
\draw[thick] (\x,\y) node[uStyle] {};
\end{scope}
\begin{scope}[yshift=0.5cm, xshift=2cm]
\draw[thick] (-0.5,-5) -- (-0.5,-4) -- (0.5,-4) -- (0.5,-5) -- (0.5,-6); %P_5 for K_3 plus a pendant edge
\foreach \x/\y in {-0.5/-5, -0.5/-4, 0.5/-4, 0.5/-5, 0.5/-6} 
\draw[thick] (\x,\y) node[uStyle] {};
\end{scope}
\begin{scope}[yshift=-2.5cm, xshift=2cm]
\draw[thick] (-0.5,-5) -- (0,-4) -- (0.5,-5) -- cycle; %K_3 for K_3 plus a pendant edge
\foreach \x/\y in {-0.5/-5, 0.5/-5, 0/-4} 
\draw[thick] (\x,\y) node[uStyle] {};
\end{scope}
\draw[thick] (2,-8.5) node {\large$\obs(\Gcal^{K_{1,3}^{s}})$};
\end{tikzpicture}
\begin{tikzpicture}
\draw (0,-1) rectangle (7,-4.9);
\begin{scope}[yshift=-0.5cm, xshift=1cm]
\draw[thick] (-0.5,-2) -- (-0.5,-1) -- (0.5,-1) -- (0.5,-2); %P_4 for K_{1,4} 
\foreach \x/\y in {-0.5/-2, -0.5/-1, 0.5/-1, 0.5/-2}
\draw[thick] (\x,\y) node[uStyle] {};
\end{scope}
\begin{scope}[xshift=1cm, yshift=-2.5cm]
\draw[thick] (-0.5,-2) -- (0,-1) -- (0.5,-2) -- cycle; %K_3 for K_{1,4} 
\foreach \x/\y in {-0.5/-2, 0.5/-2, 0/-1}
\draw[thick] (\x,\y) node[uStyle] {};
\end{scope}
\begin{scope}[yshift=-0.5cm, xshift=0.5cm]
\draw[thick] (2,-2) -- (4,-1) -- (3,-2) (4,-2) -- (4,-1) -- (5,-2) (4,-1) -- (6,-2); %K_{1,5} for K_{1,4}
\foreach \x/\y in {2/-2, 3/-2, 4/-2, 5/-2, 6/-2, 4/-1}
\draw[thick] (\x,\y) node[uStyle] {};
\end{scope}
\draw[thick] (4.5,-4) node {\large$\obs(\Gcal^{K_{1,4}})$};
\draw (0,-5.05) rectangle (7,-9);
\begin{scope}[yshift=-0.5cm, xshift=2cm]
\draw[thick] (-1.5,-6) -- (-1.5,-5) -- (-0.5,-5) -- (0.5,-5) -- (1.5,-5) -- (1.5,-6); %P_6 for P_5
\foreach \x/\y in {-0.5/-5, 0.5/-5, -1.5/-5, 1.5/-5, -1.5/-6, 1.5/-6}
\draw[thick] (\x,\y) node[uStyle] {};
\end{scope}
\begin{scope}[xshift=2cm, yshift=-3.5cm]
\draw[thick] (-0.5,-5) -- (0,-4) -- (0.5,-5) -- cycle; %K_3 for P_5 
\foreach \x/\y in {-0.5/-5, 0.5/-5, 0/-4}
\draw[thick] (\x,\y) node[uStyle] {};
\end{scope}
\begin{scope}[yshift=-0.5cm, xshift=1.5cm]
\draw[thick] (3,-6) -- (4,-5) -- (5,-6) (4,-5) -- (4,-6); %K_{1,3} for P_5
\foreach \x/\y in {3/-6, 4/-5, 5/-6, 4/-6}
\draw[thick] (\x,\y) node[uStyle] {};
\end{scope}
\draw[thick] (4.5,-8) node {\large $\obs(\Gcal^{P_{5}})$};
\end{tikzpicture}
}
\caption{The second-order obstruction set for the class property $\Cbbb_{\leq \nicefrac{3}{4}}$.  We use $K_{1,3}^{s}$ to denote the tree obtained from $K_{1,3}$ if we subdivide one of its edges. In particular, 
${\sf OBS}(\Cbbb_{\leq \nicefrac{3}{4}})=\big\{\obs(\Gcal^{1}),\obs(\Gcal^2),\obs(\Gcal^{K_{1,3}^{s}}),\obs(\Gcal^{K_{1,4}}),\obs(\Gcal^{P_{5}})\big\}$ or, alternatively, ${\sf OBS}(\Cbbb_{\leq \nicefrac{3}{4}})=\big\{\{P_{4},K_{1,3}\}, \{K_3,2\cdot K_2\}, \{K_{1,4},P_{5},K_3\},\{P_{4},K_{1,5},K_3\},\{P_{6},K_{1,3},K_3\}\big\}.$ }
\label{fig_sioskl_pavk}
\end{figure}

We present some applications   of \ref{it_kolwe}. First of all 
it is easy to observe that 
$\cobs(\Cbbb_{\leq 0})$ contains  the class of all graphs with maximum degree 1 and $\Gcal^2$, therefore 
${\sf OBS}(\Cbbb_{\leq 0})=\big\{\{P_{3}\},\{K_3,2\cdot K_2\}\big\}$.
Similarly, $\cobs(\Cbbb_{\leq \nicefrac{1}{2}})$ contains  the class of the graphs whose components are paths of length at most two 
and $\Gcal^2$, therefore  ${\sf OBS}(\Cbbb_{\leq \nicefrac{1}{2}})=\big\{\{P_{4}, K_{1,3}, K_{3}\},\{K_3,2\cdot K_2\}\big\}.$
Notice that all classes defined so far are proper subsets of $\Gcal^{1}$,
therefore the minimization excludes them from the corresponding class obstructions
and also excludes $\Zcal_{1}$ from the corresponding second order obstructions.

The case of $\cobs(\Cbbb_{\leq \nicefrac{2}{3}})$ is slightly more complicated.
It contains, apart from $\Gcal^{1}$ and $\Gcal^{2}$, two more classes, one for each of the two 
trees on $4$ vertices: the first class contains all graphs 
whose components are minors of $P_4$ and the other contains all graphs whose components are minors of $K_{1,3}$. Using \ref{it_kolwe}, it is easy to verify that   
${\sf OBS}(\Cbbb_{\leq \nicefrac{2}{3}})=\big\{\Zcal_{1},\Zcal_{2},{\{P_{5}, K_{1,3}, K_{3}\},\{P_{4}, K_{1,4},K_{3}\}}\big\}$ (recall that  $\Zcal_{1}=\{P_{4},K_{1,3}\}$, $\Zcal_{2}=\{K_3,2\cdot K_2\}$).

Notice now that   $\cobs(\Cbbb_{\leq \nicefrac{3}{4}})$ contains, apart from $\Gcal^{1},\Gcal^{2}$, three more classes, one for each of the three 
trees on $5$ vertices. For each such tree $T$, the corresponding graph class 
is the one containing all graphs whose components are minors of $T$.
Each of these trees  generates an obstruction set according to \eqref{eq_size_b}.
Therefore, according to  \ref{it_kolwe}, there is a total of five obstructions in ${\sf OBS}(\Cbbb_{\leq \nicefrac{3}{4}})$, all of which are depicted in \cref{fig_sioskl_pavk}.\medskip

We next present the simplest possible application of \ref{it_kolwes}, that is, the case of $\Cbbb_{\leq 1}$. The set
${\sf OBS}(\Cbbb_{\leq 1})$  consists of  the minimization  of $\Zcal_4$ and $\Zcal_5$, depicted in \cref{fig_obsCDM}, the obstructions in  
 $\big\{\Zcal^{G}\mid G\in \Ecal_{\delta}\big\}$ and the obstructions in 
  $\big\{\Zcal^{G}\mid G\in \Acal \big\}$.
  Recall that $\mathcal  U_{\delta} = \bigcup_{i \in [4]}\Ecal^i_{\beta_i(\delta)}$ and $\Ecal_{\delta} = \mathsf{min}(\Ucal_{\delta})$. It is easy to verify that $\Ecal_{1}=\Ecal_{0}^{2}\cup \Ecal_{2}^{1}=\{D,B\}$ where $D$ is the \emph{diamond graph}, which is the unique 2-tree on 4 vertices, and $B$ is the \emph{bow tie graph}, which is $\mathscr{C}_2$. The obstructions corresponding to both graphs are generated using \eqref{eq_size_b} and are depicted in \cref{fig_obsCDM_more}.  
Observe now that $D$ and $B$ are minors of graphs in $\Acal$ (and in $\Ecal_{\beta_3(1)}^3$ and $\Ecal_{\beta_4(1)}^4$, which is why they are not part of $\Ecal_1$); therefore, no further obstructions are generated by $\big\{\Zcal^{G}\mid G\in \Acal \big\}$ because of \cref{obs_smith_trtansfer}.

We avoid the precise descriptions of the second order obstructions for limit-densities bigger than $1$ as they are too complex. 
We just mention that 
$|\cobs(\Cbbb_{\leq \nicefrac{6}{5}})|=5$,
$|\cobs(\Cbbb_{\leq \nicefrac{5}{4}})|=8$,
$|\cobs(\Cbbb_{\leq \nicefrac{9}{7}})|=10$.

\begin{figure}[h]
\centering
\scalebox{.81}{
\begin{tikzpicture}
\begin{scope}
\draw (-0.5,1.5) rectangle (9.5,-1.5);
\draw[thick] (0,-1) -- (0,0) -- (0,1) -- (1,1) -- (1,0) (2,-1) -- (2.5,0) -- (3,-1) (2.5,0) -- (2.5,1) -- (3.5,1) (4,-1) -- (5.5,0) -- (5,-1) (6,-1) -- (5.5,0) -- (7,-1) (8,0) -- (9,0) -- (9,-1) -- (8,-1) -- cycle (8,0) -- (9,-1) (8,-1) -- (9,0);
\foreach \x/\y in {0/-1, 0/0, 0/1, 1/1, 1/0, 2/-1, 2.5/0, 3/-1, 2.5/1, 3.5/1, 4/-1, 5.5/0, 5/-1, 6/-1, 7/-1, 8/0, 9/0, 9/-1, 8/-1}
\draw[thick] (\x,\y) node[uStyle] {};
\draw[thick] (6.5,1) node {\large $\obs(\Gcal^D)$};
\end{scope}
\begin{scope}[yshift=-3.2cm]
\draw (-2,1.5) rectangle (11,-3);
\draw[thick] (-0.5,0) -- (0,1) -- (0.5,0) (-1,0) -- (0,1) -- (1,0) (0,1) -- (0,0) (2,0) -- (3,0) -- (3,1) -- (2,1) -- cycle (4,0) -- (4,1) -- (5,1) -- (5,0) (5,1) -- (6,1) -- (6,0) (7,1) -- (8,1) -- (8.5,0) -- (9,1) -- (10,1) (8,1) -- (9,1); 
\draw[thick] (0.5,-1) -- (-0.5,-1) -- (-1.5,-1) -- (-1.5,-2) -- (-0.5,-2) -- (0.5,-2) (2.5,-1) -- (1.5,-1) -- (1.5,-2) -- (2.5,-2) -- (3.5,-1.5) (2.5,-2) -- (3.5,-2.5) (4.5,-1) -- (5.5,-1.5) -- (6.5,-1.5) -- (7.5,-1) (5.5,-1.5) -- (4.5,-2) (6.5,-1.5) -- (7.5,-2) (8.5,-1) -- (8.5,-2) -- (9.5,-2) -- (10.5,-2) (10.5,-1.5) -- (9.5,-2) -- (10.5,-2.5);
\foreach \x/\y in {-0.5/0, 0/1, 0.5/0, -1/0, 1/0, 0/0, 2/0, 3/0, 3/1, 2/1, 4/0, 4/1, 5/1, 5/0, 6/1, 6/0, 7/1, 8/1, 8.5/0, 9/1, 10/1, 0.5/-1, -0.5/-1, -1.5/-1, -1.5/-2, -0.5/-2, 0.5/-2, 2.5/-1, 1.5/-1, 1.5/-2, 2.5/-2, 3.5/-1.5, 3.5/-2.5, 4.5/-1, 5.5/-1.5, 6.5/-1.5, 7.5/-1, 4.5/-2, 7.5/-2, 8.5/-1, 8.5/-2, 9.5/-2, 10.5/-2, 10.5/-1.5, 10.5/-2.5}
\draw[thick] (\x,\y) node[uStyle] {};
\draw[thick] (6,-2.5) node {\large $\obs(\Gcal^B)$};
\end{scope}
\end{tikzpicture}
}
  \caption{The sets of graphs that, along with $\Zcal_4$ and $\Zcal_5$, depicted in {\cref{fig_obsCDM}}, form the second-order obstruction set for the class property $\Cbbb_{\leq 1}$. Here, $D$ (diamond) and $B$ (bow tie) are the unique graphs in $\Ecal^2_0$ and $\Ecal^1_2$, respectively.}
    \label{fig_obsCDM_more}
\end{figure}
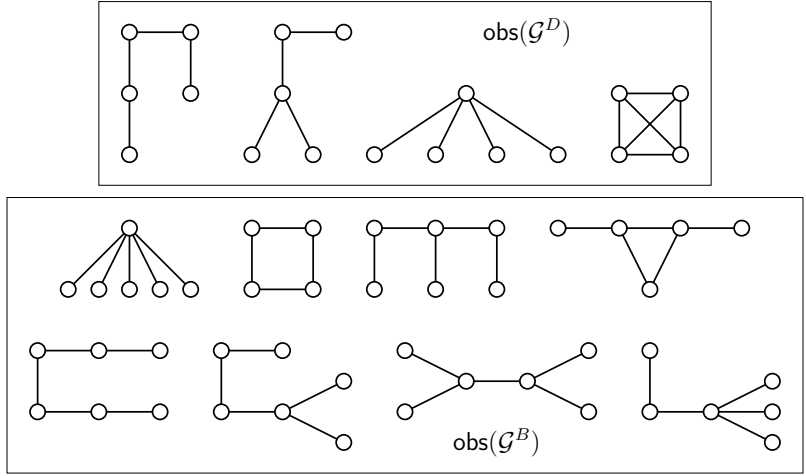

\subsection{Upper bounds on obstruction sizes}

All our bounds on sizes of obstructions are counting unlabeled $K_{4}$-minor-free graphs that are known to be exponentially many.  For this, one may  
use the following asymptotic bound proved by Drmota, Fusy, Kang,  Kraus, and Rué in 
\cite{DrmotaFKKR11Asymptotic}.
 
\begin{proposition}
\label{prop_bountsp}
If $\Gcal=\excl(K_{4})$, then for every $n\in\Nbbb_{\ge1}$, $|\Gcal_{n}|\sim c\cdot n^{-3/2}·\rho^{n}$, where $\rho\approx9.38527$ and $c$  is a constant.
\end{proposition}

For every $\delta\in[0,\nicefrac{3}{2})$, we define $x_{\delta}$ as follows:
\begin{eqnarray}
x_{\delta}  &  = & 
\begin{cases}
\lfloor \frac{\delta}{1-\delta}\rfloor, & \text{when } \delta \in [0,1), \\[6pt]
\lfloor\frac{\delta}{3-2\delta}\rfloor, & \text{when } \delta \in \left[1,\nicefrac{3}{2}\right).
\end{cases}\label{label_delta_x}
\end{eqnarray}

\begin{lemma}
\label{cor_all}
Let $\delta$ be a limiting density in $[0,\nicefrac{3}{2})$. 
 Then the following holds:
    \begin{enumerate}[label=\textit{\roman*.}, ref=Theorem~\thetheorem.\roman*]
        \item For every $\Zcal \in {\sf OBS}(\Cbbb_{\leq \delta})$ and for every $G \in \Zcal$, \[|V(G)| \le 
\begin{cases}
\max\{x_{\delta} + 3,\, 4\}, & \text{when } \delta \in [0,1), \\[6pt]
\max\{2x_{\delta} + 3,7\}, & \text{when } \delta \in \left[1,\nicefrac{3}{2}\right).
\end{cases}
\]\label{theo_sdvsghsvj}
        \item For every $\Zcal \in {\sf OBS}(\Cbbb_{\leq \delta})$, $|\Zcal| = 2^{\Ocal(x_{\delta})}$.
          
        \item $|{\sf OBS}(\Cbbb_{\leq \delta})|=2^{\Ocal(x_{\delta})}$.
    \end{enumerate}    
    Also,  $|{\sf OBS}(\Cbbb_{<0})|=1$,  $|{\sf OBS}(\Cbbb_{<1})|=3$, and  $|{\sf OBS}(\Cbbb_{<\nicefrac{3}{2}})|={33}$. 
\end{lemma}
\begin{proof}  We prove $i.$, $ii.$, and $iii.$ first.\medskip

 \noindent  \textsl{Proof of $i.$}  Let $\delta$ be a limiting density in $[0,1)$. We will compare the number of vertices of the graphs in $\Zcal_1,\Zcal_2$ and in $\Zcal^T$, where $T\in\Tcal_{x_{\delta}+2}$. Observe that every graph in  $\Tcal_{x_{\delta}+2}$ has $x_{\delta}+2$ vertices, so by \eqref{eq_size_b}, we have that every graph $G \in \Zcal^{T}$ where $T \in \Tcal_{x_{\delta}+2}$ has at most $x_{\delta}+3$ vertices. Also, observe that for every $G \in \Zcal_1 \cup \Zcal_2$, we have $|V(G)|\leq 4$. As a result, $|V(G)| \leq \max\{x_{\delta}+3,4\}$ for every $G \in \Zcal$ and $\Zcal \in {\sf OBS}(\Cbbb_{\leq \delta})$, by \ref{toipkdi0ol} (we take the maximum only for the case where $\delta = 0$ so $x_{\delta}+3 \geq  4$ fails). 
  
  Now let $\delta$ be a limiting density in $[1,\nicefrac{3}{2})$. We will now compare the number of vertices of the graphs in $\Zcal_4,\Zcal_5$ and in $\Zcal^G$, where $G\in\Acal\cup\Ecal_\delta$. Observe that $|V(G)|\le6\le\max\{2x_\delta+3,7\}$ for every $G\in\Zcal_4\cup\Zcal_5$, as desired.

  Recall that 
  \[
  \beta_1(\delta) = \lfloor \frac{\delta}{3-2\delta}\rfloor+1 , \beta_2(\delta)=\lfloor\frac{4\delta-5}{3-2\delta}\rfloor+1, \beta_3(\delta)= \max\{\lfloor\frac{5\delta-7}{3-2\delta}\rfloor+1,0\}, \beta_4(\delta) = \max\{\lfloor\frac{7\delta-10}{3-2\delta}\rfloor+1,0\}.
  \]
For every graph $H \in \mathcal{E}_\delta$, we claim that 
\begin{equation}
\label{eq_vbrvbeb}
|V(H)| \le 2 \lfloor \frac{\delta}{3-2\delta}\rfloor+3.
\end{equation}
Since we have that every graph in $\mathcal{E}_\delta$ is a graph from $\mathcal{E}^i_{\beta_i(\delta)}$, where $i \in [4]$, we just need to compare the value $$2 \lfloor \frac{\delta}{3-2\delta}\rfloor+3$$ which is the number of vertices of the graphs in the set $\mathcal{E}^1_{\beta_1(\delta)}$, with the values $$2(\lfloor\frac{4\delta-5}{3-2\delta}\rfloor+1)+4, 2(\max\{\lfloor\frac{5\delta-7}{3-2\delta}\rfloor+1,0\})+5,2(\max\{\lfloor\frac{7\delta-10}{3-2\delta}\rfloor+1,0\})+7,$$ which are the numbers of vertices of the graphs in the sets $\mathcal{E}^2_{\beta_2(\delta)},\mathcal{E}^3_{\beta_3(\delta)},$ and $\mathcal{E}^4_{\beta_4(\delta)}$, respectively. We claim that 
\begin{equation}
\label{eq_verghef}
    2 \lfloor \frac{\delta}{3-2\delta}\rfloor+3 \ge 2(\max\{\lfloor\frac{5\delta-7}{3-2\delta}\rfloor+1,0\})+5.
\end{equation} To prove \eqref{eq_verghef} we proceed as follows. For every limiting density $\delta\in[1,\nicefrac{3}{2})$, we have $\frac{\delta}{3-2\delta}-\frac{5\delta-7}{3-2\delta}
 = \frac{7-4\delta}{3-2\delta} \ge3$, and therefore \begin{eqnarray}\label{eq_efjhsdjgh}
     \lfloor\frac{\delta}{3-2\delta}\rfloor
\ge \lfloor\frac{5\delta-7}{3-2\delta}\rfloor+3.
 \end{eqnarray} 
 
We distinguish the following cases.

If $\lfloor\frac{5\delta-7}{3-2\delta}\rfloor+1\le0$, then $\delta < \frac{7}{5}$. This implies
$\max\{ \lfloor\frac{5\delta-7}{3-2\delta}\rfloor+1,0\}=0$ and the right–hand side of \eqref{eq_verghef} is $5$. 
Since $\delta\in [1,\nicefrac{3}{2})$, it follows that $\frac{\delta}{3-2\delta}\ge1$, therefore 
$2 \lfloor \frac{\delta}{3-2\delta}\rfloor+3 \ge 5$, so the inequality holds.

If $\lfloor\frac{5\delta-7}{3-2\delta}\rfloor+1>0$, then $\delta \ge \frac{7}{5}$. 
This implies
$\max\{\lfloor\frac{5\delta-7}{3-2\delta}\rfloor+1,0\}=\lfloor\frac{5\delta-7}{3-2\delta}\rfloor+1$ and the right–hand side of \eqref{eq_verghef} is $2\lfloor\frac{5\delta-7}{3-2\delta}\rfloor+7$. 
Now \eqref{eq_efjhsdjgh} implies
$2 \lfloor \frac{\delta}{3-2\delta}\rfloor+3\ge2(\lfloor\frac{5\delta-7}{3-2\delta}\rfloor+3)+3 > 2\lfloor\frac{4\delta-5}{3-2\delta}\rfloor+7.$ So, we have the desired inequality and \eqref{eq_verghef} holds.\medskip  

With exactly the same arguments, we show that $2 \lfloor \frac{\delta}{3-2\delta}\rfloor+3 \ge 2(\lfloor\frac{4\delta-5}{3-2\delta}\rfloor+1)+4$ and $2 \lfloor \frac{\delta}{3-2\delta}\rfloor+3 \ge 2(\max\{\lfloor\frac{7\delta-10}{3-2\delta}\rfloor+1,0\})+7$
hold for every limiting density $\delta\in[1,\nicefrac{3}{2})$, thus proving \eqref{eq_vbrvbeb} (for $\delta=1$, it still holds that $2 \lfloor \frac{\delta}{3-2\delta}\rfloor+3 \ge 2(\lfloor\frac{4\delta-5}{3-2\delta}\rfloor+1)+4$, but $2 \lfloor \frac{\delta}{3-2\delta}\rfloor+3 \ge 2(\max\{\lfloor\frac{7\delta-10}{3-2\delta}\rfloor+1,0\})+7$ fails, which is why we take $\max\{2x_\delta+3,7\}$). 
By \eqref{eq_vbrvbeb} and \eqref{eq_size_b}, we have that $|V(H)| \leq  2x_{\delta}+3$ for every graph $H \in \Zcal^G$, where $G \in \Ecal_{\delta}$.\medskip 

Observe that by \eqref{eq_size_b} and by the fact that $\Acal_4,\Acal_6,\Acal_8,$ and $\Acal_{10}$ consist of graphs with at most $4,6,8,$ and $10$ vertices, respectively, it follows that the graphs of $\Zcal^G$ for $G \in \Acal_i$, where $i \in \{4,6,8,10\}$, have at most $5,7,9,$ and $11$ vertices, respectively. 

Moreover, observe that for every $\delta \in [1,\nicefrac{3}{2})$ we have that $2x_{\delta}+3 \geq  5 $. As a result, we need to show what happens in the case $H \in \Zcal^G$, where $G \in \Acal_6\cup\Acal_8 \cup \Acal_{10}$. We observe that $2x_{\delta}+3 \geq  11$ for every $\delta \in [\nicefrac{4}{3},\nicefrac{3}{2})$ (resp. $2x_{\delta}+3 \geq  9$ for every $\delta \in [\nicefrac{9}{7},\nicefrac{3}{2})$ and $2x_{\delta}+3 \geq  7$ for every $\delta \in [\nicefrac{5}{4},\nicefrac{3}{2})$), so it remains to show what happens for the limiting densities $\delta \in \{\nicefrac{6}{5},\nicefrac{5}{4},\nicefrac{9}{7}\}$ (resp. $\delta \in \{\nicefrac{6}{5},\nicefrac{5}{4}$\} and $\delta \in \{ 1,\nicefrac{6}{5}$\}). 

For the densities $\delta \in\{1,\nicefrac{6}{5}\}$ we have that $\mathsf{cobs}(\mathbb{C}_{\le\delta})= \{\mathcal D,\mathcal M\}\cup\{\mathcal G^H \; | \; H \in \Ecal_{\delta}\}$; therefore, ${\sf OBS}(\Cbbb_{\leq \delta}) $ becomes $\{\Zcal_4,\Zcal_5\}\cup\big\{\Zcal^{G}\mid G \in \Ecal_{\delta}\big\}$. This is because the graph classes $\Gcal^H$ where $H \in  \Acal$ always contain some other graph class $\Gcal^H $ where $ H \in \Ecal_{\delta}$. Similarly, for the densities $\delta \in \{\nicefrac{5}{4},\nicefrac{9}{7}\}$, we have that $\mathsf{cobs}(\mathbb{C}_{\le\delta})= \Abbb_4 \cup \{\mathcal D,\mathcal M\}\cup\{\mathcal G^H \; | \; H \in \Ecal_{\delta}\}$; 
therefore, ${\sf OBS}(\Cbbb_{\leq \delta}) $ becomes $\{\Zcal_4,\Zcal_5\}\cup\big\{\Zcal^{G}\mid G \in \Ecal_{\delta}\cup\Acal_4\big\}$.  
By the analysis above and by \ref{theo_avjkrrv}, we have that $|V(G)| \leq 2x_\delta+3$ for every $G \in \Zcal$ and $\Zcal \in {\sf OBS}(\Cbbb_{\leq \delta})$.\\

 \noindent  \textsl{Proof of $ii.$} Let $\delta$ be a limiting density in $[0,1)$.\medskip

\noindent For $\Zcal \in \{\Zcal_1, \Zcal_2\}$ we have that $|\Zcal| = \Ocal(1)$ because they are finite classes and their sizes do not depend on $\delta$.\medskip

\noindent We claim that for every $G \in \Tcal_{x_\delta +2}$ with $\delta\geq \frac{1}{2}$ it holds that $K_{3}\in \Zcal^G$. To see this, observe that every such $G$ contains the unique tree on $3$ vertices as a minor. Furthermore, every proper minor of $K_3$ is a minor of $G$ while $K_3$ is not a minor of $G$. Finally, $\mathsf{conn}(K_3) = \{K_3\}.$ 
Therefore, $K_3 \in \Zcal^G$. For $\delta = 0$ we can easily observe that $\Tcal_{x_\delta +2} = \{K_2\}$  and that $\Zcal^{K_2} = \{\{P_3\}\}$. As a result, in both cases, every graph in $\Zcal^G$ is either $K_3$ or a $K_3$-minor excluding graph. Therefore, all graphs in $\Zcal^G$ are $K_4$-minor-free. 

\medskip

 \noindent Let $\delta$ be a limiting density in $[1,\nicefrac{3}{2})$.

\medskip

\noindent For $\Zcal \in \{\Zcal_4, \Zcal_5\}\cup\{\Zcal^G\mid G \in \Acal\}$ we have that $|\Zcal| = \Ocal(1)$ because they are finite classes and their sizes do not depend on $\delta$.\medskip

\begin{claim}\label{cl_vrgsdf}
 If $G \in \Ecal_\delta$, then every graph in $\Zcal^G$ is either $K_4$ or is $K_4$-minor-free.
\end{claim}

\begin{claimproof}
Since every graph in $\Ecal_\delta$ is some graph from the sets of graphs $\Ecal^i_{\beta_i(\delta)}$ for
$i \in [4]$, we prove the claim for each $\Ecal^i_{\beta_i(\delta)}$.
For this, we distinguish two cases.

\noindent \textsl{Case 1.} Let $G\in  \Ecal^i_{\beta_i(\delta)}$, where $i \in \{2,3,4\}$.
Observe that $G$ contains the diamond graph as a minor.  Moreover, every proper minor of $K_4$ is a minor of $G$ while $K_4$ is not a minor of $G$. Finally, $\mathsf{conn}(K_4) = \{K_4\}$. As a result, $K_4 \in \Zcal^G$, which implies every other graph in $\Zcal^G$ is $K_4$-minor-free.\\

\noindent \textsl{Case 2.} Let $G\in\Ecal^1_{\beta_1(\delta)}$ with $n$ vertices.
Since $G$ is a tree of triangles, it is $K_4$-minor-free, and hence every minor of $G$ is also $K_4$-minor-free. We define
\[
M = {\sf min}(\mathsf{minors}(K_{n+1}) \setminus \mathsf{minors}(G)).
\]
Recall that $\Zcal^G={\sf min}({\sf conn}(M))$, by \eqref{eq_size_b}. Let $H \in M$. We first show that every $H\in M$ is $K_4$-minor-free. 

Suppose, for a contradiction, that $H$ has $K_4$ as a minor. Let $H'$ be a minor of $H$ that is isomorphic to $K_4$. 
If $H'$ is a proper minor of $H$, then $H' \notin \mathrm{minors}(G)$ (since $G$ is $K_4$-minor-free), which contradicts the fact that every proper minor of $H$ lies in $\mathrm{minors}(G)$. If $H = K_4$, then $C_4$ is a proper minor of $H$ and $C_4$ is not a minor of $G$ (every cycle in a minor of $G$ has length 3), which is again a contradiction. Therefore, every $H \in M$ is $K_4$-minor-free. 

Now let $Z \in \mathsf{conn}(H)$. It sufficed to show that $Z$ is $K_4$-minor-free. Observe that $Z$ is obtained from $H$ by adding exactly $\mathsf{cc}(H)-1$ edges between different components, and these new edges are all bridges, by minimality. Moreover, any $3$-connected minor of a graph must be contained entirely in one of its blocks. Since $K_4$ is $3$-connected, in order for $K_4$ to be a minor of $Z$, it should be contained entirely in one of its blocks or, equivalently, in one of the connected components of $H$. This is not possible since $H$ is $K_4$-minor-free. Thus, every graph in $\Zcal^G$, where $G \in \Ecal^1_{\beta_1(\delta)}$, is $K_4$-minor-free.
\end{claimproof}
Given the above claim, the result follows from \cref{prop_bountsp} 
\medskip

\noindent  \textsl{Proof of $iii.$} 
We have that ${\sf OBS}(\Cbbb_{\leq \delta})={\sf MIN}\big(\{\Zcal_1,\Zcal_2\}\cup\{\Zcal^{G}\mid G\in\Tcal_{x+2}\}\big)$ for limiting densities $\delta \in [0,1)$ and ${\sf OBS}(\Cbbb_{\leq \delta})={\sf MIN}\big(\{\Zcal_4,\Zcal_5\}\cup\{\Zcal^{G}\mid G\in \Acal \cup \Ecal_{\delta}\}\big)$ for limiting densities $\delta \in [1,\nicefrac{3}{2})$. So, we just need to bound the number of graphs in $\Tcal_{x_\delta+2}$ and in $\Ecal_\delta$, respectively. We know that every graph $G \in \Tcal_{x_\delta+2}$ is $K_4$-minor-free since it is a tree and $|V(G)|\leq x_{\delta}+2$.  We also know that every graph $G \in \Ecal_\delta$ is $K_4$-minor-free, and $|V(G)|\leq 2x_{\delta}+3$ from part $i$. So, the result follows, by \cref{prop_bountsp}.
\medskip

\noindent By
\ref{it_olopks}, we have that $|{\sf OBS}(\Cbbb_{<0})|=1$, $|{\sf OBS}(\Cbbb_{<1})|=3$ and  $|{\sf OBS}(\Cbbb_{<\nicefrac{3}{2}})|=33$.
\end{proof}

 \section{Algorithmic consequences}
\label{alg_ons}

We now proceed with the algorithmic consequences of our results.

\subsection{Preliminary algorithmic results}
We need the following  algorithmic result that  follows from  the main result of
\cite{AdlerDFST11Faster}. 

\begin{proposition}
\label{prop_asdfvert}
There is an algorithm 
that, given a graph  $Z$  and a $y$-vertex graph $Y$, checks whether $Y$ is a minor of $Z$ in time $2^{\poly(k)}\cdot y^{\Ocal(k)}\cdot 2^{\Ocal(y^2)}\cdot |E(Z)|$ where $k=\tw(Z)$. 
\end{proposition}

The above proposition relies on the notion of \emph{treewidth}, denoted $\tw$, a graph parameter that measures how closely the structure of a graph resembles that of a tree.
We omit the formal definition, as it does not provide additional insight for the purposes of the proof.

\begin{lemma}\label{lem_vegdgdg}
There is an algorithm that, 
given a finite set $\Ycal$ of graphs containing at least one planar graph
and given a graph $Z$, outputs 
whether at least one of the graphs in $\Ycal$ is a minor of $Z$ in $2^{\poly(y)} \cdot |V(Z)|$ 
time, where all graphs in $\Ycal$ have size at most $y$.
\end{lemma}

\begin{proof}
Let $Y$ be the planar graph in $\Ycal$ where we know that $|V(Y)|\leq y$.
According to \cite{chuzhoy2021towards}, there is a polynomial function $g: \Nbbb \to \Nbbb$ 
such that if  $\tw(Z)\geq g(y)$, then $Z$ contains $Y$ as a minor.
Furthermore, \cite{Bodlaender96ALinear}
  provides an algorithm such that, 
given a graph $Z$ and an integer $k$, 
either reports that  $\tw(Z)\leq k$ or that $\tw(Z)>k$ running in $2^{\poly(k)}\cdot |V(Z)|$.
Combining these two results, we can either 
report that $Y$
is a minor of $Z$ or that $\tw(Z)\leq g(y)$
in 
$2^{\poly(y)} \cdot |V(Z)|$
time.
In the first case,  our  algorithm gives a positive answer. If not,  
we use the algorithm of \cref{prop_asdfvert}
and check, for every $Y\in \Ycal$
whether $Y\leq Z$ in 
$2^{\poly(y)} \cdot |V(Z)|$
time.
\end{proof}

Recall that $n(\Zcal) = \sum_{Z \in \Zcal}|V(Z)|$ for every finite set of graphs $\Zcal$.
The next corollary follows directly from \cref{lem_vegdgdg}. 

\begin{corollary}
    \label{cor_asdghf}
    One can construct an algorithm that with input 
    \begin{itemize}
        \item a set of sets of graphs  $\mathbb{Y}=\{\Ycal_1,\ldots,\Ycal_r\}$ where for every $i \in [r]$, \begin{itemize}
            \item $|\mathcal{Y}_i| \leq  q$ and
            \item $\Ycal_i$ contains a planar graph   and every graph in $\Ycal_i$ has at most $y$ vertices, and
        \end{itemize} 
        \item a set of graphs $\Zcal$, 
    \end{itemize}
    outputs whether there is an $\Ycal \in \mathbb{Y}$ such that $\Ycal \le_{*} \Zcal$ in time $\Ocal(r \cdot q \cdot 2^{\poly(y)}\cdot n)$, where $n = n(\Zcal)$.
\end{corollary}

We also need the following restatement of \eqref{pposg99fdhye}. 
 
\begin{corollary}
    \label{klk89sghds3f2w}
    Let $\Gcal$ be a minor-closed graph class and let $\Cbbb$ be  some class property. Now
    $\Gcal\in \Cbbb$ iff for every $\Qcal\in{\sf OBS}(\Cbbb)$ it holds that $\Qcal\nleq_{*} \obs(\Gcal)$. 
\end{corollary}

\subsection{Algorithms for  deciding limit-densities}
Notice now that every graph in the obstructions mentioned {in the proof of} \cref{cor_all} is either $K_4$ or is  $K_4$-

minor-free. This allows the derivation of linear algorithms using \cref{cor_asdghf}.
By \cref{cor_asdghf}, \cref{klk89sghds3f2w}, \ref{lem_asver}, and the fact that $|{\sf OBS}(\Cbbb_{<1})|=3$ and  $|{\sf OBS}(\Cbbb_{<\nicefrac{3}{2}})|=33$ (\cref{cor_all}), the following hold.

\begin{theorem}

\label{th_fitrop}
One can construct an algorithm that, given an anti-chain $\Zcal$, decides whether
$\delta(\excl(\Zcal))< 1$ in $\Ocal(n)$ time, where $n = n(\Zcal)$. 
\end{theorem}

\begin{theorem}
\label{th_secsec}
One can construct an algorithm that, given an anti-chain $\Zcal$, decides whether
$\delta(\excl(\Zcal))< \frac{3}{2}$ in $\Ocal(n)$ time, where $n =  n(\Zcal)$.
\end{theorem}

By \ref{theo_avjkrrv}, \ref{toipkdi0ol}, \cref{cor_all}, and \cref{klk89sghds3f2w},  we can apply \cref{cor_asdghf} for $r = 2^{\Ocal(x_{\delta})},$  $q = 2^{\Ocal(x_{\delta})}$ and $y = \Ocal(x_{\delta})$ and derive the following.

\begin{theorem}
\label{th_oththe}
 One can construct an algorithm that, given some $\delta\in [0,\nicefrac{3}{2})$, and some minor anti-chain $\Zcal$, decides whether
$\delta(\excl(\Zcal))\leq \delta$ in  $2^{\poly(x_{\delta})}\cdot n$
time, where $x_{\delta}$ is defined as in \eqref{label_delta_x} and $n:= n(\Zcal)$.
\end{theorem}

\medskip

 We stress that the result
 of \cref{prop_asdfvert} is  constructive in the sense that 
if $H$ is a minor of $G$, then it the algorithm also certifies it by providing a  minor model of $H$ in $G$. As a consequence of this, \cref{cor_asdghf} is constructive in the same 
way: in case of a positive answer it finds for every $Z\in\Zcal$ 
a minor model of a graph  $Y\in \Ycal$  in $Z$,
and in case of a negative answer it finds some $Z\in\Zcal$ which minor-excludes all graphs in $\Ycal$ for every $\Ycal\in\Ybbb$.

For \cref{th_oththe}, this implies that if the algorithm gives a negative answer, then this is certified by some  $\Ycal\in{\sf OBS}(\Cbbb_{\leq \delta})$ where $\Ycal\leq_{*} \obs(\Zcal)$. 
This, in turn, is certified by producing, for every graph $Z\in \obs(\Zcal)$, 
a minor model of some graph of $\Ycal$.
These minor models may serve as a proof that  $\delta(\excl(\Zcal))>\delta$.
If, instead, the algorithm gives a positive answer, then the absence of some  $\Ycal\in{\sf OBS}(\Cbbb_{\leq \delta})$ where $\Ycal\leq_{*} \obs(\Zcal)$ is the proof that 
$\delta(\excl(\Zcal))\leq \delta$. Analogous remarks can be made for the algorithms of \cref{th_fitrop} and \cref{th_secsec}.
\medskip

\subsection{Computing limit-densities} Another relevant question is whether it is possible to compute the \textsl{exact} value of $\delta(\excl(\Zcal)),$ given a minor anti-chain $\Zcal$.
Although it is unclear whether this problem is decidable in general, it is certainly decidable when restricted to the interval $[0,\nicefrac{3}{2})$, even if there are infinitely many values of $\delta(\excl(\Zcal))$ in this interval.

\begin{theorem}
\label{op_de_ee}
 One can construct an algorithm that, given an anti-chain $\Zcal$, either reports that $\delta(\excl(\Zcal))\geq  \frac{3}{2}$ or outputs the value of $\delta(\excl(\Zcal))$ in $2^{\poly(n)}$ time, where $n\coloneqq n(\Zcal)$.
\end{theorem}
\begin{proof}
From \cref{obs_nulnul}, we may directly assume that every graph in $\Zcal$ contains at least one edge (otherwise $\excl(\Zcal)$ is finite and the algorithm outputs $-\infty$).

Our algorithm checks first whether for some $\Ycal\in \{\Zcal_4,\Zcal_5,\Zcal_6\}\cup\{\Zcal^G\mid G \in \Acal\}$
it holds that $\Ycal\leq _{*}\Zcal$. This can be done in $\Ocal(n)$ time, where $n\coloneqq n(\Zcal)$, by \cref{cor_asdghf}. If the answer is positive, we can safely report that $\delta(\excl(\Zcal))\geq \frac{3}{2}$, by \cref{klk89sghds3f2w} and by \ref{lem_asver}. 
If the answer is negative, then we check whether for some $\Ycal\in \{\Zcal_1,\Zcal_2,\Zcal_3\}$
it holds that $\Ycal\leq _{*}\Zcal$. This can also be done in $\Ocal(n)$ time, by \cref{cor_asdghf}. We distinguish two cases, depending on the answer that we receive.\medskip

\noindent \textsl{Case 1}. If the answer is positive then we know that $1\leq \delta(\excl(\Zcal))<\frac{3}{2}$, by \cref{klk89sghds3f2w} and by \ref{lem_asver}. The fact that $\Ycal \not \leq_{*}\Zcal$ for every $\Ycal\in \{\Zcal_4,\Zcal_5,\Zcal_6\}\cup\{\Zcal^G\mid G \in \Acal\}$ along with \eqref{eq_oksui} implies that 
\begin{eqnarray}
    \label{eq_sdbbhhas}
    \text{for every } \Ycal\in \{\Zcal_4,\Zcal_5,\Zcal_6\}\cup\{\Zcal^G\mid G \in \Acal\} \text{, it holds that }\excl(\Ycal) \not \subseteq \excl(\Zcal).
\end{eqnarray}
When  $\Ycal\in \{\Zcal^G\mid G \in \Acal\}$,  \eqref{eq_sdbbhhas} implies that
$\Gcal \not \subseteq \excl(\Zcal)$ for every $\Gcal \in \Abbb$.
When  $\Ycal\in \{\Zcal_4,\Zcal_5,\Zcal_6\}$, 
\eqref{eq_sdbbhhas}
implies that there are three graphs $Z_4,Z_5,Z_6\in\Zcal$ such that $Z_4 \in \Dcal$, $Z_5 \in \Mcal$, and $Z_6 \in \Ccal$.
Observe now that there is some $k,$ depending on $Z_4,Z_5,Z_6,$ such that $\excl(\Zcal)$ excludes $\mathscr{D}_k$, $\mathscr{M}_{k}$, and $\mathscr{C}_k$ as minors. In particular, $k$ is the minimum integer such that $Z_4$ is a minor of $\mathscr{D}_k$, $Z_5$ is a minor of $\mathscr{M}_k$, and $Z_6$ is a minor of $\mathscr{C}_k$. Notice that $k$ is at most $\max\{|V(Z_4)|,|V(Z_5)|,|V(Z_6)|\} \leq  n$.
 
By \cref{theo_asdvre}, we know that $$\delta(\excl(\Zcal))\leq \frac{3}{2}-\frac{1}{4k^2 + 14}.$$ Notice that in order to detect each $Z_i$ for $i \in \{4,5,6\}$ we go through all graphs of $\Zcal$ and we find one that minor-excludes all graphs in $\Zcal_i$. Therefore, $Z_4,Z_5,Z_6$ can be computed in $\Ocal(n)$ time, given that we just need to check if they are minors of some $\mathscr{D}_k, \mathscr{C}_k,\mathscr{M}_k$ for some $k$, respectively.

Now we apply the algorithm of \cref{th_oththe} on $\Zcal$
for every $\delta= \frac{3b+a_i}{\,2b+c_i\,}$ with $i\in[4]$
and $b \in \Nbbb$ (or $b \in \Nbbb_{\geq1}$ in case $i= 1$) where $\delta \le \frac{3}{2} - \frac{1}{4k^2-14}$, likewise we can detect $\delta(\excl(\Zcal))$ among the values  

\begin{eqnarray}
\label{eq_vrehsdv}
\Delta\coloneqq \Big\{\frac{3b+a_i}{\,2b+c_i\,} \; \bigg| \; i\in[4], b \in \Nbbb \text{ (or } b\in \Nbbb_{\geq 1} \text{ in case } i=1), \text{ and}\; \frac{3b+a_i}{\,2b+c_i\,}\leq \frac{3}{2}-\frac{1}{4k^2 + 14} \Big\}.
\end{eqnarray}
Notice that $\Delta$ is a subset 
of $\Lbbb\cap [1,\nicefrac{3}{2})$. This is because if $i=1$ (so $(a_1,c_1)=(0,1)$), set $(x,y)=(b,0)$ then $\frac{3b}{2b+1}=\frac{3x+2y}{2x+y+1}$, if $i=2$ (so $(a_2,c_2)=(5,4)$), set $(x,y)=(b+1,1)$ then $\frac{3b+5}{2b+4}=\frac{3(b+1)+2}{2(b+1)+2}=\frac{3x+2y}{2x+y+1}$, if $i=3$ (so $(a_3,c_3)=(7,5)$), set $(x,y)=(b+1,2)$ then $\frac{3b+7}{2b+5}=\frac{3(b+1)+4}{2(b+1)+3}=\frac{3x+2y}{2x+y+1}$, and if $i=4$ (so $(a_4,c_4)=(10,7)$), set $(x,y)=(b+2,2)$ then $\frac{3b+10}{2b+7}=\frac{3(b+2)+4}{2(b+2)+3}=\frac{3x+2y}{2x+y+1}$ (see the definition of $\Lbbb$ in \eqref{eq_p_kil}).  

Observe that the set $\Delta$, defined in \eqref{eq_vrehsdv}, has at most $7k^2+19$ values. Given that $k$ is at most $n$, we can determine $\delta(\excl(\Zcal))$ by applying the algorithm of \cref{th_oththe} a number of 
 $7k^2+19 = \Ocal(n^2)$ times, which can be done in $\Ocal(2^{\poly(\max\{x_{\delta} \; \mid \; \delta \in \Delta\})}\cdot n^3)$. Since $\delta \leq  \frac{3}{2}-\frac{1}{4n^2 + 14}$ and $x_\delta$ is increasing, we have that $\max\{x_{\delta} \; \mid \; \delta \in \Delta\}\leq  3n^2+10$. Thus, we can determine $\delta(\excl(\Zcal))$ in $\Ocal(2^{\poly(n)}\cdot n^3)$ time. \medskip

\noindent\textsl{Case 2:} If the answer is negative,  then this means that for every $\Ycal\in \{\Zcal_1,\Zcal_2,\Zcal_3\}$ we have $\Ycal \not \leq_{*}\Zcal$. This along with \eqref{eq_oksui} implies that
\begin{eqnarray}
    \text{for every } \Ycal\in \{\Zcal_1,\Zcal_2,\Zcal_3\} \text{, it holds that } \excl(\Ycal) \not \subseteq \excl(\Zcal).\label{aghuaha}
\end{eqnarray}
When  $\Ycal = \Zcal_1$, \eqref{aghuaha} implies that
$\Gcal^1 \not \subseteq \excl(\Zcal)$.
When  $\Ycal\in \{\Zcal_2,\Zcal_3\}$, 
\eqref{aghuaha} implies there are two graphs $Z_2,Z_3\in\Zcal$ such that, $Z_2 \in \Gcal^2$ and $Z_3 \in \Gcal^3$. Notice that in order to detect each $Z_i$ for $i \in \{2,3\}$ we go through all graphs of $\Zcal$ and we find one that minor-excludes all graphs in $\Zcal_i$. Therefore, $Z_2,Z_3$ can be computed in $\Ocal(|\Zcal| \cdot n)$ time, by \cref{lem_vegdgdg}. This, in turn, implies that there is some $k$ depending on $Z_2,Z_3$ such that $\excl(\Zcal)$ excludes $K_{1,k}$ and $P_{k}$ as minors. In particular, $k$ is the minimum integer such that $Z_2$ is a minor of $K_{1,k}$ and $Z_3$ is a minor of $P_k$.  Notice as before that $k$ is at most $\max\{|V(Z_2)|,|V(Z_3)|\} \leq  n$. By \cref{l664gdjd}, we know that $\delta(\excl(\Zcal))\leq 1-\frac{1}{k^2}$. Now we apply the algorithm of \cref{th_oththe} on $\Zcal$
for every $\delta= \frac{i}{i+1}$, where $i \in \Nbbb$
and $\delta \leq  1-\frac{1}{k^2}$. Equivalently, we can determine $\delta(\excl(\Zcal))$ among the 
values $
\Delta\coloneqq \Big\{\frac{i}{i+1} \; \mid \; i \in \Nbbb \text{ and}\; \frac{i}{i+1}\leq 1-\frac{1}{k^2} \Big\}$. 
Observe that $|\Delta|\leq k^2$ {and $\max\{x_\delta\; \mid \; \delta\in\Delta\}\le n^2-1$}. Thus, similar to Case 1, we can determine $\delta(\excl(\Zcal))$ in $\Ocal(2^{\poly(n)}\cdot n^3)$ time.
\end{proof}

\section{Research directions}
This paper raises the broader question of obtaining precise characterizations 
of all minor-closed graph classes with a given limiting density. To this end, we 
exploit the expressive power of parametric graphs, which we determine for 
every limiting density in the interval $[0,\nicefrac{3}{2})$.

Determining which (rational, according to \cite{KapadiaN20Densities}) numbers 
occur as limiting densities—i.e., identifying the set $\Lbbb$—is a challenging  
open problem initiated by Eppstein \cite{Eppstein10Densities}, who determined 
$\Lbbb\cap [0,\nicefrac{3}{2})$. Our results extend Eppstein’s work in a 
qualitative way: not only do we know the possible values of the limiting densities in 
this interval, but we can also describe \textsl{precisely} which minor-closed classes 
achieve each of them.

Extending this understanding to limiting densities $\delta\ge\frac{3}{2}$ is 
substantially more challenging. Some progress is still possible for 
$\delta \in [\nicefrac{3}{2},2)$, where the set $\Lbbb\cap[\nicefrac{3}{2},2)$ 
has been completely characterized by McDiarmid and Przykucki \cite{McDiarmidP19On}. 
Moreover, as observed in \cite{Eppstein10Densities}, the accumulation points of 
$\Lbbb$ within this interval  also form an infinite set; in particular,
$\Bbbb\cap[\nicefrac{3}{2},2)=\{x+1\mid x\in\Lbbb\cap[\nicefrac{1}{2},1)\}$. 
A natural next step is therefore to combine our approach with the results of 
\cite{McDiarmidP19On} in order to obtain a full characterization of the 
minor-closed graph classes in $\Cbbb_{\leq \delta}$ and $\Cbbb_{<\delta}$ for every 
$\delta\in[\nicefrac{3}{2},2)$. We believe that the techniques introduced in this 
paper may contribute meaningfully to this direction.

When $\delta\ge 2$, the situation becomes even less clear, as the set of limiting 
densities beyond this point remains  unknown (notice that $2$ is the first point that is an accumulation point of accumulation points).
Motivated by our 
positive results below $\nicefrac{3}{2}$, we conjecture that, for every rational 
$\delta$, the classes $\Cbbb_{\leq \delta}$ and $\Cbbb_{<\delta}$ admit a finite obstruction

characterization.

\begin{conjecture}
\label{conj_all}
For every $\delta\in\Qbbb$, both $\cobs(\Cbbb_{\leq δ})$ and $\cobs(\Cbbb_{<δ})$ are finite.
\end{conjecture}

This conjecture can be viewed as a special instance of the $\omega^2$-wqo 
conjecture. Given that the well-quasi-ordering of graphs under the minor relation 
is not expected to admit a constructive proof 
(see \cite{FriedmanRS87them,krombholz2019upper}), a constructive resolution of the 
$\omega^2$-wqo conjecture appears similarly unlikely. Nevertheless, this does not 
preclude the existence of a constructive proof of \cref{conj_all}. Motivated by 
this possibility, we formulate the following refinement of 
\cref{conj_all}.

\begin{conjecture}
\label{conj_all1}
It is possible to construct a function   $f:\Qbbb\to\Nbbb$ such that, for every 
$\delta\in\Qbbb$, if $\Gcal\in\cobs(\Cbbb_{\leq δ})$ (or $\Gcal\in\cobs(\Cbbb_{<δ})$), 
then every graph in $\obs(\Gcal)$ has at most $f(\delta)$ vertices.
\end{conjecture}

Our results show that \cref{conj_all1} holds for all
$\delta \in [0, \nicefrac{3}{2})$, and the corresponding bounds on the function
$f$ are provided in \cref{cor_all}. It is an interesting question whether  further progress can be made
on ``larger fragments'' of \cref{conj_all} for densities exceeding $\nicefrac{3}{2}$, based on constructive proofs and, preferably, on explicit bounds.
We believe that the analogous question on \cref{conj_all1}
is much harder, especially when $\delta>3$.

\end{document}